\documentclass[12pt,twoside]{amsart}
\usepackage[margin=3cm]{geometry}
\usepackage[colorlinks=false]{hyperref}
\usepackage[english]{babel}
\usepackage{graphicx,titling}
\usepackage{float}
\usepackage{amsmath,amsfonts,amssymb,amsthm}
\usepackage{lipsum}
\usepackage[T1]{fontenc}
\usepackage{fourier}
\usepackage{color}
\usepackage[latin1]{inputenc}
\usepackage{esint}
\usepackage{caption}
\usepackage{ccicons}

\makeatletter
\def\blfootnote{\xdef\@thefnmark{}\@footnotetext}
\makeatother

\newcommand\ccnote{
    \blfootnote{\copyright\,\, Michael Struwe}
    \blfootnote{\ccLogo\, \ccAttribution\,\, Licensed under a \href{https://creativecommons.org/licenses/by/4.0/}{Creative Commons Attribution License (CC-BY)}.}
}

\usepackage[export]{adjustbox}
\numberwithin{equation}{section}
\usepackage{setspace}
\renewcommand{\le}{\leqslant}

\renewcommand{\ge}{\geqslant}

\renewcommand{\mathbb}{\varmathbb}
\usepackage{fancyhdr}
\newtheorem{theorem}{Theorem}[section]
\newtheorem{lemma}[theorem]{Lemma}
\newtheorem{corollary}[theorem]{Corollary}
\newtheorem{proposition}[theorem]{Proposition}

\newcommand{\C}{\mathbb C}

\newcommand{\R}{\mathbb R}

\newcommand{\N}{\mathbb N}

\def\Xint#1{\mathchoice
   {\XXint\displaystyle\textstyle{#1}}%
   {\XXint\textstyle\scriptstyle{#1}}%
   {\XXint\scriptstyle\scriptscriptstyle{#1}}%
   {\XXint\scriptscriptstyle\scriptscriptstyle{#1}}%
  \!\int}
\def\XXint#1#2#3{{\setbox0=\hbox{$#1{#2#3}{\int}$}
     \vcenter{\hbox{$#2#3$}}\kern-.5\wd0}}

\def\dashint{\Xint-}

\address{Michael Struwe, Departement Mathematik, ETH-Z\"urich, CH-8092 Z\"urich}
\email{michael.struwe@math.ethz.ch}

\begin{document}

\thispagestyle{empty}

\begin{minipage}{0.28\textwidth}
\begin{figure}[H]
\includegraphics[width=2.5cm,height=2.5cm,left]{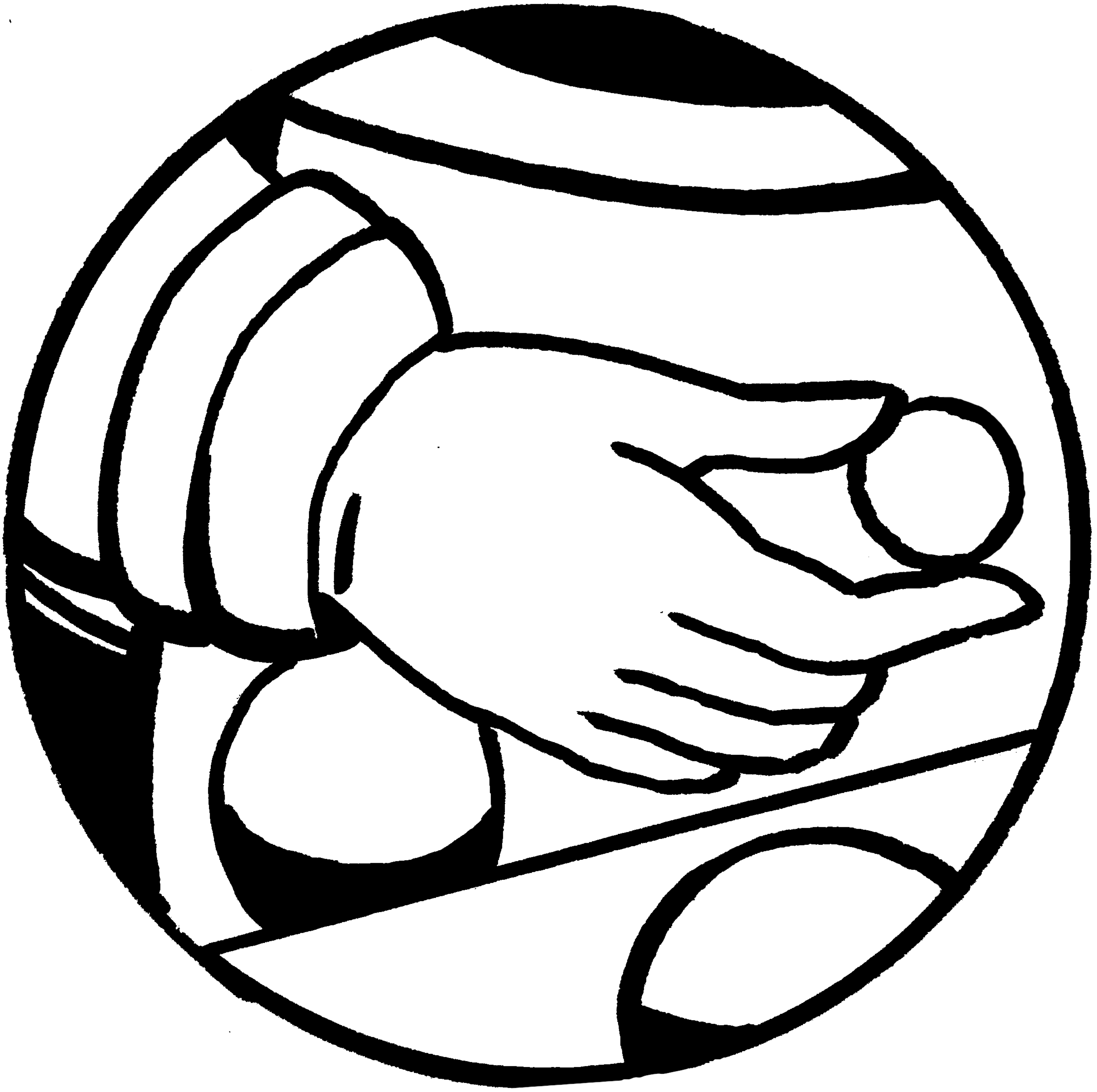}
\end{figure}
\end{minipage}
\begin{minipage}{0.7\textwidth} 
\begin{flushright}
Ars Inveniendi Analytica (2024), Paper No. 5, 59 pp.
\\
DOI 10.15781/md83-f042
\\
ISSN: 2769-8505
\end{flushright}
\end{minipage}

\ccnote

\vspace{1cm}


\begin{center}
\begin{huge}
\textit{The prescribed curvature flow on the disc}


\end{huge}
\end{center}

\vspace{1cm}


\begin{center}
{\large{\bf{Michael Struwe}}}\\
\vskip0.15cm
\footnotesize{ETH-Z\"urich}
\end{center}
\noindent

\vspace{1cm}


\begin{center}
\noindent \em{Communicated by Manuel Del Pino}
\end{center}
\vspace{1cm}


\noindent \textbf{Abstract.} \textit{
For given functions $f$ and $j$ on the disc $B$ and its boundary $\partial B=S^1$,
we study the existence of conformal metrics $g=e^{2u}g_{\R^2}$ 
with prescribed Gauss curvature $K_g=f$ and boundary geodesic curvature $k_g=j$.
Using the variational characterization of such metrics obtained by
Cruz-Blazquez and Ruiz \cite{Cruz-Blazquez-Ruiz-2018},
we show that there is a canonical negative gradient flow
of such metrics, either converging to a solution of the prescribed curvature problem,
or blowing up to a spherical cap. In the latter case, similar to our work \cite{Struwe-2005}
on the prescribed curvature problem on the sphere, we are able to exhibit a $2$-dimensional
shadow flow for the center of mass of the evolving metrics from which we
obtain existence results complementing
the results recently obtained by Ruiz \cite{Ruiz-2021} by degree-theory.}
\vskip0.3cm

\noindent \textbf{Keywords.} Conformal geometry, geometric evolution equations. 
\vspace{0.5cm}


\section{Background and results} 
\subsection{Prescribed curvature} Beginning with the work of Berger \cite{Berger-1971} and 
Kazdan-Warner \cite{Kazdan-Warner-1974} the problem of finding conformal metrics $g$ on a surface
$M$ having prescribed Gauss curvature $K_g=f$ for a given function $f$ has attracted geometric
analysts. In particular, Nirenberg's problem, that is, the study of the case 
when $M=S^2$, has given rise to sophisticated analytic approaches and deep insights
into the interplay of analysis and geometry. 

A variation of this famous problem is the case when $M$ has non-empty boundary, in 
particular, the case when $M$ is the unit disc $B=B_1(0)\subset\R^2$, where in addition 
to the Gauss curvature $K_g$ of the metric we also would like to prescribe the geodesic 
curvature $k_g$ of the boundary. 

Writing $g=e^{2u}g_0$ for the conformal metric, where $g_0$ is the Euclidean background 
metric, we have 
\begin{equation}\label{1.1}
   K_g=e^{-2u}(-\Delta u)\ \hbox{ in } B
\end{equation}
and 
\begin{equation}\label{1.2}
   k_g=e^{-u}\big(\frac{\partial u}{\partial\nu_0}+1\big)\ \hbox{ on } \partial B,
\end{equation}
respectively. Here and in the following, $\nu_0$ denotes the outward unit normal
in the Euclidean metric. For given functions $f$ and $j$, prescribing $K_g=f$ and $k_g=j$
then is equivalent to solving the nonlinear equation
\begin{equation}\label{1.3}
   -\Delta u=fe^{2u} \hbox{ in } B
\end{equation}
with the nonlinear Neumann boundary condition
\begin{equation}\label{1.4}
   \frac{\partial u}{\partial\nu_0}+1=je^u \hbox{ on } \partial B=S^1.
\end{equation}

Note that the Gauss-Bonnet theorem moreover gives the geometric constraint 
\begin{equation}\label{1.5}
   \int_BK_gd\mu_g+\int_{\partial B}k_gds_g
   =\int_Bfe^{2u}dz+\int_{\partial B}je^uds_0=2\pi,
\end{equation}
as can also be seen by integrating equations \eqref{1.3}, \eqref{1.4}. Here, $ds_g$ and $ds_0$
denote the line elements in the metrics $g$ and $g_0$, respectively. Thus, it is 
natural to assume that $f$ and $j$ are non-negative and that at least one of these
functions is positive somewhere. In fact, in our results below we will suppose that both 
$f$ and $j$ are strictly positive.

\subsection{Variational problem}
Cruz-Blazquez and Ruiz \cite{Cruz-Blazquez-Ruiz-2018} observed that the problem is variational,
and that solutions to \eqref{1.3}, \eqref{1.4} with the help of an auxiliary variable 
$0<\rho<\pi$ may be characterized as critical points of the functional 
\begin{equation}\label{1.6}
  \begin{split}
   E(u,\rho)&=E_{f,j}(u,\rho)=\frac12\int_B|\nabla u|^2dz+\int_{\partial B}uds_0\\
   &\quad-\rho\log\big(\int_Bfe^{2u}dz\big)-2(\pi-\rho)\log\big(\int_{\partial B}je^uds_0\big)\\
   &\quad+2(\pi-\rho)\log\big(2(\pi-\rho)\big)+\rho+\rho\log(2\rho).
  \end{split}
\end{equation}

Indeed, if $(u,\rho)$ with $u\in C^2(\bar{B})$, $0<\rho<\pi$ is a critical point of $E$, 
for the partial differential in direction $\varphi\in C^2(\bar{B})$ there holds 
\begin{equation}\label{1.7}
  \begin{split}
   0=\langle&\partial_uE(u,\rho),\varphi\rangle=\int_B\nabla u\nabla\varphi dz
    +\int_{\partial B}\varphi ds_0\\
   &-\frac{2\rho}{\int_Bfe^{2u}dz}\int_Bfe^{2u}\varphi dz
   -\frac{2(\pi-\rho)}{\int_{\partial B}je^uds_0}\int_{\partial B}je^u\varphi ds_0,
  \end{split}
\end{equation}
while for the partial differential with respect to $\rho$ we have
\begin{equation}\label{1.8}
  \begin{split}
   0&=\partial_{\rho}E(u,\rho)\\
   &=2\log\big(\int_{\partial B}je^uds_0\big)-\log\big(\int_Bfe^{2u}dz\big)
   -2\log\big(2(\pi-\rho)\big)+\log(2\rho)\\
   &=\log\Big(\frac{2\rho}{\int_Bfe^{2u}dz}\Big)
   -2\log\Big(\frac{2(\pi-\rho)}{\int_{\partial B}je^uds_0}\Big).
  \end{split}
\end{equation}
Thus, considering only variations $\varphi\in C^{\infty}_c(B)$ vanishing near $\partial B$,
from \eqref{1.7} we obtain the identity
\begin{equation}\label{1.9}
   -\Delta u=\frac{2\rho}{\int_Bfe^{2u}dz}fe^{2u} \hbox{ in } B.
\end{equation}
Using this, and now considering arbitrary smooth variations in \eqref{1.7}, 
we then also find the equation 
\begin{equation}\label{1.10}
   \frac{\partial u}{\partial\nu_0}+1
   =\frac{2(\pi-\rho)}{\int_{\partial B}je^uds_0}je^u \hbox{ on } \partial B=S^1.
\end{equation}
Finally, \eqref{1.8} yields  
\begin{equation*}
  \begin{split}
   \frac{2\rho}{\int_Bfe^{2u}dz}
   =\Big(\frac{2(\pi-\rho)}{\int_{\partial B}je^uds_0}\Big)^2.
  \end{split}
\end{equation*}
Therefore, if we set $\beta=\frac{2(\pi-\rho)}{\int_{\partial B}je^uds_0}>0$, we have 
\begin{equation*}
   -\Delta u=\beta^2fe^{2u} \hbox{ in } B
\end{equation*}
and
\begin{equation*}
   \frac{\partial u}{\partial\nu_0}+1=\beta je^u \hbox{ on } \partial B=S^1.
 \end{equation*}
The function $\tilde{u}=u+\log\beta$ then solves \eqref{1.3}, \eqref{1.4}. 

As shown in \cite{Cruz-Blazquez-Ruiz-2018}, Proposition 2.7,
the functional $E$ is uniformly bounded from below. Indeed, letting 
$\dashint_{\partial B}\varphi\,ds_0=\frac1{2\pi}\int_{\partial B}\varphi\,ds_0$
denote the average of a function $\varphi$ on $\partial B$,
from the Lebedev-Milin inequality 
\begin{equation}\label{1.11}
  \begin{split}
   \frac{1}{4\pi}\int_B|\nabla u|^2dz+\dashint_{\partial B}uds_0
   \ge\log\big(\dashint_{\partial B}e^uds_0\big)
  \end{split}
\end{equation}
(see for instance \cite{Osgood-et-al-1988}, formula (4')) and the Moser-Trudinger type
estimate
\begin{equation*}
  \begin{split}
   \frac{1}{2\pi}\int_B|\nabla u|^2dz+2\dashint_{\partial B}uds_0
   \ge\log\big(\dashint_Be^{2u}dz\big)
  \end{split}
\end{equation*}
proved in \cite{Cruz-Blazquez-Ruiz-2018}, Corollary 2.5,
we obtain the uniform lower bound
\begin{equation}\label{1.12}
  \begin{split}
   \inf_{u\in H^1(B),\,0<\rho<\pi}E(u,\rho)\ge C(\|f\|_{L^{\infty}},\|j\|_{L^{\infty}})>-\infty.
  \end{split}
\end{equation}

\subsection{Flow approach}
For constant functions $f$ and $j$, with one of them vanishing, flow approaches to the
solution of \eqref{1.3}, \eqref{1.4} were developed by Osgood et al. \cite{Osgood-et-al-1988}
and Brendle \cite{Brendle-2002a}. In fact, for $f\equiv 0$ one can consider families of
harmonic functions on the disc with traces evolving in time.
For non-constant functions $j>0$ and $f\equiv 0$, such a flow
approach was devised by Gehrig \cite{Gehrig-2020}, modeled on our work \cite{Struwe-2005} on a 
flow approach to the Nirenberg problem for conformal metrics of prescribed Gauss curvature 
on the sphere $S^2$. 

If neither $f$ nor $j$ vanishes, however, it is not possible to either solve \eqref{1.3} 
for each time or to impose \eqref{1.4} as boundary constraint. Instead we use the negative 
gradient flow of $E$ (in the evolving metric $g=e^{2u}g_0$) to define the prescribed 
curvature flow in this case. Thus, we seek to solve the equations 
\begin{equation}\label{1.13}
   \frac{du}{dt}=\alpha f-K=\alpha f+e^{-2u}\Delta u\hbox{ in } B\times [0,\infty[
\end{equation}
and
\begin{equation}\label{1.14}
   \frac{du}{dt}=\beta j-k=\beta j-e^{-u}\big(\frac{\partial u}{\partial\nu_0}+1\big)
   \hbox{ on } \partial B\times [0,\infty[
\end{equation}
as well as 
\begin{equation}\label{1.15}
  \frac{d\rho}{dt}=\log(\beta^2/\alpha)=-\partial_{\rho}E(u,\rho)\hbox{ on } [0,\infty[
\end{equation}
for given initial condition 
\begin{equation}\label{1.16}
   (u,\rho)\big|_{t=0}=(u_0,\rho_0), 
\end{equation}
where we let
\begin{equation}\label{1.17}
   \alpha=\alpha(t)=\frac{2\rho}{\int_Bfe^{2u}dz},\
   \beta=\beta(t)=\frac{2(\pi-\rho)}{\int_{\partial B}je^uds_0},
\end{equation}
for all $t>0$. For brevity, in the following we let $u_t=\frac{du}{dt}$ and so on. 

Then, if $(u(t),\rho(t))\equiv (v,\sigma)$ is a rest point of the
flow \eqref{1.13} - \eqref{1.15}, we may rescale $u=v+\log\beta$
to obtain a solution of \eqref{1.3}, \eqref{1.4}.

For constant functions $f>0$ and $j>0$,
equations similar to \eqref{1.13} and \eqref{1.14} were already
proposed by Brendle \cite{Brendle-2002b}, who proved global existence and exponential 
convergence of the flow towards a conformal metric having both constant Gauss curvature
and constant geodesic boundary curvature. Note that the coupling of equation \eqref{1.13} 
with the boundary condition \eqref{1.14} involves a Neumann type boundary condition of 
second order, and its treatment requires special care. Fortunately, Brendle's analysis
may be carried over to the case of non-constant functions $f$ and $j$ and non-vanishing
initial data in standard manner. Thus, analogous to Brendle's \cite{Brendle-2002b} result,
Theorem 2.5, for any smooth data there is $T>0$ such that there exists a unique solution to the
initial value problem for \eqref{1.13} - \eqref{1.15} on $[0,T]$, which is
continuos on $\bar{B}\times[0,T]$, smooth for $t>0$, and which continuously depends on the data.
However, as we shall see, 
the behavior of the flow \eqref{1.13} - \eqref{1.15} for large time may be quite subtle, and 
also equation \eqref{1.15} plays an important role.

Observe that the equations \eqref{1.13} and \eqref{1.14} give the identity
\begin{equation}\label{1.18}
 (K-\alpha f)=-u_t=(k-\beta j)\ \hbox{ on } \partial B
\end{equation}
for any $t>0$. At time $t=0$ this equation gives a compatibility condition on the data 
$u_0$, $f$, and $j$ for smoothness of the flow up to the initial time.

\subsection{Energy identity and conservation of volume} 
Integrating \eqref{1.13}, \eqref{1.14} with respect to the evolving metric $g=e^{2u}g_0$,
we see that for smooth solutions of \eqref{1.13}, \eqref{1.14} there holds
\begin{equation}\label{1.19}
  \begin{split}
   \frac{d}{dt}&\Big(\frac12\int_Be^{2u}dz+\int_{\partial B}e^uds_0\Big)\\
   &=\alpha\int_Bfe^{2u}dz+\beta\int_{\partial B}je^uds_0-\int_{\partial B}ds_0
   =2\rho+2(\pi-\rho)-2\pi=0;
  \end{split}
\end{equation}
that is, the sum
\begin{equation}\label{1.20}
   m_0:=\frac12\int_Be^{2u}dz+\int_{\partial B}e^uds_0
\end{equation}
of the area of $B$ and length of the boundary $\partial B$ is conserved for all $t>0$.

Moreover, multiplying both \eqref{1.13}, \eqref{1.14} with $u_t=\frac{du}{dt}$ and integrating
with respect to $g$, and multiplying \eqref{1.15} with $\rho_t$, we find the identity 
\begin{equation}\label{1.21}
   \int_Be^{2u}u_t^2dz+\int_{\partial B}e^uu_t^2ds_0+\rho_t^2+\frac{d}{dt}E(u)=0.
\end{equation}
In particular, the energy $E$ is non-increasing in time. Integrating also in time, 
and using that $E$ is bounded from below, for a global smooth 
solution $(u,\rho)$ of \eqref{1.13} - \eqref{1.15} we thus infer
\begin{equation}\label{1.22}
  \begin{split}  
   \int_0^{\infty}\int_B|\alpha f-K|^2&d\mu_gdt+\int_0^{\infty}\int_{\partial B}|\beta j-k|^2ds_gdt
   +\int_0^{\infty}\rho_t^2dt\\
   &=E(u_0,\rho_0)-\lim_{T\to\infty}E(u(T),\rho(T))<\infty,
  \end{split}
\end{equation}
and there exists a sequence $t_l\to\infty$ such that 
\begin{equation}\label{1.23}
  \begin{split}  
   \int_{B\times\{t_l\}}|\alpha f-K|^2d\mu_g
   +\int_{{\partial B}\times\{t_l\}}|\beta j-k|^2ds_g+\rho_t^2(t_l)\to 0\ (l\to\infty).
  \end{split}
\end{equation}

\subsection{Main results}
Even though condition \eqref{1.23} is somewhat weaker than the conditions required 
by Jevnikar et al. in \cite{Jevnikar-et-al-2020}, in Corollary \ref{cor4.4} below
we will show that similar to the conclusion of their Theorem 1.1 a dychotomy holds:
Either a subsequence of the conformal metrics $g_l=e^{2u(t_l)}g_0$ converges in
$H^{3/2}(B)\cap H^1(\partial B)$ and uniformly to a metric $g_{\infty}=e^{2u_{\infty}}g_0$
inducing a solution of \eqref{1.3}, \eqref{1.4}, or the metrics $g_l$ subsequentially
concentrate at a boundary point $z_0\in\partial B$ in the sense of measures, exhibiting
blow-up in a spherical cap.

In the latter case, moreover, in Lemmas \ref{lemma5.6} and \ref{lemma5.8} below
we are able to show that the motion of the center of mass of the evolving metrics $g(t)$
is essentially driven by a combination of the gradients of the functions
$f$ and $j$, where we extend $j$ as a harmonic function on the disc.
Quite miraculously, in Lemma \ref{lemma5.9} we are able to relate this combination to
the gradient of the function
\begin{equation}\label{1.24}
  J=j+\sqrt{j^2+f}
\end{equation}
introduced in \cite{Jevnikar-et-al-2020}. Our Proposition \ref{prop5.11} then shows that,
similar to our analysis of the prescribed curvature flow on $S^2$ in \cite{Struwe-2005}, 
if the flow concentrates at a point $z_0\in\partial B$ where $\partial J(z_0)/\partial\nu_0\neq 0$
the flow dynamics may be reduced to the $2$-dimensional ``shadow flow''
for the center of mass in terms of the components of $\nabla J$ stated in \eqref{5.31} .

This analysis yields existence results that complement recent results of
Ruiz \cite{Ruiz-2021} obtained by degree theory; in fact, the assertions in parts i) and ii) of
the following theorem also follow from his work. In contrast, a degree-theoretic argument does
not seem to be available for the result stated in part iii), which extends the result
of Gehrig \cite{Gehrig-2020} for the case $f\equiv 0$ to the case of arbitrary
smooth functions $f>0$.

\begin{theorem}\label{thm1.1}
Let $ J$ be given by \eqref{1.24} above, where we extend $j$ harmonically to the disc.
If i) $\partial J/\partial\nu_0>0$ on $\partial B$, or if
ii) $\partial J/\partial\nu_0<0$ on $\partial B$, there exists a solution to
problem \eqref{1.3}, \eqref{1.4}.

iii) Suppose that there exist points $z_i^{\pm}=e^{i\phi_i^{\pm}}\in\partial B$, $1\le i\le 2$,
with
\begin{equation*}
  0\le\phi_1^+<\phi_1^-<\phi_2^+<\phi_2^-<2\pi
\end{equation*}
such that 
\begin{equation*}
   \frac{\partial J(z_i^-)}{\partial \nu_0}<0
   <\frac{\partial J(z_i^+)}{\partial \nu_0},\ 1\le i\le 2. 
\end{equation*}
Then there exists a solution to problem \eqref{1.3}, \eqref{1.4}.
\end{theorem}

\subsection{Outline}
In Section 2 we collect some standard results about the conformal group on the disc.
For a solution $(u,\rho)$ to \eqref{1.13} - \eqref{1.15} these results allow to define a
family of normalized companion flows satisfying uniform $H^1$-bounds
in terms of the energy. These estimates in Section 3 are used to show that the flow equations
\eqref{1.13} - \eqref{1.15} admit a solution for all time $t>0$, whose long-time behavior
we analyze in Section 4.
In Section 5, finally, we focus on the regime when the flow concentrates and
derive the key equation \eqref{5.31} for the $2$-dimensional shadow flow of the
center of mass of the evolving metrics, from which we deduce Theorem \ref{thm1.1}.

\subsection{Acknowledgement}
I thank David Ruiz for helpful comments.

\section{Preliminaries} 
As in our earlier work \cite{Struwe-2005} on the prescribed curvature flow on $S^2$, for the
asymptotic analysis of the flow $u(t)$ it will be convenient to work with companion flows
$v(t)$ that are suitably normalized with respect to the action of the M\"obius group.

\subsection{M\"obius group}
Identifying $\R^2\cong\C$, we denote as $M$ the $3$-dimensional M\"obius group 
of conformal transformations of the unit disc, given by
\begin{equation*}
   M=\{\Phi(z)=e^{i\theta}\frac{z+a}{1+\bar{a}z}\in C^{\infty}(\bar{B};\bar{B}):\ 
   |a|<1,\ \theta\in\R\}.
\end{equation*}
Letting 
\begin{equation*}
   \Phi_a(z)=\frac{z+a}{1+\bar{a}z}\ \hbox{ for any }|a|<1,
\end{equation*}
we have 
\begin{equation}\label{2.1}
   \Phi_{e^{i\theta}a}(e^{i\theta}z)
   =e^{i\theta}\Phi_a(z)\ \hbox{ for any }|a|<1,\ \theta\in\R.
\end{equation}
We therefore may assume $0\le a\in\R$, whenever convenient.

Note that, letting $a=a_1+ia_2\in\C$, from the functions
\begin{equation*}
  \zeta_1(z)=\frac{d\Phi_a}{da_1}\big|_{a=0}(z)=1-z^2,\ \
  \zeta_2(z)=\frac{d\Phi_a}{da_2}\big|_{a=0}(z)=i(1+z^2),
\end{equation*}
together with the generator of pure rotations 
\begin{equation*}
  \zeta_0=iz,\ z=x+iy\in\C\cong\R^2,
\end{equation*}
we obtain a basis for the tangent space $T_{id}M$. We also may observe that for $z\in\partial B$
with $1=|z|^2=z\bar{z}$ we can easily express
\begin{equation*}
  \zeta_1(z)=|z|^2-z^2=z(\bar{z}-z)=-2y\tau,\ \
  \zeta_2(z)=i(|z|^2+z^2)=iz(\bar{z}+z)=2x\tau
\end{equation*}
in terms of the tangent vector field $\tau=iz$ along $\partial B$.

\subsection{Normalization}
For $u\in H^1(B)$, $g=e^{2u}g_0$, $\Phi\in M$ we define 
\begin{equation}\label{2.2}
  v=u\circ\Phi+\frac12\log(\det d\Phi)\ \hbox{ in } B,\ h=\Phi^*g=e^{2v}g_0.
\end{equation}
Note that since $\Phi$ is conformal, if we denote as $\Phi'$ the derivative of the restriction
$\Phi\colon\partial B\to\partial B$ of $\Phi$ to the boundary,
there holds $\log|\Phi'|=\frac12\log(\det d\Phi)$ on $\partial B$.
Thus, for any function $w$ with $d\mu_g=e^{2u}dz$, $ds_g=e^uds_0$, etc., we have the identities
\begin{equation}\label{2.3}
   \int_Bw\,d\mu_g=\int_Bw\circ\Phi\,d\mu_h, \
   \int_{\partial B}w\,ds_g=\int_{\partial B}w\circ\Phi\,ds_h.
\end{equation}

The  normalization that we choose will depend on a number $0<R<1$. Indeed, for any $z_0\in\bar{B}$
we let
\begin{equation}\label{2.4}
  R(z_0)=\sqrt{1+j(z_0)^2/f(z_0)}-j(z_0)/\sqrt{f(z_0)}
\end{equation}
and we set  
\begin{equation}\label{2.5}
  0<R_0:=\inf_{z_0\in B}R(z_0)\le\sup_{z_0\in B}R(z_0)=R_1.
\end{equation}
For any $R_0\le R\le R_1$, we let 
\begin{equation*}
  \begin{split}
    \psi_R(z):=\frac{2Rz}{1+|Rz|^2}=\pi_{\R^2}\Psi_R(z),
  \end{split}
\end{equation*}
where $\Psi_R(z)=\Psi(Rz)$ is the scaled inverse $\Psi\colon\R^2\to S^2$ of stereographic
projection, and where $\pi_{\R^2}\colon\R^3\to\R^2$ is the orthogonal projection.
Then given any $u\in H^1(B)$,
following the proof of Chang-Liu \cite{Chang-Liu-1996}, Theorem 3.1,
Onofri \cite{Onofri-1982}, p.324, or Chang-Yang \cite{Chang-Yang-1987}, Appendix,
for any $R_0\le R\le R_1$ we can find $\Phi\in M$ such that for $v$ as given in \eqref{2.2}
above there holds
\begin{equation}\label{2.6}
  \begin{split}
    \frac12\int_B\psi_R(z)e^{2v}dz+\int_{\partial B}\psi_R(z)e^vds_0=0.
  \end{split}
\end{equation}
We interpret this condition as fixing the center of mass of the normalized metric
$h=e^{2v}g_0$ lifted to the sphere $S^2$ by means of $\Psi_R$.

For reasons that will become clear in \eqref{4.5} in the proof of Proposition \ref{prop4.1}
we call $R(z_0)$ the {\it scaling radius} at a point $z_0\in\partial B$. Note that computing
\begin{equation*}
  J(z_0)R(z_0)\sqrt{f(z_0)}
  =\big(\sqrt{f(z_0)+j(z_0)^2}+j(z_0)\big)\big(\sqrt{f(z_0)+j(z_0)^2}-j(z_0)\big)=f(z_0)
\end{equation*}
we can interpret $J(z_0)=\sqrt{f(z_0)}/R(z_0)$ at any $z_0\in\partial B$.

Given any solution $u=u(t)$ of the flow equations \eqref{1.13}-\eqref{1.15} we associate
with $u$ the family of flows $v_R=v_R(t)$, $R_0\le R\le R_1$, normalized in this way. For each such
$v=v_R$ the following considerations apply.

Note that condition \eqref{2.6} is conserved if we rotate our system of coordinates. 
Suitably normalizing with respect to rotations, however, the map $\Phi$ achieving \eqref{2.6}
smoothly depends on $u$, and for any family $u=u(t)\in C^1(I,H^1(B))$ for the corresponding 
families $v=v_R(t)$ of normalized functions as in \cite{Struwe-2005}, formulas (17) and (18), 
there holds
\begin{equation*}
  v_t=u_t\circ\Phi+\frac12e^{-2v}div(\xi e^{2v})\hbox{ on } B,\
  v_t=u_t\circ\Phi+e^{-v}\frac{\partial(\xi^{\tau}e^v)}{\partial\tau}\hbox{ on } \partial B,
\end{equation*}
where $\xi=(d\Phi)^{-1}\Phi_t\in T_{id}M$ and $\xi ^{\tau}=\tau\cdot\xi$ on $\partial B$.
Differentiating, observing that $\xi\cdot\nu_0=0$,
$\frac{\partial\psi_R}{\partial\tau}=\frac{2R\tau}{1+R^2}$ on $\partial B$,
from \eqref{2.6} we obtain
\begin{equation*}
  \begin{split}
    0=\int_B\psi_Rv_td\mu_h+\int_{\partial B}&\psi_Rv_tds_h
    =\int_B\psi_Ru_t\circ\Phi d\mu_h+\int_{\partial B}\psi_Ru_t\circ\Phi ds_h\\
    &-\frac12\int_Bd\psi_R\,\xi d\mu_h-\frac{2R}{1+R^2}\int_{\partial B}\tau\,\xi^{\tau}ds_h.
  \end{split}
\end{equation*}
Since $T_{id}M$ is finite-dimensional, with a uniform constant $C>0$ for any $h$ near a positive
scalar multiple of the Euclidean metric $g_0$ there holds
\begin{equation*}
    \|\xi\|_{L^{\infty}}
    \le C\big|\frac12\int_Bd\psi_R\,\xi d\mu_h
    +\frac{2R}{1+R^2}\int_{\partial B}\tau\,\xi^{\tau}ds_h \big|.
\end{equation*}
Thus, from \eqref{2.3} with H\"older's inequality and \eqref{1.19} we obtain the bound 
\begin{equation}\label{2.7}
  \begin{split}
    \|\xi\|_{L^{\infty}}
    &\le C\big|\int_B\psi_Ru_t\circ\Phi d\mu_h+\int_{\partial B}\psi_Ru_t\circ\Phi ds_h\big|\\
    &\le C\int_B|u_t|d\mu_g+C\int_{\partial B}|u_t|ds_g
    \le C\big(\int_Bu_t^2d\mu_g+\int_{\partial B}u_t^2ds_g\big)^{1/2}.
     \end{split}
\end{equation}

\subsection{Improved bounds}
With the help of arguments by Aubin \cite{Aubin-1979}, for the normalized functions $v$
the Lebedev-Milin inequality \eqref{1.11} may be improved. In fact, we have
the following result similar to Osgood et al. \cite{Osgood-et-al-1988}, formula (5). 

\begin{lemma}\label{lemma2.1}
With a constant $C=C(R_0,R_1)\ge 0$, for any $v\in H^1(B)$ satisfying condition \eqref{2.6} 
for some $R\in[R_0,R_1]$ there holds 
\begin{equation*}
  \frac{1}{6\pi}\int_B|\nabla v|^2dz+\dashint_{\partial B}vds_0
  \ge\max\big\{\log\big(\dashint_{\partial B}e^vds_0\big),
    \frac12\log\big(\dashint_Be^{2v}dz\big)\big\}-C
\end{equation*}
\end{lemma}

\begin{proof}
i) The divergence theorem and H\"older's inequality give
\begin{equation}\label{2.8}
  \begin{split}
    \dashint_{\partial B}&e^vds_0=\dashint_{\partial B}e^vz\cdot\nu_0ds_0
    =\frac1{2\pi}\int_Bdiv(ze^v)dz\\
    &=\dashint_Be^vdz+\frac1{2\pi}\int_Bz\cdot\nabla ve^vdz
      \le\big(1+\|\nabla v\|_{L^2(B)}\big)\big(\dashint_Be^{2v}dz\big)^{1/2}.
  \end{split}
\end{equation}
Similarly we have
\begin{equation*}
  \dashint_{\partial B}vds_0
  =\frac1{2\pi}\int_Bdiv(zv)dz=\dashint_Bvdz+\frac1{2\pi}\int_Bz\cdot\nabla vdz,
\end{equation*}
and letting $\bar{v}=\dashint_B vdz$ we find
\begin{equation}\label{2.9}
  \big|\dashint_{\partial B}vds_0-\bar{v}\big|\le\|\nabla v\|_{L^2(B)}.
\end{equation}

Also splitting $e^{2v}=e^{2\bar{v}}e^{2(v-\bar{v})}$ to obtain
\begin{equation*}
   \frac12\log\big(\dashint_Be^{2v}dz\big)
   =\bar{v}+\frac12\log\big(\dashint_Be^{2(v-\bar{v})}dz\big),
\end{equation*}
and bounding $log\big(1+\|\nabla v\|_{L^2(B)}\big)\le 1+\|\nabla v\|_{L^2(B)}$,
from \eqref{2.8} we then conclude
\begin{equation}\label{2.10}
  \log\big(\dashint_{\partial B}e^vds_0\big)-\dashint_{\partial B}vds_0
  \le 1+2\|\nabla v\|_{L^2(B)}+\frac12\log\big(\dashint_Be^{2(v-\bar{v})}dz\big).
\end{equation}

ii) Following Aubin \cite{Aubin-1979}, proof of Theorem 6, we let
$\Omega^{\pm}_i=\{z\in\bar{B};\; \pm z_i\ge1/2\}$, 
$K^{\pm}_i=\{z\in\bar{B};\; \pm z_i\ge 0\}$, $1\le i\le 2$, and set 
$\Omega_0=B_{3/4}(0)$. We then also let
$0\le\varphi^{\pm}_i,\psi^{\pm}_i\le 1$, $0\le\varphi_0\le 1$ be smooth
cut-off functions such that
\begin{equation*}\varphi^{\pm}_i=1\hbox{ in }\Omega^{\pm}_i,\
\ \psi^{\pm}_i=1\hbox{ in }K^{\pm}_i,\ \varphi_0=1\hbox{ in }\Omega_0,
\end{equation*}
and satisfying the conditions
\begin{equation*}
  supp(\varphi_0)\subset B,\ supp(\varphi^{\pm}_i)\cap supp(\psi^{\mp}_i)=\emptyset,\ 1\le i\le 2.
\end{equation*}

Noting that $\sqrt{2}\ge4/3$, we see that
\begin{equation*}
  \partial B\subset\cup_{1\le i\le 2}(\Omega^+_i\cup\Omega^-_i),\
  \bar{B}\subset\cup_{1\le i\le 2}(\Omega^+_i\cup\Omega^-_i)\cup\Omega_0.
\end{equation*}
Thus we have
\begin{equation*}
  \begin{split}
    \int_{\partial B}e^{v-\bar{v}}ds_0\le\sum_{1\le i\le 2}
    \int_{\partial B\cap\Omega^+_i}e^{v-\bar{v}}ds_0
    +\sum_{1\le i\le 2}\int_{\partial B\cap\Omega^-_i}e^{v-\bar{v}}ds_0,
  \end{split}
\end{equation*}
as well as
\begin{equation*}
  \begin{split}
    \int_Be^{2(v-\bar{v})}dz\le\sum_{1\le i\le 2}\int_{\Omega^+_i}e^{2(v-\bar{v})}dz
    +\sum_{1\le i\le 2}\int_{\Omega^-_i}e^{2(v-\bar{v})}dz+\int_{\Omega_0}e^{2(v-\bar{v})}dz.
  \end{split}
\end{equation*}

First suppose that there holds
\begin{equation*}
  \begin{split}
    \int_{\Omega_0}&e^{2(v-\bar{v})}dz\\
    &\ge\sup\{\int_{\Omega^{\pm}_i}e^{2(v-\bar{v})}dz
    +2\int_{\partial B\cap\Omega^{\pm}_i}e^{v-\bar{v}}ds_0;\ 1\le i\le 2\}=:A.
  \end{split}
\end{equation*}
We then have
\begin{equation*}
  \begin{split}
    \int_B&e^{2(v-\bar{v})}dz\le5\int_{\Omega_0}e^{2(v-\bar{v})}dz\\
    &\le 5\int_Be^{2(v-\bar{v})\varphi_0}dz
    \le Cexp(\frac1{4\pi}\|\nabla \big((v-\bar{v})\varphi_0\big)\|^2_{L^2(B)}),
   \end{split}
\end{equation*}
where we have used Moser's sharp form of the critical Sobolev space embedding 
as in \cite{Aubin-1982}, Corollary 2.49, in the last estimate.
With \eqref{2.10} it thus also follows that
\begin{equation*}
  \begin{split}
    \log\big(\dashint_{\partial B}e^vds_0\big)-\dashint_{\partial B}vds_0
    \le C+2\|\nabla v\|_{L^2(B)}
    +\frac1{8\pi}\|\nabla\big( (v-\bar{v})\varphi_0\big)\|^2_{L^2(B)}.
   \end{split}
\end{equation*}
Arguing as Aubin \cite{Aubin-1979}, proof of Theorem 6, we then obtain the claim.

Similarly, if for some $1\le i_0\le 2$ there holds
\begin{equation*}
  \begin{split}
    A=\int_{\Omega^+_{i_0}}e^{2(v-\bar{v})}dz
    +2\int_{\partial B\cap\Omega^+_{i_0}}e^{v-\bar{v}}ds_0
    \ge\int_{\Omega_0}e^{2(v-\bar{v})}dz,
    \end{split}
\end{equation*}
we can bound
\begin{equation}\label{2.11}
  \begin{split}
    \int_B&e^{2(v-\bar{v})}dz\le 5A=5\Big(\int_{\Omega^+_{i_0}}e^{2(v-\bar{v})}dz
    +2\int_{\partial B\cap\Omega^+_{i_0}}e^{v-\bar{v}}ds_0\Big)
   \end{split}
\end{equation}
Continuing to argue as Aubin \cite{Aubin-1979}, now suppose that
\begin{equation*}
  \|\nabla\big((v-\bar{v})\varphi^+_{i_0}\big)\|^2_{L^2(B)}
  \le\|\big(\nabla (v-\bar{v})\psi^-_{i_0}\big)\|^2_{L^2(B)}
\end{equation*}
so that for arbitrarily small $\varepsilon>0$ we obtain
\begin{equation*}
  \begin{split}
    2\|\nabla\big((v-\bar{v})&\varphi^+_{i_0}\big)\|^2_{L^2(B)}
    \le\|\nabla\big((v-\bar{v})\varphi^+_{i_0}\big)\|^2_{L^2(B)}+
    \|\nabla\big((v-\bar{v})\psi^-_{i_0}\big)\|^2_{L^2(B)}\\
    &\le(1+\varepsilon)\|\nabla (v-\bar{v})\|^2_{L^2(B)}+C(\varepsilon))\|v-\bar{v}\|^2_{L^2(B)}.
  \end{split}
\end{equation*}
Extending $v$ by letting $v(z)=v(z/|z|^2)$ for $|z|>1$, and similarly for $\varphi^+_{i_0}$,
we then have 
\begin{equation*}
  \begin{split}
    &\int_{\Omega^+_{i_0}}e^{2(v-\bar{v})}dz\le\int_Be^{2(v-\bar{v})\varphi^+_{i_0}}dz
    \le\int_{\R^2}e^{2(v-\bar{v})\varphi^+_{i_0}}dz\\
    &\quad\le C\exp\big(\frac1{4\pi}\|\nabla\big((v-\bar{v})\varphi^+_{i_0}\big)\|^2_{L^2(\R^2)}\big)
    =C\exp\big(\frac1{2\pi}\|\nabla\big((v-\bar{v})\varphi^+_{i_0}\big)\|^2_{L^2(B)})\big)\\
    &\quad\le C\exp\big(\frac{1+\varepsilon}{4\pi}\|\nabla (v-\bar{v})\|^2_{L^2(B)}
    +C(\varepsilon)\|v-\bar{v}\|^2_{L^2(B)}\big).
   \end{split}
\end{equation*}
Moreover, using \eqref{1.11} and \eqref{2.9}, we can estimate
\begin{equation*}
    \int_{\partial B\cap\Omega^+_{i_0}}e^{v-\bar{v}}ds_0
    \le\int_{\partial B}e^{v-\bar{v}}ds_0
    \le C\exp\big(\frac1{4\pi}\|\nabla v\|^2_{L^2(B)}+\|\nabla v\|_{L^2(B)}\big)
\end{equation*}
and with the help of \eqref{2.10}, \eqref{2.11} the proof again may be completed as in 
Aubin \cite{Aubin-1979}, proof of Theorem 6.

On the other hand, if
\begin{equation*}
  \|\nabla\big((v-\bar{v})\varphi^+_{i_0}\big)\|^2_{L^2(B)}
  >\|\nabla\big((v-\bar{v})\psi^-_{i_0}\big)\|^2_{L^2(B)},
\end{equation*}
in view of \eqref{2.6} we may estimate
\begin{equation*}
  \begin{split}
    A&\le\frac2R\Big(\int_{K^+_{i_0}}\frac{2Rz_{i_0}}{1+R^2|z|^2}e^{2(v-\bar{v})}dz
    +2\int_{K^+_{i_0}\cap\partial B}\frac{2Rz_{i_0}}{1+R^2|z|^2}e^{v-\bar{v}}ds_0\Big)\\
    &=-\frac2R\Big(\int_{K^-_{i_0}}\frac{2Rz_{i_0}}{1+R^2|z|^2}e^{2(v-\bar{v})}dz
    +2\int_{K^-_{i_0}\cap\partial B}\frac{2Rz_{i_0}}{1+R^2|z|^2}e^{v-\bar{v}}ds_0\Big)\\
    &\le4\Big(\int_Be^{2(v-\bar{v})\psi^-_{i_0}}dz
    +2\int_{\partial B}e^{(v-\bar{v})\psi^-_{i_0}}ds_0\Big),
  \end{split}
\end{equation*}
where now   
\begin{equation*}
  \begin{split}
  2\|\nabla\big((v-\bar{v})&\psi^-_{i_0}\big)\|^2_{L^2(B)}
  \le\|\nabla\big((v-\bar{v})\varphi^+_{i_0}\big)\|^2_{L^2(B)}
  +\|\nabla\big((v-\bar{v})\psi^-_{i_0}\big)\|^2_{L^2(B)}\\
  &\le(1+\varepsilon)\|\nabla (v-\bar{v})\|^2_{L^2(B)}+C(\varepsilon)\|v-\bar{v}\|^2_{L^2(B)}.
  \end{split}
\end{equation*}
As above we then can bound
\begin{equation*}
  \begin{split}
    \int_Be^{2(v-\bar{v})\psi^-_{i_0}}dz
    \le C\exp\big(\frac{1+\varepsilon}{4\pi}\|\nabla (v-\bar{v})\|^2_{L^2(B)}
    +C(\varepsilon))\|v-\bar{v}\|^2_{L^2(B)}\big)
   \end{split}
\end{equation*}
as well as
\begin{equation*}
    \int_{\partial B}e^{(v-\bar{v})\psi^-_{i_0}}ds_0
    \le\int_{\partial B}e^{v-\bar{v}}ds_0
    \le C\exp\big(\frac1{4\pi}\|\nabla v\|^2_{L^2(B)}+\|\nabla v\|_{L^2(B)}\big),
\end{equation*}
and the proof may be completed as before. 

The same arguments may be applied if $A$ is attained on some $\Omega^-_{i_0}$.
\end{proof}

As a consequence, for functions normalized by \eqref{2.6} the $H^1$-norm is 
bounded by the energy. 

\begin{lemma}\label{lemma2.2}
For any $v\in H^1(B)$ satisfying the condition \eqref{2.6} for some $R\in[R_0,R_1]$
and any $0<\rho<\pi$, with a constant $C=C(R_0,R_1,\|f\|_{L^{\infty}},\|j\|_{L^{\infty}})>0$
there holds 
\begin{equation*}
  \begin{split}
   E(v,\rho)\ge\frac{1}{6}\int_B|\nabla v|^2dz-C.
  \end{split}
\end{equation*}
\end{lemma}

Since, as we next observe, the energy is invariant under conformal transformations,
for any $u\in H^1(B)$ the $H^1$-norm of any normalized representative $v$ of 
$u\in H^1(B)$, given by \eqref{2.2}, is in fact bounded by the energy of $u$.

\begin{lemma}\label{lemma2.3}
For any $u\in H^1(B)$, any $0<\rho<\pi$, and any $\Phi\in M$ there holds 
\begin{equation*}
  \begin{split}
   E_{f,j}(u,\rho)=E_{f\circ\Phi,j\circ\Phi}(v,\rho)
  \end{split}
\end{equation*}
for $v$ as given in \eqref{2.2}.
\end{lemma}

\begin{proof}
Let 
\begin{equation*}
  E_0(u)=\frac12\int_B|\nabla u|^2dz+\dashint_{\partial B}e^uds_0,\ u\in H^1(B).
\end{equation*}
It suffices to show that $E_0(u)=E_0(v)$, where $v$ is as above. But this is precisely the 
assertion of Chang-Yang \cite{Chang-Yang-1987}, Proposition 2.1, or 
Chang-Liu \cite{Chang-Liu-1996}, Theorem 2.1.
\end{proof}

Combining Lemmas \ref{lemma2.2} and \ref{lemma2.3} with \eqref{1.12},
we have the following useful bound.

\begin{corollary}\label{cor2.4}
For any $u\in H^1(B)$, any $0<\rho<\pi$, with a constant $C>0$ 
depending only on $R_0$, $R_1$, $\|f\|_{L^{\infty}}$, and $\|j\|_{L^{\infty}}$ there holds
\begin{equation*}
  \begin{split}
   E_{f,j}(u,\rho)\ge\frac16\int_B|\nabla v|^2dz-C
  \end{split}
\end{equation*}
for $v$ as given in \eqref{2.2}, satisfying \eqref{2.6} for some $R\in[R_0,R_1]$, 
and $v$ as well as $e^{pv}$ are bounded in $L^2(\partial B)$ and in $L^2(B)$ for any $p\in\R$
in terms of $E_{f,j}(u,\rho)$ and the number $m_0$ defined in \eqref{1.20}.
\end{corollary}

\begin{proof}
It remains to prove the assertions about integrability of $v$ and $e^v$. 
In fact, once we achieve to bound the averages $\hat{v}=\dashint_{\partial B}vds_0$ and 
$\bar{v}=\dashint_Bvdz$, these will follow from Poincar\'e's inequality, or
the Lebedev-Milin inequality \eqref{1.11} and the Moser-Trudinger inequality,
respectively, applied to multiples $w=pv$ of $v$.

Note that by a variant of the Poincar\'e inequality with a uniform constant $C>0$ we have 
\begin{equation*}
  \int_{\partial B}|v-\hat{v}|^2ds_0+\int_B|v-\bar{v}|^2dz
  \le C\int_B|\nabla v|^2dz,
\end{equation*}
and then also $|\bar{v}-\bar{v}|\le C\|\nabla v\|_{L^2(B)}$.

Writing $V:=e^{\bar{v}}$, on the other hand from \eqref{1.20} we then infer the equation
\begin{equation*}
  \begin{split}
    2m_0&=\int_Be^{2v}dz+2\int_{\partial B}e^vds_0\\
    &=e^{2\bar{v}}\int_Be^{2(v-\bar{v})}dz+2e^{\hat{v}}\int_{\partial B}e^{v-\hat{v}}ds_0
    =AV^2+2BV
  \end{split}
\end{equation*}
with coefficients
\begin{equation*}
    A=\int_Be^{2(v-\bar{v})}dz,\
    B=e^{(\hat{v}-\bar{v})}\int_{\partial B}e^{v-\hat{v}}ds_0>0,
\end{equation*}
bounded uniformly from above and away from zero in terms of $\|\nabla v\|_{L^2(B)}$.
The desired bounds follow.
\end{proof}

\section{Global existence of the flow}
Given smooth data $(u_0,\rho_0)$,
the analysis of Brendle \cite{Brendle-2002b} guarantees the existence
of a unique solution $u=u(t)$ to the flow \eqref{1.13} - \eqref{1.15} on a time interval
$[0,T]$ for some $T>0$, which is continuous on $[0,T]$ and smooth for $t>0$. Our aim in this 
section is to show that $u$ may be extended for all time $0<t<\infty$.
In a first step we will show that the function $\rho=\rho(t)$ stays 
strictly bounded away from the values $\rho=0$ and $\rho=\pi$, and we establish analogous 
bounds for the functions $\alpha$ sand $\beta$. Constants appearing below may
tacitly depend on the data $(u_0,\rho_0)$ as well as $\|f\|_{L^{\infty}}$ and $\|j\|_{L^{\infty}}$. 

\subsection{Bounds for $\rho$, $\alpha$, and $\beta$}
The results in the previous section will help us establish the following proposition.

\begin{proposition}\label{prop3.1}
There are numbers $0<\rho_1\le\rho_2<\pi$ independent of $T$ such that for any $0<t<T$ 
there holds $\rho_1\le\rho(t)\le\rho_2$.
\end{proposition}

For the proof we need the following auxiliary result, complementing \eqref{1.20}.

\begin{lemma}\label{lemma3.2}
There are constants $c,d>0$ independent of $T>0$ such that for any $t< T$ there holds 
\begin{equation*}
    \int_Be^{2u}dz\ge c,\ \int_{\partial B}e^uds_0\ge d.
\end{equation*}
\end{lemma}

\begin{proof} 
As in Corollary \ref{cor2.4}, for any $0<t<T$ and some fixed $R\in[R_0,R_1]$
consider the normalized function 
$v=v(t)=u\circ\Phi+\frac12\log(\det d\Phi)$ related to $u=u(t)$. 
By Corollary \ref{cor2.4} and the energy identity \eqref{1.21} we have the uniform bound 
\begin{equation}\label{3.1}
 \begin{split}
    \frac{1}{6}\int_B|\nabla v|^2dz&-C\le E_{f\circ\Phi,j\circ\Phi}(v,\rho)\\
    &=E_{f,j}(u,\rho)\le E_{f,j}(u_0,\rho(0))=:C_0<\infty.
 \end{split}
\end{equation}
Let $0<t_l<T$ be such that 
\begin{equation*}
    \int_Be^{2u(t_l)}dz=\int_Be^{2v(t_l)}dz\to\inf_{0<t<T}\int_Be^{2u(t)}dz.
\end{equation*}
By \eqref{3.1} and Corollary \ref{cor2.4}
we may assume that $v_l=v(t_l)\rightharpoondown v$ weakly in $H^1(B)$ 
as $l\to\infty$. Compactness of the map $H^1(B)\ni u\to e^{2u}\in L^1(B)$ then gives 
convergence
\begin{equation*}
    \int_Be^{2v_l}dz\to\int_Be^{2v}dz,
\end{equation*}
and it follows that $c:=\inf_{0<t<T}\int_Be^{2u(t)}dz>0$, with a constant depending 
only on $C_0$ but independent of $T$. 

Similarly, we find the uniform lower bound $d:=\inf_{0<t<T}\int_{\partial B}e^{u(t)}ds_0>0$.
\end{proof}

\begin{proof}[Proof of Proposition \ref{prop3.1}]
Arguing indirectly, suppose by contradiction that there is a sequence of times $0<t_l<T$ 
such that as $l\to\infty$ we have $\rho(t_l)\downarrow 0$, while $\rho_t(t_l)\le 0$.
But by Lemma \ref{lemma3.2} then 
\begin{equation*}
  -\log\Big(\frac{2\rho}{\int_Bfe^{2u}dz}\Big)\to\infty \ \hbox{ at } t_l
  \hbox{ as } l\to\infty,
\end{equation*}
and for sufficiently large $l\in\N$ at $t=t_l$ by \eqref{1.15} there holds 
\begin{equation*}
   \frac{d\rho}{dt}=2\log\Big(\frac{2(\pi-\rho)}{\int_{\partial B}je^uds_0}\Big)
   -\log\Big(\frac{2\rho}{\int_Bfe^{2u}dz}\Big)>0,
\end{equation*}
contrary to assumption. The bound $\rho(t)\le\rho_2<\pi$ is obtained similarly.
\end{proof}

Proposition \ref{prop3.1}, \eqref{1.20},
and Lemma \ref{lemma3.2} then also imply the following bounds.

\begin{proposition}\label{prop3.3}
There are numbers $0<\alpha_0\le\alpha_1<\infty$, $0<\beta_0\le\beta_1<\infty$ 
independent of $T$ such that for any $0<t<T$ there holds 
$\alpha_0\le\alpha(t)\le\alpha_1$, $\beta_0\le\beta(t)\le\beta_1$.
\end{proposition}

\subsection{Bounds for $u$, $K_g$, and $k_g$}
Upper bounds for $u$ can easily be obtained with the help of the maximum principle.

\begin{proposition}\label{prop3.4}
For any $t<T$ there holds 
\begin{equation*}
  \sup_Bu(t)\le\sup_Bu_0+t\big(\alpha_1\|f\|_{L^{\infty}}+\beta_1\|j\|_{L^{\infty}}\big)
  =:u_1(t).
\end{equation*}
\end{proposition}

\begin{proof}
Suppose by contradiction that $\sup_{B\times[0,T[}(u(t)-u_1(t))>0$ and let 
$(z_0,t_0)\in\bar{B}\times]0,T[$ be a point such that 
\begin{equation*}
  \sup_{B\times[0,t_0]}(u(t)-u_1(t))=(u(z_0,t_0)-u_1(t_0))>0.
\end{equation*}
Note that we necessarily have $t_0>0$.
If $z_0\in B$, clearly $\Delta u(z_0,t_0)\le 0$, and by \eqref{1.13} at $(z_0,t_0)$ there holds 
\begin{equation*}
  u_t\le\alpha_1\|f\|_{L^{\infty}}<u_{1,t},
\end{equation*}
contradicting the choice of $(z_0,t_0)$. Similarly, if $z_0\in\partial B$,
using \eqref{1.14} and the fact that $\frac{\partial u}{\partial\nu_0}(z_0,t_0)\ge 0$ we find
\begin{equation*}
  u_t\le\beta_1\|j\|_{L^{\infty}}<u_{1,t},
\end{equation*}
at $(z_0,t_0)$, and again we obtain a contradiction.
\end{proof}
 
Similarly, bounds for $K=K_g$ and $k=k_g$ follow from the evolution equations for curvature.
For convenience, for each $t<T$ we extend the function $k=k(t)$ as a harmonic function in $B$.

From \eqref{1.1} and \eqref{1.13} we obtain 
\begin{equation}\label{3.2}
  K_t=-2u_tK-e^{-2u}\Delta u_t=2(K-\alpha f)K+e^{-2u}\Delta(K-\alpha f)\ \hbox{ in } B,
\end{equation}
and from \eqref{1.2} and equations \eqref{1.13} and \eqref{1.14} we have
\begin{equation}\label{3.3}
  k_t=-u_tk+e^{-u}\frac{\partial u_t}{\partial\nu_0}=(k-\beta j)k
  -e^{-u}\frac{\partial(K-\alpha f)}{\partial\nu_0}\ \hbox{ on } \partial B.
\end{equation}
Moreover, the definitions \eqref{1.17} and \eqref{1.19} give
\begin{equation}\label{3.4}
  \frac{\alpha_t}{\alpha}=\frac{\rho_t}{\rho}-\frac{2\int_Bfe^{2u}u_tdz}{\int_Bfe^{2u}dz}
  =\frac{\rho_t+\alpha\int_Bfe^{2u}(K-\alpha f)dz}{\rho}
\end{equation}
and
\begin{equation}\label{3.5}
  \frac{\beta_t}{\beta}=\frac{-\rho_t}{\pi-\rho}
  -\frac{\int_{\partial B}je^uu_tds_0}{\int_{\partial B}je^uds_0}
   =\frac{-2\rho_t+\beta\int_{\partial B}je^u(k-\beta j)ds_0}{2(\pi-\rho)},
\end{equation}
respectively. Lower bounds for $K$ and $k$ now can again be obtained with the help of the 
maximum principle.

Let
\begin{equation*}
  C_1:=\max\{\sup_t(\alpha\rho_t\rho^{-1})\|f\|_{L^{\infty}},
  \sup_t(\alpha^2\rho^{-1})\|f\|^2_{L^{\infty}}\}<\infty,
\end{equation*}
and set 
\begin{equation*}
 C_2:=\max\{\sup_t(2\beta|\rho_t|(\pi-\rho)^{-1})\|j\|_{L^{\infty}},
 \sup_t(\beta^2(\pi-\rho)^{-1})\|j\|^2_{L^{\infty}}\}<\infty.
\end{equation*}
By Propositions \ref{prop3.1} and \ref{prop3.3} and \eqref{1.15} the constants $C_1,C_2>0$
are independent of $T>0$.

\begin{proposition}\label{prop3.5}
For any $t<T$ there holds 
\begin{equation*}
  \inf_{\partial B}(k-\beta j)\ge\inf_B(K-\alpha f)\ge\kappa
 \end{equation*}
where $\kappa<0$ is any constant independent of $T>0$ such that
\begin{equation*}
 \kappa<-2\big(\|K_{g_0}\|_{L^{\infty}}+\alpha_1\|f\|_{L^{\infty}}+\beta_1\|j\|_{L^{\infty}}\big)
\end{equation*}
and such that, in addition, there holds
  $\kappa^2-C_i+2m_0C_i\kappa>0$, $i=1,2$. 
\end{proposition}

\begin{proof}
Fix any $\kappa<0$ as above. Suppose by contradiction that for some $0<t<T$ there holds
\begin{equation*}
  \inf_B(K(t)-\alpha(t)f)<\kappa,
\end{equation*}
and let $(z_0,t_0)\in\bar{B}\times[0,T[$ be a point such that 
\begin{equation*}
  \inf_{B\times[0,t_0]}(K(t)-\alpha(t)f)=(K(z_0,t_0)-\alpha(t_0)f(z_0))<\kappa<0.
\end{equation*}
Note that we necessarily have $t_0>0$ as well as 
\begin{equation*}
  K(z_0,t_0)<-\alpha(t_0) f(z_0)<0.
\end{equation*}
Thus, from \eqref{3.2} and \eqref{3.4} with $C_1$ as above we deduce the lower bound
\begin{equation*}
  \begin{split}
  (K&-\alpha f)_t=2(K-\alpha f)K+e^{-2u}\Delta(K-\alpha f)\\
  &\qquad\qquad\qquad\qquad-\frac{\alpha\rho_tf+\alpha^2f\int_Bfe^{2u}(K-\alpha f)dz}{\rho}\\
  &\ge(K-\alpha f)^2+e^{-2u}\Delta(K-\alpha f)-C_1\int_Be^{2u}(K-\alpha f)_+dz-C_1
   \end{split}
\end{equation*}
at $(z_0,t_0)$, where $s_{\pm}=\max\{\pm s,0\}$ for $s\in\R$.

But by \eqref{1.5} and \eqref{1.17} at any time $t>0$ we have 
\begin{equation*}
  \int_Be^{2u}(K-\alpha f)dz+\int_{\partial B}e^u(k-\beta j)ds_0=0,
\end{equation*}
so that 
\begin{equation}\label{3.6}
  \begin{split}
    \int_B&e^{2u}(K-\alpha f)_+dz\le\int_Be^{2u}(K-\alpha f)_+dz+\int_{\partial B}e^u(k-\beta j)_+ds_0\\
    &=\int_Be^{2u}(K-\alpha f)_-dz+\int_{\partial B}e^u(k-\beta j)_-ds_0\\
    &\le\int_Be^{2u}(K-\alpha f)_-dz+2\int_{\partial B}e^u(k-\beta j)_-ds_0
   \end{split}
\end{equation}
Recalling \eqref{1.18} we next observe that we have
\begin{equation}\label{3.7}
  \inf_{\partial B}(k-\beta j)=\inf_{\partial B}(K-\alpha f)\ge\inf_B(K-\alpha f)
\end{equation}
for each $t>0$. Thus, with \eqref{1.20}  from \eqref{3.6} we can bound 
\begin{equation}\label{3.8}
  \begin{split}
    &\int_Be^{2u}(K-\alpha f)_+dz
    \le\int_Be^{2u}(K-\alpha f)_-dz+2\int_{\partial B}e^u(k-\beta j)_-ds_0\\
    &\quad\le\big(\int_Be^{2u}dz+2\int_{\partial B}e^uds_0\big)\sup_B(K-\alpha f)_-
    =2m_0\sup_B(K-\alpha f)_-.
   \end{split}
\end{equation}

Hence, if $z_0\in B$, in view of $\Delta(K-\alpha f)(z_0,t_0)\ge 0$,  by definition of $\kappa$
we have 
\begin{equation*}
  (K-\alpha f)_t\ge(K-\alpha f)^2-C_1+2m_0C_1(K-\alpha f)>0
\end{equation*}
at $(z_0,t_0)$, contradicting the choice of $(z_0,t_0)$. 

On the other hand, if $z_0\in\partial B$, we use \eqref{3.3} and \eqref{3.5} to write
\begin{equation*}
  \begin{split}
 (K-\alpha f)_t&=(k-\beta j)_t=(k-\beta j)k
  -e^{-u}\frac{\partial(K-\alpha f)}{\partial\nu_0}\\
  &\qquad\qquad
  +\frac{2\beta\rho_t j-\beta^2j\int_{\partial B}je^u(k-\beta j)ds_0}{2(\pi-\rho)}.
  \end{split}
\end{equation*}
where now $\frac{\partial(K-\alpha f)}{\partial\nu_0}(z_0,t_0)\le 0$. With \eqref{1.18} and
by definition of $\kappa$ we also have
\begin{equation*}
  (k-\beta j)k=(K-\alpha f)(K-\alpha f+\beta j)\ge\frac12(K-\alpha f)^2
\end{equation*}
at $(z_0,t_0)$. Thus, and recalling the definition of $C_2>0$, we can bound 
\begin{equation*}
  \begin{split}
 (K-\alpha f)_t&\ge\frac12\big((K-\alpha f)^2-C_2-C_2\int_{\partial B}e^u(k-\beta j)_+ds_0\big)
   \end{split}
\end{equation*}
at $(z_0,t_0)$. But by \eqref{3.6} and \eqref{3.8} we have
\begin{equation*}
  \begin{split}
    \int_{\partial B}&e^u(k-\beta j)_+ds_0
    \le\int_Be^{2u}(K-\alpha f)_-dz+2\int_{\partial B}e^u(k-\beta j)_-ds_0\\
    &\le 2m_0\sup_B(K-\alpha f)_- .
   \end{split}
\end{equation*}
Hence at $(z_0,t_0)$ there results the bound
\begin{equation*}
  \begin{split}
  (K-\alpha f)_t&\ge\frac12\big((K-\alpha f)^2-C_2+2m_0C_2(K-\alpha f)\big),
   \end{split}
\end{equation*}
where the term on the right again is positive by definition of $\kappa$. Thus, 
we obtain a contradiction as before, and the bound for $K$ follows.
With \eqref{3.7}, finally, we also obtain the asserted bound for $k$.
\end{proof}

\begin{corollary}\label{cor3.6}
For any $t<T$ there holds the bound $K,k\ge\kappa$, 
where $\kappa<0$ is as in Proposition \ref{prop3.5}.
\end{corollary}

Following \cite{Schwetlick-Struwe-2003}, from the preceding lower curvature bounds
and the bound for $u$ from above we are able to also deduce a uniform lower bound for $u=u(t)$.
 
\begin{proposition}\label{prop3.7}
For any $T>0$ there exists a constant $\ell_0>-\infty$ such that for any $t<T$ there holds 
\begin{equation*}
  \inf_Bu(t)\ge \ell_0.
\end{equation*}
\end{proposition}

\begin{proof}
As in \cite{Schwetlick-Struwe-2003}, Theorem A.2, we use Moser iteration on the 
equations \eqref{1.1} and \eqref{1.2} to prove this claim. Multiplying \eqref{1.1} with 
the testing function $e^{2u}u_-^{2p-1}\ge 0$ and integrating by parts, observing that 
$\nabla u\nabla u_-=-|\nabla u_-|^2\le 0$, for any $p\ge 1$ we obtain 
\begin{equation*}
  \begin{split}
  (2p-1)&\int_B|\nabla u_-|^2u_-^{2p-2}dz
  =\int_B\Delta uu_-^{2p-1}dz-\int_{\partial B}\frac{\partial u}{\partial\nu_0}u_-^{2p-1}ds_0\\
  &=-\int_BKu_-^{2p-1}e^{2u}dz
  -\int_{\partial B}\big(\frac{\partial u}{\partial\nu_0}+1\big)u_-^{2p-1}ds_0
  +\int_{\partial B}u_-^{2p-1}ds_0\\
  &=-\int_BKu_-^{2p-1}e^{2u}dz-\int_{\partial B}ku_-^{2p-1}e^uds_0
  +\int_{\partial B}u_-^{2p-1}ds_0\\
  &\le -\kappa e^{2M}\|u_-\|_{L^{2p-1}(B)}^{2p-1}+(1-\kappa e^M)\|u_-\|^{2p-1}_{L^{2p-1}(\partial B)},
  \end{split}
\end{equation*}
where $M=sup_{z\in B,t<T}u(z,t)\le u_1(T)$ with $u_1$ as defined in Proposition \ref{prop3.4}.
Thus, with constants $L',L\ge 2$ independent of $p$, for any $p\ge 1$ we have
\begin{equation*}
 \begin{split}
  \int_B|\nabla(u_-)^p|^2dz
  &\le\frac{p^2}{2p-1}\Big(-\kappa e^{2M}\|u_-\|_{L^{2p-1}(B)}^{2p-1}
  +(1-\kappa e^M)\|u_-\|^{2p-1}_{L^{2p-1}(\partial B)}\Big)\\
  &\le L'p\big(\|u_-\|_{L^{2p-1}(B)}^{2p-1}+\|u_-\|^{2p-1}_{L^{2p-1}(\partial B)}\big)\\
  &\le Lp\max\{\big(\|u_-\|_{L^{2p}(B)}^{2p}+\|u_-\|_{L^{2p}(\partial B)}^{2p}\big),1\},
 \end{split}
\end{equation*}
where we used H\"older's and Young's inequalities in the last step,
and there results the bound
\begin{equation}\label{3.9}
 \begin{split}
 \|(u_-)^p&\|^2_{H^1(B)}=\|\nabla(u_-)^p\|^2_{L^2(B)}+\|(u_-)^p\|_{L^2(B)}^2\\
 &\le (Lp+1)\max\{\big(\|(u_-)^p\|_{L^2(B)}^2+\|(u_-)^p\|_{L^2(\partial B)}^2\big),1\}.
 \end{split}
\end{equation}

By the divergence theorem, and using Young's inequality, for any $v\in H^1(B)$ 
and any $0<\varepsilon<1$ we have $v\in L^2(\partial B)$ with 
\begin{equation*}
 \begin{split}
 \int_{\partial B}v^2ds_0&=\int_{\partial B}v^2z\cdot\nu_0ds_0
 =\int_Bdiv(v^2z)dz= 2\int_Bv^2dz+2\int_Bv\,z\cdot\nabla vdz\\
 &\le 3\varepsilon^{-1}\|v\|^2_{L^2(B)}+\varepsilon\|\nabla v\|^2_{L^2(B)}.
 \end{split}
\end{equation*}
Applying this inequality with $v=(u_-)^p\in H^1(B)$ and $\varepsilon=(2Lp+2)^{-1}$, we can 
absorb the boundary integral on the right of \eqref{3.9} in the left hand side at the expense
of increasing the constant on the right. 
With the help of Sobolev's embedding $H^1(B)\hookrightarrow L^4(B)$, with 
constants $C>1$ independent of $p$ we then obtain the bound 
\begin{equation*}
 \begin{split}
   \|u_-\|^{2p}_{L^{4p}(B)}&=\|(u_-)^p\|_{L^4(B)}^2\le C\|(u_-)^p\|^2_{H^1(B)}
   \le C^2p^2\max\{\|u_-\|_{L^{2p}(B)}^{2p},1\}.
 \end{split}
\end{equation*}
Taking the $2p^{th}$ root and iterating, for $p_l=2^l$ we then inductively find 
the estimate
\begin{equation*}
 \begin{split}
 \|u_-\|_{L^{4p_l}(B)}&\le C^{1/p_l}p_l^{1/p_l}\max\{\|u_-\|_{L^{2p_l}(B)},1\}\\
 &\le\dots\le C\exp(\log2\cdot\sum_{j\le l}j/2^j)\max\{\|u_-\|_{L^{2}(B)},1\}
 \end{split}
\end{equation*}
for any $l\in\N$. Passing to the limit $l\to\infty$ we obtain the bound
\begin{equation*}
 \|u_-\|_{L^{\infty}(B)}\le C\max\{\|u_-\|_{L^{2}(B)},1\}\le C\max\{\|u\|_{L^{2}(B)},1\}.
\end{equation*}
But by \eqref{2.7} and the energy inequality \eqref{1.22} for any $R_0\le R\le R_1$
and the corresponding normalised function $v=v_R$ we have 
$\xi\in L^2([0,T],L^{\infty}(B))$
and thus $\|u\|_{L^{2}(B)}\le C\|v\|_{L^{2}(B)}+C$
for some $C=C(T)>0$. From Corollary \ref{cor2.4} we then obtain the claim.
\end{proof}

In view of Propositions \ref{prop3.4} and \ref{prop3.7} the flow equations 
\eqref{1.13} and \eqref{1.14} are uniformly parabolic. Global existence of the flow 
thus follows from the work of Brendle \cite{Brendle-2002b}.

\section{Concentration-Convergence}
\subsection{Alternative scenarios}
Recall that by \eqref{1.23} for a sequence $t_l\to\infty$ there holds
\begin{equation}\label{4.1}
  \begin{split}  
   \int_{B\times\{t_l\}}|\alpha f-K|^2d\mu_g
   +\int_{{\partial B}\times\{t_l\}}|\beta j-k|^2ds_g+\rho_t^2(t_l)\to 0\ (l\to\infty).
  \end{split}
\end{equation}
Similar to \cite{Struwe-2005}, for any $R_0\le R\le R_1$ we may replace $u=u(t)$ 
by its normalized representative $v=v_R(t)$ given by \eqref{2.2} with suitable 
$\Phi=\Phi(t)\in M$.
By geometric invariance of the curvature integrals from \eqref{4.1} then we have 
\begin{equation*}
 \begin{split}
   \int_{B\times\{t_l\}}|\alpha f_{\Phi}-K_{\Phi}|^2d\mu_h
   +\int_{{\partial B}\times\{t_l\}}|\beta j_{\Phi}-k_{\Phi}|^2ds_h
   +\rho_t^2(t_l)\to 0\ (l\to\infty),
 \end{split}
\end{equation*}
where $f_{\Phi}=f\circ\Phi$, $h=\Phi^*g=e^{2v}g_{\R^2}$, 
$K_{\Phi}=K_h=K\circ\Phi$, and so on. By Proposition \ref{prop3.3} and \eqref{1.15},
in addition we may assume that $\alpha_l=\alpha(t_l)\to\alpha>0$, $\beta_l=\beta(t_l)\to\beta>0$
as $l\to\infty$, where $\alpha=\beta^2$.

\begin{proposition}\label{prop4.1}
Let $t_l\to\infty$ be a sequence satisfying \eqref{4.1} above and for any fixed
$R_0\le R\le R_1$ also let $v_l=v(t_l)$ with the normalized representative $v$ of $u$ 
given by \eqref{2.2} for suitable $\Phi_l=\Phi(t_l)\in M$, $l\in\N$.
Then a subsequence $v_l\to v$ in $H^{3/2}(B)\cap H^1(\partial B)\cap C^0(\bar{B})$, 
and either
i) $\Phi_l\to\Phi$ for some $\Phi\in M$ and
$u_l=u(t_l)\to u$ in $H^{3/2}(B)\cap H^1(\partial B)\cap C^0(\bar{B})$,
where $u$ solves \eqref{1.3}, \eqref{1.4}, or
ii) $\Phi_l\to\Phi_0\equiv z_0$ weakly in $H^1(B)$ for some $z_0\in\partial B$ independent
of the choice of $R_0\le R\le R_1$, and for $R=R(z_0)$ we have convergence 
\begin{equation*}
  v_l\to v=v(z)=\log\Big(\frac{2R/\sqrt{\alpha f(z_0)}}{1+|Rz|^2}\Big).
\end{equation*}
\end{proposition}

\begin{proof}
Fixing any $R_0\le R\le R_1$, with the help of \eqref{1.1} and \eqref{1.2} for the 
corresponding $v_l$, with errors $\delta_l\to 0$ in $L^2(B,h(t_l))$ and $\varepsilon_l\to 0$ 
in $L^2(\partial B,h(t_l))$ as $l\to\infty$, from \eqref{4.1} we obtain
\begin{equation}\label{4.2}
  \begin{split}  
     -\Delta v_l=(\alpha_l f_{\Phi_l}+\delta_l)e^{2v_l}\ \hbox{ in } B,\quad
     \frac{\partial v_l}{\partial\nu_0}+1=(\beta_l j_{\Phi_l}+\varepsilon_l)e^{v_l}\ 
     \hbox{ on } \partial B.
  \end{split}
\end{equation}
Note that by Corollary \ref{cor2.4}
and \eqref{1.21} the sequence $(v_l)$ is bounded in $H^1(B)$, and we may assume that 
$v_l\rightharpoondown v$ weakly in $H^1(B)$ with $e^{2v_l}\to e^{2v}$ in $L^p(B)$,
$e^{v_l}\to e^{v}$ in $L^p(\partial B)$ for any $p<\infty$. Thus, $\delta_le^{2v_l}\to 0$
in $L^q(B)$, $\varepsilon_le^{v_l}\to 0$ in $L^q(\partial B)$ as $l\to\infty$ for any $q<2$.

Since $M$ is a bounded subset of $H^1(B;\R^2)$, in addition we may assume that $\Phi_l\to\Phi$
weakly in $H^1(B;\R^2)$ and almost everywhere, where $\Phi\in H^1(B;\R^2)$ 
either belongs to $M$ or is constant; in particular, as $l\to\infty$ we also have
$f_{\Phi_l}\to f_{\Phi}$ almost everywhere and therefore in $L^p(B)$ for any $p<\infty$ by 
boundedness of $f$. Likewise we have $j_{\Phi_l}\to j_{\Phi}$ almost everywhere on $\partial B$, and, since $j$ is bounded, $j_{\Phi_l}\to j_{\Phi}$ in $L^p(\partial B)$
for any $p<\infty$ as $l\to\infty$.

Testing equation \eqref{4.2} with $v_l$ we then obtain strong convergence 
$v_l\to v$ in $H^1(B)$ as $l\to\infty$, where 
\begin{equation}\label{4.3}
     -\Delta v=\alpha f_\Phi e^{2v}\ \hbox{ in } B,\quad
     \frac{\partial v}{\partial\nu_0}+1=\beta j_{\Phi}e^v\ \hbox{ on } \partial B.
\end{equation}
Similarly, now testing equation \eqref{4.2} with $e^{4v_l}$, 
we also find strong convergence $h_l=h(t_l)\to h$ in $H^1(B)$ as $l\to\infty$.

If $\Phi\in M$, in view of the identity $\alpha=\beta^2$, when replacing $v$ with the
function $w=v+\log\beta$ we find that there holds 
\begin{equation}\label{4.4}
     -\Delta w=f_\Phi e^{2w}\ \hbox{ in } B,\quad
     \frac{\partial w}{\partial\nu_0}+1=j_{\Phi}e^w\ \hbox{ on } \partial B.
\end{equation}
Thus, if $\Phi\in M$, the function $u=w\circ\Phi^{-1}+\frac12\log\det(d\Phi^{-1})$
is a solution of \eqref{1.3}, \eqref{1.4}.

On the other hand, if $\Phi\equiv z_0\in\partial B$ so that $f_\Phi\equiv f(z_0)$,
$j_{\Phi}\equiv j(z_0)$,
the metric $\tilde{h}=\alpha f(z_0)h=\alpha f(z_0)e^{2v}g_{\R^2}$ by \eqref{4.3} has constant
Gauss curvature $\tilde{K}=K_{\tilde{h}}\equiv 1$ and constant boundary geodesic curvature
\begin{equation*}
    \tilde{k}=k_{\tilde{h}}=\beta j(z_0)/\sqrt{\alpha f(z_0)}=j(z_0)/\sqrt{f(z_0)}>0.
\end{equation*}
By Mindig's theorem, the surface $(B,\tilde{h})$ is isometric to a coordinate ball 
$\tilde{B}$ on $S^2$. Centering $\tilde{B}$
around the North pole, we may represent $\tilde{B}=\Psi_{\tilde{R}}(B)$ for some 
$\tilde{R}>0$, and $\tilde{h}=\Psi_{\tilde{R}}^*g_{S^2}
=\big(\frac{2\tilde{R}}{1+|\tilde{R}z|^2}\big)^2g_{\R^2}$, where we recall that
$\Psi_{\tilde{R}}(z)=\Psi(\tilde{R}z)$ with the inverse $\Psi$ of stereographic projection
from the South pole. 

In addition, $\tilde{k}>0$ implies that $0<R=\tilde{R}<1$. In fact, we can precisely
determine $R$ in terms of $\tilde{k}$. Indeed,
from \eqref{4.3} and the Gauss-Bonnet identity \eqref{1.5} we obtain the equation 
\begin{equation*}
 \begin{split}
   2\pi&=\int_Bd\mu_{\tilde{h}}+\int_{\partial B}k_{\tilde{h}}ds_{\tilde{h}}
   =\int_B\big(\frac{2R}{1+|Rz|^2}\big)^2dz
    +k_{\tilde{h}}\int_{\partial B}\frac{2R}{1+|R|^2}ds_0\\
   &=2\pi\int_0^1\big(\frac{2R}{1+|Rr|^2}\big)^2r\,dr
    +2\pi k_{\tilde{h}}\frac{2R}{1+|R|^2}.
 \end{split}
\end{equation*}
Changing variables $s=1+(Rr)^2$, we find
\begin{equation*}
  \int_0^1\big(\frac{2R}{1+|Rr|^2}\big)^2r\,dr=2\int_1^{1+R^2}\frac{ds}{s^2}=2(1-\frac{1}{1+R^2})
  =\frac{2R^2}{1+|R|^2},
\end{equation*}
and we conclude that there holds
\begin{equation*}
    1=\frac{2R^2}{1+|R|^2}+k_{\tilde{h}}\frac{2R}{1+|R|^2}.
\end{equation*}
That is, we have
\begin{equation*}
    1=R^2+2Rk_{\tilde{h}},
\end{equation*}
and with $k_{\tilde{h}}=j(z_0)/\sqrt{f(z_0)}$ we obtain
\begin{equation}\label{4.5}
  R=\sqrt{1+k_{\tilde{h}}^2}-k_{\tilde{h}}=R(z_0),\ 
  k_{\tilde{h}}=\frac{1-R^2}{2R}=:k_R,
\end{equation} 
and $v(z)=\log(\frac{2R/\sqrt{\alpha f(z_0)}}{1+|Rz|^2})$ with $R=R(z_0)$, 
as claimed.

Note that for any other choice $R_0\le\hat{R}\le R_1$ of normalisation parameter
with corresponding $\hat{\Phi}_l$ we also have $\hat{\Phi}_l\to z_0$, and \eqref{2.6}
implies that also $v_{\hat{R}}(t_l)\to v$.

To finish the proof we now only need to show convergence $v_l\to v$ in the stated norm.
Recall from the above that we may assume that
$v_l\to v$ in $H^1(B)$ with $e^{2v_l}\to e^{2v}$ in $L^p(B)$,
$e^{v_l}\to e^{v}$ in $L^p(\partial B)$ for any $p<\infty$, while $\delta_le^{2v_l}\to 0$
in $L^q(B)$, $\varepsilon_le^{v_l}\to 0$ in $L^q(\partial B)$ as $l\to\infty$ for any $q<2$.

Our claim will follow easily once we can show that $v_l\in L^{\infty}(B)$ with a uniform bound
$\|v_l\|_{L^{\infty}(B)}\le C<\infty$. In order to obtain this bound, we
decompose $v_l=v_l^{(1)}+ v_l^{(2)}+v_l^{(3)}$, where the functions
$v_l^{(1)}$, $v_l^{(2)}$ are solutions of 
\begin{equation}\label{4.6}
 \begin{split}
  -\Delta v_l^{(1)}=(\alpha_lf_{\Phi_l}+\delta_l)e^{2v_l}+4c_l=:s_l^{(1)}\ \hbox{ in } B,\quad
   {\partial v_l^{(1)}}/{\partial\nu_0}=0 \hbox{ on } \partial B.
  \end{split}
\end{equation}
and
\begin{equation}\label{4.7}
 \begin{split}
  -\Delta v_l^{(2)}=0\ \hbox{ in } B,\quad
   {\partial v_l^{(2)}}/{\partial\nu_0}=(\beta_lj_{\Phi_l}+\varepsilon_l)e^{v_l}-1-2c_l=:s_l^{(2)}
   \hbox{ on } \partial B,
  \end{split}
\end{equation}
respectively, normalized so that $\int_Bv_l^{(1)}dz=\int_Bv_l^{(2)}dz=0$, and with constants
$c_l\in\R$ such that 
\begin{equation*}
  \int_Bs_l^{(1)}dz=\int_{\partial B}s_l^{(2)}ds_0=0.
\end{equation*}

Note that this choice is possible since by \eqref{1.5} we have
\begin{equation*}
  \int_B(\alpha_lf_{\Phi_l}+\delta_l)e^{2v_l}dz
  +\int_{\partial B}\big((\beta_lj_{\Phi_l}+\varepsilon_l)e^{v_l}-1\big)ds_0=0
\end{equation*}
so that we can set
\begin{equation*}
 \begin{split}
  4c_l\pi:=-\int_B(\alpha_lf_{\Phi_l}+\delta_l)e^{2v_l}dz
  =\int_{\partial B}\big((\beta_lj_{\Phi_l}+\varepsilon_l)e^{v_l}-1\big)ds_0.
  \end{split}
\end{equation*}
Our assumptions then imply that $|c_l|\le C<\infty$, uniformly in $l\in\N$.

Also letting $v_l^{(3)}(z)=c_l|z|^2+d_l$ for $z\in B$, with $d_l\in\R$ determined so that 
\begin{equation*}
   \int_Bv_ldz=\int_B(v_l^{(1)}+ v_l^{(2)}+v_l^{(3)})dz=\int_Bv_l^{(3)}dz,
\end{equation*}
the remainder $w_l=v_l^{(1)}+ v_l^{(2)}+v_l^{(3)}-v_l$ then satisfies
\begin{equation*}
 \begin{split}
  -\Delta w_l=0\ \hbox{ in } B,\quad
   \frac{\partial w_l}{\partial\nu_0}=0\hbox{ on } \partial B,
  \end{split}
\end{equation*}
and thus $w_l\equiv\dashint_Bw_ldz=0$. Again note that there holds $|d_l|\le C<\infty$,
uniformly in $l\in\N$.

With elliptic regularity theory, as explained in Lemma \ref{lemma4.2} below, 
from \eqref{4.6} and our above assumptions
for any $1<q<2$ we first obtain the uniform bound
\begin{equation*}
  \|v_l^{(1)}\|_{W^{2,q}(B)}\le C<\infty, \ \hbox{ for all }l\in\N.
\end{equation*}
Thus, by Sobolev's embedding $W^{2,q}(B)\hookrightarrow L^{\infty}(B)$ there also holds
the uniform bound $\|v_l^{(1)}\|_{L^{\infty}(B)}\le C<\infty$, $l\in\N$.

Similarly, for any $1<q<2<p<2q$ we find that $v_l^{(2)}\in W^{1,p}(B)\hookrightarrow L^{\infty}(B)$
with
\begin{equation*}
  \|v_l^{(2)}\|_{L^{\infty}(B)}\le C\|v_l^{(2)}\|_{W^{1,p}(B)}\le C<\infty,
  \ \hbox{ uniformly in }j\in\N.
\end{equation*}
Since the uniform bounds for the constants $c_l$ and $d_l$ also imply that 
\begin{equation*}
   \|v_l^{(3)}\|_{L^{\infty}(B)}\le C<\infty, \ \hbox{ uniformly in }j\in\N,
\end{equation*}
it then follows that $v_l \in L^{\infty}(B)$ with $\sup_{l\in\N}\|v_l\|_{L^{\infty}(B)}<\infty$,
as claimed.

Thus, we now have $L^2$-convergence $s_l^{(1)}\to s^{(1)}$ for some $s^{(1)}\in L^2(B)$,
as well as $s_l^{(2)}\to s^{(2)}$ for some $s^{(2)}\in L^2(\partial B)$, and the $L^2$-theory
for \eqref{4.6}, \eqref{4.7} yields convergence $v_l^{(1)}\to v^{(1)}$ in $H^2(B)$ for some
$v^{(1)}\in H^2(B)$ as well as convergence $v_l^{(2)}\to v^{(2)}$ in $H^{3/2}(B)$ for
some $v^{(2)}\in H^{3/2}(B)$ as $l\to\infty$.
Since clearly we also have $H^2$-convergence $v_l^{(3)}\to v^{(3)}(z)=c|z|^2+d$
for $c=\lim_{l\to\infty}c_l$, $d=\lim_{l\to\infty}d_l$, convergence $v_l\to v$ in
$H^{3/2}(B)\hookrightarrow H^1(\partial B)\cap L^{\infty}(B)$ follows, as claimed.

Finally, if $\Phi_l\to\Phi\in M$ in $H^1(B)$ and hence smoothly, since $M$ is
finite-dimensional, we also have $u_l=v_l\circ\Phi_l^{-1}+\frac12\log\det(d\Phi^{-1})\to u$
in $H^{3/2}\cap H^1(\partial B)\cap L^{\infty}(B)$, and the proof is complete.
\end{proof}

The regularity results used above do not seem standard. In the next lemma we 
therefore give a detailed proof.

\begin{lemma}\label{lemma4.2}
  i) For any $1<q<2$ and any $s^{(1)}\in L^q(B)$ with $\int_Bs^{(1)}dz=0$
  there is a unique solution $v^{(1)}\in W^{2,q}(B)$ of problem \eqref{4.6} with
  $\int_Bv^{(1)}\,dz=0$, and
  \begin{equation*}
    \|v^{(1)}\|_{W^{2,q}(B)}\le C\|s^{(1)}\|_{L^q(B)}.
  \end{equation*}

ii) For any $1<q<2$ and any $s^{(2)}\in L^q(\partial B)$ with $\int_{\partial B}s^{(2)}ds_0=0$
there is a unique solution $v^{(2)}\in H^1(B)$ of problem \eqref{4.7} with $\int_Bv^{(2)}\,dz=0$,
and there holds
$\nabla v^{(2)}\in L^p(B)$ for any $p<2q$, with
\begin{equation*}
  \|v^{(2)}\|_{W^{1,p}(B)}\le C\|s^{(2)}\|_{L^q(\partial B)}.
\end{equation*}
\end{lemma}

\begin{proof} i) Problem \eqref{4.6} has a unique weak solution
  \begin{equation*}
    v^{(1)}\in H:=\{v\in H^1(B);\;\int_Bv\,dz=0\},
  \end{equation*}
  characterized variationally as minimizer of the energy
  \begin{equation*}
    E^{(1)}(v)=\frac12\int_B\big(|\nabla v|^2-2s^{(1)}v\big)dz,\ v\in H.
  \end{equation*}
  Note that, by Sobolev's embedding $H^1(B)\hookrightarrow L^p(B)$ for $1\le p<\infty$, the
  functional $E^{(1)}$ is well-defined on $H$. Thus, from $E^{(1)}(v^{(1)})\le E^{(1)}(0)=0$
  and Poincar\'e's inequality we conclude $\|v^{(1)}\|_{H^1(B)}\le C\|s^{(1)}\|_{L^q(B)}$.

  Extending $v^{(1)}(x):=v^{(1)}(x/|x|^2)$ for $x\not\in B$, by conformal invariance of the
  Laplace operator and in view of the Neumann boundary condition
  ${\partial v_j^{(1)}}/{\partial\nu_0}=0$ on $\partial B$
  the extended function $v^{(1)}\in H^1_{loc}(\R^2)$ satsifies
  \begin{equation*}
    -\Delta v^{(1)}=:\tilde{s}^{(1)}\in L^q_{loc}(\R^2)\ \hbox{ with }
    \|v^{(1)}\|_{H^1(B_2(0))}+\|\tilde{s}^{(1)}\|_{L^q(B_2(0)}\le C\|s^{(1)}\|_{L^q(B}.
  \end{equation*}
  Letting $\varphi\in C^{\infty}_c(\R^2)$ with
  $0\le\varphi\le 1$ satisfy $\varphi(x)=1$ for $|x|\le 3/2$, $\varphi(x)=0$ for $|x|\ge 2$,
  the function $w^{(1)}=v^{(1)}\varphi\in H^1_0(B_2(0))$ solves the equation
  \begin{equation*}
    -\Delta w^{(1)}=\varphi\tilde{s}^{(1)}-2\nabla\varphi\nabla v^{(1)}-\Delta\varphi v^{(1)}
    \in L^q(B_2(0))
  \end{equation*}
  Thus, by the $L^q$-estimates for the Dirichlet problem, proved for instance in
  \cite{Giaquinta-1993}, we have
  $w^{(1)}\in W^{2,q}(B_2(0))$, and $v^{(1)}\in W^{2,q}(B)$ with 
  \begin{equation*}
    \|v^{(1)}\|_{W^{2,q}(B)}\le\|w^{(1)}\|_{W^{2,q}(B_2(0))}\le C\|\Delta w^{(1)}\|_{L^q(B_2(0))}
    \le C\|s^{(1)}\|_{L^q(B)},
  \end{equation*}
  as claimed.

  ii) Also problem \eqref{4.7} for any $s^{(2)}\in L^q(\partial B)$ with vanishing average
  has a unique weak solution $v^{(2)}\in H$,
  which may be characterized variationally as minimizer of the energy
  \begin{equation*}
    E^{(2)}(v)=\frac12\int_B|\nabla v|^2dz-\int_{\partial B}s^{(2)}v\,ds_0,
    \ v\in H,
  \end{equation*}
  where we now use the trace embedding $H^1(B)\hookrightarrow L^p(\partial B)$ for any
  $1\le p<\infty$.

  Letting $\Gamma$ be the fundamental solution of the Laplace operator satisfying
  \begin{equation*}
    -\Delta\Gamma(\cdot,z_0)=\pi\delta_{\{z=z_0\}}-1 \hbox{ in } B
  \end{equation*}
  with boundary condition
  \begin{equation*}
   \frac{\partial\Gamma(\cdot,z_0)}{\partial\nu_0}=0 \hbox{ on } \partial B
  \end{equation*}
  for every $z_0\in B$, we can represent $v^{(2)}$ as 
  \begin{equation*}
    \pi v^{(2)}(z_0)=\int_B\nabla v^{(2)}\nabla\Gamma(\cdot,z_0)dz
    =\int_{\partial B}\frac{\partial v^{(2)})}{\partial\nu_0}\Gamma(\cdot,z_0)ds_0
    =\int_{\partial B}s^{(2)}\,\Gamma(\cdot,z_0)ds_0.
  \end{equation*}
Differentiating in $z_0$, hence we find
  \begin{equation*}
    \nabla v^{(2)}(z_0)=\frac1\pi\int_{\partial B}s^{(2)}\,\nabla_{z_0}\Gamma(\cdot,z_0)ds_0
  \end{equation*}
  for every $z_0\in B$. 

On a half-space $\R^2_+=\{(x,y);\;y>0\}$ the corresponding fundamental solution is given by
  \begin{equation*}
    \Gamma_{\R^2_+}(z,z_0)=\frac12\big(\log(|z-z_0|)+\log(|z-\overline{z_0}|\big),
  \end{equation*}
  where $\overline{(x,y)}=(x,-y)$. Similarly, there holds
  \begin{equation*}
    |\nabla_{z_0}\Gamma(z,z_0)|\le C|z-z_0|^{-1}\in L^{(2,\infty)}(B)\subset L^p(B)
  \end{equation*}
for every $1\le p<2$,  where $L^{(2,\infty)}(B)$ is the space of functions weakly in $L^2$.
Thus, for $s^{(2)}\in L^1(\partial B)$ the function $t^{(2)}$ given by
  \begin{equation*}
    t^{(2)}(z_0):=\int_{\partial B}s^{(2)}\,\nabla_{z_0}\Gamma(\cdot,z_0)ds_0
  \end{equation*}
  belongs to $L^p(B)$ for every $1\le p<2$, and
  \begin{equation*}
   \begin{split}
    \|t^{(2)}\|^p_{L^p(B)}
     &=\int_B\big|\int_{\partial B}(s^{(2)}(z))^{1/p}\nabla_{z_0}\Gamma(z,z_0)(s^{(2)}(z))^{1-1/p}
       ds_0(z)\big|^pdz_0\\
     &\le\int_B\int_{\partial B}|s^{(2)}(z)||\nabla_{z_0}\Gamma(z,z_0)|^pds_0(z)dz_0\,
       \|s^{(2)}\|_{L^1(\partial B)}^{p-1}\\
    &\le\sup_{z\in B}\|\nabla_{z_0}\Gamma(z,z_0)\|_{L^p(B,dz_0)}^p\|s^{(2)}\|_{L^1(\partial B)}^p
   \end{split}
  \end{equation*}
  by H\"older's inequality and Fubini's theorem.
  (In fact, we conjecture that $t^{(2)}\in L^{(2,\infty)}(B)$;
  but we will not need this here.)

  Moreover, for $s^{(2)}\in L^2(\partial B)$, by estimates
  of Lions-Magenes \cite{Lions-Magenes-1972} and Sobolev's embedding, there holds
  $v^{(2)}\in H^{3/2}(B)\hookrightarrow W^{1,4}(B)$.
  To see this directly, let $w^{(2)}$ be a conjugate
  harmonic function, so that $u^{(2)}:=v^{(2)}+iw^{(2)}$ is analytic. In complex coordinates
  $z=re^{i\phi}$ then the Cauchy-Riemann equations give
  \begin{equation*}
    |w^{(2)}_{\phi}|=|v^{(2)}_r|=|s^{(2)}|\in L^2(\partial B).
  \end{equation*}
  Expanding $w^{(2)}$ in a Fourier series, moreover, we see that $w^{(2)}\in\dot{H}^{3/2}(B)$
  with
  \begin{equation*}
    \|w^{(2)}\|^2_{\dot{H}^{3/2}(B)}\le C\|w^{(2)}_{\phi}\|^2_{L^2(\partial B)},
  \end{equation*}
  where $\dot{H}^{3/2}(B)$ is the homogeneous Sobolev space. Hence, again by the
  Cauchy-Riemann equations, there also holds $v^{(2)}\in H^{3/2}(B)$
  with
  \begin{equation*}
  \|v^{(2)}\|^2_{H^{3/2}(B)}\le C\|s^{(2)}\|^2_{L^2(\partial B)}.
  \end{equation*}
  Interpolating, then for any $1<q<2$ and $s^{(2)}\in L^q(\partial B)$ we have
  $\nabla v^{(2)}\in L^p(B)$ for any $p<2q$, and the estimate holds, as claimed.
\end{proof}

We next show that \eqref{1.23} holds true unconditionally for {\it every} sequence
$t_l\to\infty$. 

\subsection{Evolution of curvature integrals}
Similar to the analysis in \cite{Struwe-2005}, we derive evolution equations for the 
curvature integrals appearing in \eqref{1.23}. First from \eqref{1.13} and \eqref{1.14}
we compute
\begin{equation*}
 \begin{split}
  \frac12&\frac{d}{dt}\Big(\int_B|\alpha f-K|^2d\mu_g+\int_{\partial B}|\beta j-k|^2ds_g
   +\rho_t^2\Big)\\
  &=\int_B\big((K-\alpha f)(K-\alpha f)_t+(\alpha f-K)^3\big)d\mu_g\\
  &\quad+\int_{\partial B}\big((k-\beta j)(k-\beta j)_t+\frac12(\beta j-k)^3\big)ds_g
  +\rho_t\rho_{tt}.
 \end{split}
\end{equation*}
From the evolution equations \eqref{3.2} - \eqref{3.5} for Gaussian and boundary geodesic
curvature we obtain the identities
\begin{equation*}
 \begin{split}
  &\int_B\big((K-\alpha f)(K-\alpha f)_t+(\alpha f-K)^3\big)d\mu_g\\
  &\quad=\int_B\big(2(K-\alpha f)^2K+(\alpha f-K)^3\big)d\mu_g
  +\frac12\int_{\partial B}\frac{\partial(K-\alpha f)^2}{\partial\nu_0}ds_0\\
  &\qquad-\int_B|\nabla(K-\alpha f)|^2dz-\frac{\alpha^2}{\rho}\Big(\int_Bf(K-\alpha f)d\mu_g\Big)^2
  -\frac{\alpha\rho_t}{\rho}\int_Bf(K-\alpha f)d\mu_g
 \end{split}
\end{equation*}
and
\begin{equation*}
 \begin{split}
  &\int_{\partial B}\big((k-\beta j)(k-\beta j)_t+\frac12(\beta j-k)^3\big)ds_g\\
  &\quad=\int_{\partial B}\big((k-\beta j)^2k+\frac12(\beta j-k)^3\big)ds_g
  -\int_{\partial B}\frac{\partial(K-\alpha f)}{\partial\nu_0}(k-\beta j)ds_0\\
  &\qquad-\frac{\beta^2}{2(\pi-\rho)}\Big(\int_{\partial B}j(k-\beta j)ds_g\Big)^2
  +\frac{\beta\rho_t}{\pi-\rho}\int_{\partial B}j(k-\beta j)ds_g,
 \end{split}
\end{equation*}
respectively. Moreover, from \eqref{1.15} we have $\rho_{tt}=2\beta_t/\beta-\alpha_t/\alpha$,
and with equations \eqref{3.4} and \eqref{3.5} we obtain
\begin{equation*}
 \begin{split}
  \rho_{tt}\rho_t=\frac{-2\rho_t^2}{\pi-\rho}
   +\frac{\beta\rho_t}{\pi-\rho}\int_{\partial B}j(k-\beta j)ds_g
   -\frac{\rho_t^2}{\rho}-\frac{\rho_t\alpha}{\rho}\int_Bf(K-\alpha f)d\mu_g.
 \end{split}
\end{equation*}
Adding, then we have
\begin{equation*}
 \begin{split}
  \frac12&\frac{d}{dt}\Big(\int_B|\alpha f-K|^2d\mu_g+\int_{\partial B}|\beta j-k|^2ds_g
   +\rho_t^2\Big)\\
  &=\int_B\big(2(K-\alpha f)^2K+(\alpha f-K)^3\big)d\mu_g
  +\frac12\int_{\partial B}\frac{\partial(K-\alpha f)^2}{\partial\nu_0}ds_0\\
  &\quad-\int_B|\nabla(K-\alpha f)|^2dz-\frac{\alpha^2}{\rho}\Big(\int_Bf(K-\alpha f)d\mu_g\Big)^2
  -\frac{2\alpha\rho_t}{\rho}\int_Bf(K-\alpha f)d\mu_g\\
  &\quad+\int_{\partial B}\big((k-\beta j)^2k+\frac12(\beta j-k)^3\big)ds_g
  -\int_{\partial B}\frac{\partial(K-\alpha f)}{\partial\nu_0}(k-\beta j)ds_0\\
  &\quad-\frac{\beta^2}{2(\pi-\rho)}\Big(\int_{\partial B}j(k-\beta j)ds_g\Big)^2
  +\frac{2\beta\rho_t}{\pi-\rho}\int_{\partial B}j(k-\beta j)ds_g
  -\frac{2\rho_t^2}{\pi-\rho}-\frac{\rho_t^2}{\rho}.
 \end{split}
\end{equation*}
Recalling \eqref{1.18}, we observe that the boundary integrals involving
$\frac{\partial(K-\alpha f)}{\partial\nu_0}$ miraculously cancel.
Moreover, writing 
\begin{equation*}
 \begin{split}
  &\alpha^2\Big(\int_Bf(K-\alpha f)d\mu_g\Big)^2
  +2\alpha\rho_t\int_Bf(K-\alpha f)d\mu_g+\rho_t^2\\
  &\quad=\Big(\rho_t+\alpha\int_Bf(K-\alpha f)d\mu_g\Big)^2
 \end{split}
\end{equation*}
as well as
\begin{equation*}
 \begin{split}
 \beta^2&\Big(\int_{\partial B}j(k-\beta j)ds\Big)^2
  -4\beta\rho_t\int_{\partial B}j(k-\beta j)ds_g       
   +4\rho_t^2\\
  &=\Big(2\rho_t-\beta\int_{\partial B}j(k-\beta j)ds_g\Big)^2,
 \end{split}
\end{equation*}
and expanding $K=K-\alpha f+\alpha f$, etc., we find the equation 
\begin{equation*}
 \begin{split}
  \frac12&\frac{d}{dt}\Big(\int_B|\alpha f-K|^2d\mu_g+\int_{\partial B}|\beta j-k|^2ds_g
   +\rho_t^2\Big)+\int_B|\nabla(K-\alpha f)|^2dz\\
  &=\int_B\big(2\alpha f(K-\alpha f)^2+(K-\alpha f)^3\big)d\mu_g
  -\frac{\Big(\rho_t+\alpha\int_Bf(K-\alpha f)d\mu_g\Big)^2}{\rho}\\
  &\quad+\int_{\partial B}\big(\beta j(k-\beta j)^2+\frac12(k-\beta j)^3\big)ds_g
  -\frac{\Big(2\rho_t-\beta\int_{\partial B}j(k-\beta j)ds_g\Big)^2}{2(\pi-\rho)}.
 \end{split}
\end{equation*}

Similar to \cite{Struwe-2005}, proof of Lemma 3.4, we may replace $u$ by a normalized 
representative $v$ given by \eqref{2.2} and use geometric 
invariance of the curvature integrals to express the latter in the form 
\begin{equation}\label{4.8}
 \begin{split}
  \frac12&\frac{d}{dt}\Big(\int_B|\alpha f_{\Phi}-K_{\Phi}|^2d\mu_h
  +\int_{\partial B}|\beta j_{\Phi}-k_{\Phi}|^2ds_h
   +\rho_t^2\Big)\\
   &=-\int_B|\nabla(\alpha f_{\bar{\Phi}}-K_{\bar{\Phi}})|^2dz
     +\int_B\big(2\alpha f_{\Phi}(\alpha f_{\bar{\Phi}}-K_{\bar{\Phi}})^2
  +(K_{\bar{\Phi}}-\alpha f_{\bar{\Phi}})^3\big)d\mu_h\\
  &\quad+\int_{\partial B}\big(\beta j_{\Phi}(k_{\Phi}-\beta j_{\Phi})^2
  +\frac12(k_{\Phi}-\beta j_{\Phi})^3\big)ds_h\\
   &\quad-\frac{\Big(\rho_t+\alpha\int_Bf_{\Phi}(K_{\Phi}-\alpha f_{\Phi})d\mu_h\Big)^2}{\rho}
   -\frac{\Big(2\rho_t-\beta\int_{\partial B}j_{\Phi}(k_{\Phi}-\beta j_{\Phi})ds_h\Big)^2}{2(\pi-\rho)},
 \end{split}
\end{equation}
where $f_{\Phi}=f\circ\Phi$, $h=\Phi^*g=e^{2v}g_{\R^2}$, 
$K_{\Phi}=K_{\Phi^*g}=K_h=K\circ\Phi$, and so on.

\subsection{Unconditional convergence}
From Proposition \ref{prop4.1} we can show that \eqref{4.1} holds true for {\it every}
sequence $t_l\to\infty$. For convenience, we let
\begin{equation*}
 \begin{split}
  F&=F(t):=\int_B|\alpha f-K|^2d\mu_g+\int_{\partial B}|\beta j-k|^2ds_g+\rho_t^2\\
  &=\int_B|\alpha f_{\Phi}-K_{\Phi}|^2d\mu_h
   +\int_{\partial B}|\beta j_{\Phi}-k_{\Phi}|^2ds_h+\rho_t^2
 \end{split}
\end{equation*}
and we set 
\begin{equation*}
   G=G(t):=\int_B|\nabla(K-\alpha f)|^2dz
      =\int_B|\nabla(\alpha f_{\bar{\Phi}}-K_{\bar{\Phi}})|^2dz.
\end{equation*}

\begin{lemma}\label{lemma4.3}
There holds $F(t)\to 0$ as $t\to\infty$.
\end{lemma}

\begin{proof}
We argue as in \cite{Struwe-2005}, proof of Lemma 3.4. Given $0<\varepsilon_0<1$, by \eqref{1.23}
there exist arbitrarily large times $t_0$ such that $F(t_0)<\varepsilon_0$. For any such $t_0$,
we choose a maximal time $t_1\ge t_0$, $t_1\le\infty$, such that 
\begin{equation*}
  \sup_{t_0\le t<t_1}F(t)<2\varepsilon_0.
\end{equation*}
By Proposition \ref{prop4.1}, if $0<\varepsilon_0<1$ is sufficiently small, the metrics
$h=h(t)=\Phi^*g$ for $t_0\le t<t_1$ will be uniformly equivalent to the Euclidean metric.
In particular, the standard Sobolev embeddings hold with uniform constants in the metrics 
$h$ for $t_0\le t<t_1$, and with uniform constants $C>0$ we can bound 
\begin{equation*}
  \begin{split}  
    \big|\int_B(K_{\Phi}&-\alpha f_{\Phi})^3d\mu_h\big|
    \le C\|K_{\Phi}-\alpha f_{\Phi}\|^3_{L^3(B,g_{\R^2})}\\
    &\le C\|K_{\Phi}-\alpha f_{\Phi}\|_{L^2(B,g_{\R^2})}
    \|K_{\Phi}-\alpha f_{\Phi}\|^2_{L^4(B,g_{\R^2})}\\
    &\le C\sqrt{\varepsilon_0}\|K_{\Phi}-\alpha f_{\Phi}\|^2_{H^1(B,g_{\R^2})}
    \le C\sqrt{\varepsilon_0}(F+G).
  \end{split}
\end{equation*}
Similarly, again using that 
$k_{\Phi}-\beta j_{\Phi}=K_{\Phi}-\alpha f_{\Phi}$ on $\partial B$, we can bound 
\begin{equation*}
  \begin{split}  
    \big|\int_{\partial B}(k_{\Phi}&-\beta j_{\Phi})^3ds_h\big|
    \le C\|k_{\Phi}-\beta j_{\Phi}\|_{L^2(\partial B,g_{\R^2})}
    \|K_{\Phi}-\alpha f_{\Phi}\|^2_{L^4(\partial B,g_{\R^2})}\\
    &\le C\sqrt{\varepsilon_0}\|K_{\Phi}-\alpha f_{\Phi}\|^2_{H^1(B,g_{\R^2})}
    \le C\sqrt{\varepsilon_0}(F+G).
  \end{split}
\end{equation*}

For sufficiently small $\varepsilon_0>0$ from \eqref{4.8} we then obtain the
differential inequality 
\begin{equation*}
  \frac{d}{dt}F\le CF \ \hbox{ for } t_0\le t<t_1
\end{equation*}
with a uniform constant $C>0$, and there results the bound 
\begin{equation*}
  \sup_{t_0\le t<t_1}F(t)<F(t_0)+C\int_{t_0}^{\infty}F(t)dt.
\end{equation*}
But by the energy bound \eqref{1.22}, the right hand side is smaller than $2\varepsilon_0$
for sufficiently large $t_0$, and we have $t_1=\infty$. 
Since $0<\varepsilon_0<1$ may be chosen arbitrarily small, the claim follows.
\end{proof}

\subsection{Concentration}
Now assume that there is no solution of \eqref{1.3}, \eqref{1.4}, which we will do from now on.
Then as $t_l\to\infty$, necessarily 
for a suitable subsequence $l\to\infty$ the metrics evolving under the flow
\eqref{1.13} - \eqref{1.15} concentrate at a point $z_0\in\partial B$, whereas
the normalized metrics $h_l=h(t_l)$ as well as the normalized functions $v_l=v(t_l)$
nicely converge to a spherical limit metric, as shown in Proposition \ref{prop4.1}.

For convenience we recall the details in the following result.

\begin{corollary}\label{cor4.4}
Suppose that there is no solution of \eqref{1.3}, \eqref{1.4}. Then for any sequence 
$t_l\to\infty$ there is a subsequence $l\to\infty$ such that for any $R_0\le R\le R_1$ 
there holds $\Phi_l=\Phi(t_l)\to\Phi_{\infty}\equiv z_0$ weakly in $H^1$
for some $z_0\in\partial B$. In addition, we may assume that $\alpha_l=\alpha(t_l)\to\alpha$,
$\beta_l=\beta(t_l)\to\beta$ for some $\alpha\in[\alpha_0,\alpha_1]$, $\beta\in[\beta_0,\beta_1]$
such that $\alpha=\beta^2$. Finally, fixing $R=R(z_0)$ as given by \eqref{2.4},
\eqref{4.5} we have $v_l=v(t_l)\to v_{\infty}$,
$h_l=h(t_l)\to h_{\infty}=e^{2v_{\infty}}g_{\R^2}$ in
$H^{3/2}(B)\cap H^1(\partial B)\cap L^{\infty}(B)$, where
$\alpha f(z_0)h_{\infty}=\Psi_R^*g_{S^2}$.
\end{corollary}

It is then natural to lift the flow to the sphere $S^2\subset\R^3$. Denote as 
$e_i,X_i$, $1\le i\le 3$, the standard basis and the restrictions of the ambient
coordinate functions in $\R^3$ to $S^2$, respectively. 
For $X=(X_1,X_2,X_3)$ we also let $Z=\pi_{\R^2}(X)=(X_1,X_2)$. 
Let $S^2_R=\Psi_R(B)$ be the spherical cap corresponding to the limit metric $h_{\infty}$ in
Corollary \ref{cor4.4}.
With the help of scaled stereographic projection $\pi_R=\Psi_R^{-1}$ we lift the metrics
$g=g(t)$ and $h=h(t)$ to $S^2_R$, as follows. Set
\begin{equation*}
  \bar{g}=\pi_R^*g=\pi_R^*\big(e^{2u}g_{\R^2}\big)=e^{2\bar{u}}g_{S^2}
\end{equation*}
and define the normalized companion metric 
\begin{equation*}
  \bar{h}=\pi_R^*h=\pi_R^*\big(e^{2v}g_{\R^2}\big)
  =e^{2\bar{v}}g_{S^2}\to (\alpha f(z_0))^{-1}g_{S^2}
\end{equation*}
with a family of functions $\bar{v}=\bar{v}(t)$ converging to the constant
$\bar{v}_{\infty}=-\frac12\log(\alpha f(z_0))$
in $H^{3/2}(S^2_R)\cap H^1(\partial S^2_R)\cap L^{\infty}(S^2_R)$ as $t=t_l\to\infty$ suitably.
Note that with the conformal diffeomorphism
$\bar{\Phi}=\bar{\Phi}(t)=\Psi_R\circ\Phi(t)\circ\pi_R\colon S^2_R\to S^2_R $
we have
\begin{equation*}
  \bar{h}=\bar{\Phi}^*\bar{g}
\end{equation*}
and thus
\begin{equation*}
  \bar{v}=\bar{u}\circ\bar{\Phi}+\frac12\log(\sqrt{det(d\bar{\Phi}^td\bar{\Phi}}))
  \hbox{ in } S^2_R,\ 
  \bar{v}=\bar{u}\circ\bar{\Phi}+\log(|\frac{\partial\bar{\Phi}}{\partial\tau}|)
  \hbox{ on } \partial S^2_R
\end{equation*}
for each $t>0$, where $\tau$ again is the oriented unit tangent vector along $\partial S^2_R$.
Moreover, with our short hand notation $K_{\Phi}=K_h=K_{\Phi^*g}=K_g\circ\Phi$, and letting
$\Phi_R=\Phi\circ\pi_R$, we have
\begin{equation*}
  K_{\bar{h}}=K_g\circ\Phi_R=K_{\Phi_R}.
\end{equation*}

By Corollary \ref{cor4.4} then for $t_l\to\infty$, with error
$o(1)\to 0$ as $l\to\infty$, for suitable $l\to\infty$ at $t=t_l$ we can write
\begin{equation*}
  \begin{split}
  \int_B\alpha f_{\Phi}(K_{\Phi}-\alpha f_{\Phi})^2&d\mu_h
   =\int_{S^2_R}(K_{\Phi_R}-\alpha f_{\Phi_R})^2d\mu_{g_{S^2}}+o(1)F\\
 \end{split}
\end{equation*}
and
\begin{equation*}
 \begin{split}
   \alpha\int_Bf_{\Phi}(K_{\Phi}-\alpha f_{\Phi})d\mu_h
   =\int_{S^2_R}(K_{\Phi_R}-\alpha f_{\Phi_R})d\mu_{g_{S^2}}+o(1)F^{1/2}.
 \end{split}
\end{equation*}

Recalling from \eqref{4.5} that $\beta j(z_0)/\sqrt{\alpha f(z_0)}=j(z_0)/\sqrt{f(z_0)}=k_R$
we likewise see that
\begin{equation*}
   \int_{\partial B}\beta j_{\Phi}(k_{\Phi}-\beta j_{\Phi})^2ds_h
   =k_R\int_{\partial S^2_R}(k_{\Phi_R}-\beta j_{\Phi_R})^2ds_{g_{S^2}}+o(1)F
\end{equation*}
as well as
\begin{equation*}
  \beta\int_{\partial B}j_{\Phi}(k_{\Phi}-\beta j_{\Phi})ds_h
  =k_R\int_{\partial S^2_R}(k_{\Phi_R}-\beta j_{\Phi_R})ds_{g_{S^2}}+o(1)F^{1/2}.
\end{equation*}
Moreover, both cubic terms from \eqref{4.8} can be bounded
\begin{equation*}
  \big|\int_B(K_{\Phi}-\alpha f_{\Phi})^3d\mu_h\big|+
  \big|\int_{\partial B}(k_{\Phi}-\beta j_{\Phi})^3ds_h\big|= o(1)(F+G).
\end{equation*}

In view of \eqref{4.8} hence we have
\begin{equation}\label{4.9}
 \begin{split}
   \frac12&\frac{dF}{dt}+G+\frac{\big(\rho_t
   +\int_{S^2_R}(K_{\Phi_R}-\alpha f_{\Phi_R})d\mu_{g_{S^2}}\big)^2}{\rho}\\
   &\quad+\frac{\big(2\rho_t
  -k_R\int_{\partial S^2_R}(k_{\Phi_R}-\beta j_{\Phi_R})ds_{g_{S^2}}\big)^2}
  {2(\pi-\rho)}+o(1)(F+G)\\
  &=2\int_{S^2_R}(K_{\Phi_R}-\alpha f_{\Phi_R})^2d\mu_{g_{S^2}}
  +k_R\int_{\partial S^2_R}(k_{\Phi_R}-\beta j_{\Phi_R})^2ds_{g_{S^2}}.
 \end{split}
\end{equation}

For clarity, in the following we also use indices to distinguish the outward
normal $\nu_{S^2_R}$ along $\partial S^2_R$ from the outward normal 
$\nu_B=\nu_0$ along $\partial B$.

\section{Finite-dimensional dynamics}
Similar to \cite{Struwe-2005} we can show that 
the flow equations \eqref{1.13} - \eqref{1.15} are shadowed by a system of ordinary
differential equations moving the center of mass of the evolving metrics in direction 
of a suitable combination of the gradients of the prescribed curvature
functions $f$ and $j$. Assuming that the metrics $g(t)$ for $t=t_l\to\infty$
concentrate at a point $z_0\in\partial B$ in the sense described in Corollary \ref{cor4.4},
for any sufficiently large $l\in\N$ we first establish the relevant equations and estimates
only for times $t\ge t_l$ where the center of mass is at a distance from $z_0$ 
comparable to the distance at the time $t_l$. As we summarise our results in 
Lemma \ref{lemma5.10}, however, we are able to assert that under the assumptions of 
Theorem \ref{thm1.1}  this condition will hold true for {\it all} times $t>t_l$ when $l$ is 
sufficiently large. In particular, we have unconditional convergence $a(t)\to z_0$ as 
$t\to\infty$,

Before entering into details we recall that the concentration point $z_0$ determines 
the parameter $R=R(z_0)$ and that we have $|Z|\equiv\frac{2R}{1+R^2}=:r$, 
$X_3\equiv\frac{1-R^2}{1+R^2}=:\sigma$ on $\partial S^2_R$ with boundary 
curvature $k_R=\frac{1-R^2}{2R}=\frac{\sigma}{r}$ given by \eqref{4.5}. 
Moreover, the outer normal along $\partial S^2_R$ is given by
$\nu_{S^2_R}=(\sigma z,-r)$.
 
The number $R$ also determines $\rho=\lim_{l\to\infty}\rho(t_l)$. Indeed,  
with error $o(1)\to 0$ as $l\to\infty$, from \eqref{1.17} we have
\begin{equation*}
  \frac{\alpha f(z_0)+o(1)}{\rho}\int_{S^2_R}d\mu_{\bar{h}}
  =\frac{\alpha}{\rho}\int_Bf_{\Phi}e^{2v}dz
  =\frac{\alpha}{\rho}\int_Bfe^{2u}dz=2
\end{equation*}
and 
\begin{equation*}
  \frac{\beta j(z_0)+o(1)}{\pi-\rho}\int_{\partial S^2_R}d\mu_{\bar{h}}
  =\frac{\beta}{\pi-\rho}\int_{\partial B}j_{\Phi}e^vds_0
  =\frac{\beta}{\pi-\rho}\int_{\partial B}je^uds_0=2.
\end{equation*}
Thus, in particular, with error $o(1)\to 0$ as $t\to\infty$ there holds
\begin{equation*}
 \begin{split}  
   2\rho+o(1)&=\alpha f(z_0)\int_{S^2_R}d\mu_{\bar{h}}+o(1)=\int_{S^2_R}d\mu_{g_{S^2}}\\
   &=2\pi\int_0^1\Big(\frac{2R}{1+|Rr|^2}\Big)^2r\,dr=\frac{4\pi R^2}{1+R^2},
 \end{split}
\end{equation*}
and
\begin{equation*}
  \rho=\frac{2\pi R^2}{1+R^2}+o(1),\ \pi-\rho=\pi(1-\frac{2R^2}{1+R^2})+o(1)=\pi\sigma+o(1).
\end{equation*}

\subsection{Expansion  in terms of Steklov eigenfunctions}
Let $w=K_{\Phi}-\alpha f_{\Phi}$ and set $w_R=w\circ\pi_R$. We expand $w_R$
in terms of a sequence $\varphi_i$, $i\in\N_0$, of
Steklov eigenfunctions of the Laplacean on $S^2_R$, satisfying the equations
\begin{equation}\label{5.1}
  -\Delta_{g_{S^2}}\varphi_i=2\lambda_i\varphi_i\ \hbox{ in }S^2_R,\
  \frac{\partial\varphi_i}{\partial\nu_{S^2_R}}=\lambda_ik_R\varphi_i
  \hbox{ on }\partial S^2_R,
\end{equation}
with eigenvalues $0\le\lambda_i\le\lambda_{i+1}$, $i\in\N_0$,
and orthonormal with respect to the measure $\hat{\mu}_R$ defined by
\begin{equation*}
  \int_{S^2_R}\varphi\,d\hat{\mu}_R=2\int_{S^2_R}\varphi\,d\mu_{g_{S^2}}
  +k_R\int_{\partial S^2_R}\varphi\,ds_{g_{S^2}}.
\end{equation*}
Note that $\lambda_i\to\infty$ as $i\to\infty$, and that there holds
\begin{equation*}
  \int_{S^2_R}\nabla\varphi_i\cdot\nabla\varphi_j\,d\mu_{g_{S^2}}
  =\lambda_i(\varphi_i,\varphi_j)_{L^2(S^2_R,\hat{\mu}_R)},\ i,j\in\N_0.
\end{equation*}

Recall that we have the mini-max characterization
\begin{equation}\label{5.2}
  \lambda_i=\inf_{X\subset H^1(S^2_R);\, dim X\ge i+1}\ \sup_{v\in X\setminus\{0\}}
  \frac{\|\nabla v\|^2_{L^2(S^2_R)}}{\|v\|^2_{L^2(S^2_R,\hat{\mu}_R)}},\ i\in\N_0,
\end{equation}
of the eigenvalues $\lambda_i$, where $L^2(S^2_R)=L^2(S^2_R,\mu_{g_{S^2}})$
In particular, we have $\lambda_0=0$ with $\varphi_0=const$; moreover, below we shall see that
the coordinate functions $X_{1,2}$ both are Steklov eigenfunctions with eigenvalues
$\lambda_1=\lambda_2=1$ and that we have the spectral gap $\lambda_i>1$ for $i\ge 3$.
However, we first focus on the constant component
\begin{equation*}
  \bar{w}_R=\dashint_{S^2_R}w_R\,d\hat{\mu}_R=\Big(2\int_{S^2_R}w_R\,d\mu_{g_{S^2}}
  +k_R\int_{\partial S^2_R}w_R\,ds_{g_{S^2}}\Big)/\int_{S^2_R}d\hat{\mu}_R,
\end{equation*}
with
\begin{equation*}
  \int_{S^2_R}d\hat{\mu}_R=2\int_{S^2_R}\,d\mu_{g_{S^2}}
  +k_R\int_{\partial S^2_R}\,ds_{g_{S^2}}=4\rho+2\pi\sigma+o(1)=2(\pi+\rho)+o(1).
\end{equation*}

Also denote as
\begin{equation*}
  \hat{w}_R=\dashint_{S^2_R}w_R\,d\mu_{g_{S^2}},\ 
  \tilde{w}_R=\dashint_{\partial S^2_R}w_R\,ds_{g_{S^2}} 
\end{equation*}
the averages of $w_R$ on $S^2_R$ and $\partial S^2_R$, respectively,
satisfying with error $o(1)\to 0$ as $t=t_l\to\infty$
\begin{equation}\label{5.3}
  \frac{4\rho\hat{w}_R+2\pi\sigma\tilde{w}_R}{4\rho+2\pi\sigma}=\bar{w}_R+o(1)F^{1/2}.
\end{equation}

\begin{lemma}\label{lemma5.1}
With error $o(1)\to 0$ as $t=t_l\to\infty$ there holds
\begin{equation*}
 \begin{split}
   \frac12\frac{dF}{dt}&+G+\frac{\pi+\rho}{\rho(\pi-\rho)}\Big(\rho_t
            +2\frac{\rho(\pi-\rho)}{\pi+\rho}(\hat{w}_R-\tilde{w}_R)\Big)^2
   =\int_{S^2_R}|w_R-\bar{w}_R|^2d\hat{\mu}_R+o(1)(F+G).
 \end{split}
\end{equation*}
\end{lemma}

\begin{proof}
The right hand side of \eqref{4.9} may be written as
\begin{equation*}
 \begin{split}
   \int_{S^2_R}|w_R|^2d\hat{\mu}_R
   =\int_{S^2_R}|w_R-\bar{w}_R|^2d\hat{\mu}_R+\int_{S^2_R}|\bar{w}_R|^2d\hat{\mu}_R.
 \end{split}
\end{equation*}

Moreover, we have
\begin{equation*}
 \begin{split}
   &\frac{\big(\rho_t+\int_{S^2_R}w_Rd\mu_{g_{S^2}}\big)^2}{\rho}
   =\frac{\big(\rho_t+2\rho\hat{w}_R\big)^2}{\rho}+o(1)F\\
   &\quad=\frac{\rho_t^2}{\rho}+4\rho_t\hat{w}_R+4\rho|\hat{w}_R|^2+o(1)F
 \end{split}
\end{equation*}
as well as
\begin{equation*}
  \begin{split}
  &\frac{\big(2\rho_t
  -k_R\int_{\partial S^2_R}w_Rds_{g_{S^2}}\big)^2}{2(\pi-\rho)}
  =\frac{2\big(\rho_t-\pi\sigma\tilde{w}_R\big)^2}{\pi-\rho}+o(1)F)\\
   &\quad=\frac{2\rho_t^2}{\pi-\rho}-4\rho_t\tilde{w}_R+2\pi\sigma|\tilde{w}_R|^2+o(1)F,
 \end{split}
\end{equation*}
so that
\begin{equation*}
 \begin{split}
   I:&=\frac{\big(\rho_t+\int_{S^2_R}w_Rd\mu_{g_{S^2}}\big)^2}{\rho}
     +\frac{\big(2\rho_t
  -k_R\int_{\partial S^2_R}w_Rds_{g_{S^2}}\big)^2}{2(\pi-\rho)}\\
   &\quad=\frac{\rho_t^2}{\rho}+4\rho_t\hat{w}_R+4\rho|\hat{w}_R|^2
   +\frac{2\rho_t^2}{\pi-\rho}-4\rho_t\tilde{w}_R+2\pi\sigma|\tilde{w}_R|^2+o(1)F(t)\\
   &\quad=\frac{\pi+\rho}{\rho(\pi-\rho)}\rho_t^2
   +4\rho_t(\hat{w}_R-\tilde{w}_R)+4\rho|\hat{w}_R|^2+2\pi\sigma|\tilde{w}_R|^2+o(1)F(t).
 \end{split}
\end{equation*}

Recalling that $\pi\sigma=\pi-\rho+o(1)$, and letting $\lambda=\frac{2\rho}{\pi+\rho}$,
$1-\lambda=\frac{\pi\sigma}{\pi+\rho}+o(1)$ for brevity, using \eqref{5.3}
we can bound
\begin{equation*}
 \begin{split}
   &(4\rho+2\pi\sigma)^{-1}\big(4\rho|\hat{w}_R|^2+2\pi\sigma|\tilde{w}_R|^2
     -(4\rho+2\pi\sigma)|\bar{w}_R|^2\big)+o(1)F\\
   &\quad=\lambda|\hat{w}_R|^2
     +(1-\lambda)|\tilde{w}_R|^2-|\lambda\hat{w}_R+(1-\lambda)\tilde{w}_R|^2+o(1)F\\
   &\quad=\lambda(1-\lambda)(|\hat{w}_R|^2+|\tilde{w}_R|^2-2\hat{w}_R\tilde{w}_R)
     =\lambda(1-\lambda)|\hat{w}_R-\tilde{w}_R|^2+o(1)F.
 \end{split}
\end{equation*}
But
\begin{equation*}
  (4\rho+2\pi\sigma)\lambda(1-\lambda)
  =(4\rho+2\pi\sigma)\frac{2\rho\pi\sigma}{(\pi+\rho)^2}+o(1)
  =4\frac{\rho(\pi-\rho)}{\pi+\rho}+o(1);
\end{equation*}
thus, we obtain
\begin{equation*}
 \begin{split}
   I&-\int_{S^2_R}|\bar{w}_R|^2d\hat{\mu}_R=I-(4\rho+2\pi\sigma)|\bar{w}_R|^2+o(1)F\\
    &=\frac{\pi+\rho}{\rho(\pi-\rho)}\rho_t^2+4\rho_t(\hat{w}_R-\tilde{w}_R)
      +4\frac{\rho(\pi-\rho)}{\pi+\rho}|\hat{w}_R-\tilde{w}_R|^2+o(1)F\\
    &=\frac{\pi+\rho}{\rho(\pi-\rho)}\Big(\rho_t
      +2\frac{\rho(\pi-\rho)}{\pi+\rho}(\hat{w}_R-\tilde{w}_R)\Big)^2+o(1)F,
 \end{split}
\end{equation*}
and the claim follows from \eqref{4.9}.
\end{proof}

Note that for $1\le i\le 2$ there holds 
\begin{equation}\label{5.4}
  -\Delta_{S^2}X_i=2X_i\ \hbox{ in }S^2_R,\
  \frac{\partial X_i}{\partial\nu_{S^2_R}}=\nu_{S^2_R}\cdot e_i
  =\sigma X_i/r\ \hbox{ on }\partial S^2_R.
\end{equation}

\begin{lemma}\label{lemma5.2}
The coordinate functions $X_{1,2}$ both are Steklov eigenfunctions with eigenvalues
$\lambda_1=\lambda_2=1$ and we have the spectral gap $\lambda_i>1$ for $i\ge 3$.
\end{lemma}

\begin{proof}
The first assertion is immediate from \eqref{5.4}.
To see the second part of the claim, suppose by contradiction that for some $0<R\le 1$
with corresponding $0\le\sigma<1$ there holds $\lambda_1\le 1$ with a corresponding 
normalized eigenfunction $\varphi_1$ satisfying
\begin{equation*}
  0=\int_{S^2_R}\varphi_1d\hat{\mu}_R=\int_{S^2_R}X_1\varphi_1d\hat{\mu}_R
  =\int_{S^2_R}X_2\varphi_1d\hat{\mu}_R.
\end{equation*}
In view of the mini-max characterization \eqref{5.2} of $\lambda_1$ the latter depends
continuously on $R$ or $\sigma$. Thus we may assume that $0\le\sigma<1$ is minimal with
this property.

We claim that $\sigma>0$. Indeed, suppose that $\sigma=0$. Then $R=1$ and $S^2_R=S^2_+$ is the
upper half-sphere with $k_R=0$. In view of \eqref{5.1} we can extend $\varphi_1$ 
by even reflection in $\partial S^2_+$ to a solution of the equation \eqref{5.1} 
with $0<\lambda_1\le 1$ on all of $S^2$. 
But any such solution is a linear combination of the functions
$X_i$, $1\le i\le 3$. In addition, even symmetry gives
$\int_{S^2}\varphi_1X_3\,d\mu_{G_{S^2}}=0$. Hence $\varphi_1$ is a linear combination
only of the functions $X_1$ and $X_2$, which is impossible by orthogonality.

Thus, $0<\sigma<1$ and $\lambda_1=1$ (by minimality of $\sigma$). The $1$-homogeneous extension 
of $\varphi_1$ to the cone over $S^2_R$ in $\R^3$, given by 
$\bar{\varphi}_1(sX)=s\varphi_1(X)$ for $s>0$, $X\in S^2_R$, then is harmonic.
But expanding $\varphi_1$ in terms of spherical harmonics (the eigenfunctions of 
$\Delta_{S^2}$) we see that only the contributions from $X_i$, $1\le i\le 3$, have a 
$1$-homogeneous harmonic extension, and $\varphi_1=\gamma X_3$ for some 
$\gamma\neq 0$. But 
$\partial X_3/\partial\nu_{S^2_R}=-r=-\frac{r}{\sigma}X_3$ on $\partial S^2_R$.
A contradiction follows, proving our claim.
\end{proof}

Expand $w_R-\bar{w}_R=\sum_{i\ge 1}\nu_i\varphi_i$, with 
$\varphi_i=X_i/\|X_i\|_{L^2(S^2_R,\hat{\mu}_R)}$ for $i=1,2$, and split 
\begin{equation*}
  \hat{F}:=\int_{S^2_R}|w_R-\bar{w}_R|^2d\hat{\mu}_R=\hat{F}_1+\hat{F}_2,
\end{equation*}
where 
$\hat{F}_1=|\nu_1|^2+|\nu_2|^2$, $\hat{F}_2=\sum_{i\ge 3}|\nu_i|^2$. Also splitting
\begin{equation*}
  G=\int_{S^2_R}|\nabla w_R|^2d\mu_{g_{S^2}}=\sum_{i\ge 1}\lambda_i|\nu_i|^2
  =\hat{G}_1+\hat{G}_2=\hat{G}_1+\frac{\lambda_3-1}{2\lambda_3}\hat{G}_2
  +\frac{\lambda_3+1}{2\lambda_3}\hat{G}_2,
\end{equation*}
where $\hat{G}_1=\lambda_1(|\nu_1|^2+|\nu_2|^2)=\hat{F}_1$ and
$\hat{G}_2=\sum_{i\ge 3}\lambda_i|\nu_i|^2\ge\lambda_3\hat{F}_2$,
from Lemma \ref{lemma5.1} we obtain the differential inequality 
\begin{equation}\label{5.5}
  \frac12\frac{dF}{dt}+\frac{\lambda_3-1}{2\lambda_3}\hat{G}_2
  +\frac{\lambda_3-1}{2}\hat{F}_2+c_{\rho}^{-1}\Big(\rho_t
      +2c_{\rho}(\hat{w}_R-\tilde{w}_R)\Big)^2
   \le o(1)F,
\end{equation}
where we set $c_{\rho}=\frac{\rho(\pi-\rho)}{\pi+\rho}$ for brevity.

\subsection{Equivalent norms}
For the following analysis we also need to expand $w_R$ with respect to the 
measure $\mu_R$ defined by
\begin{equation*}
  \int_{S^2_R}\varphi\,d\mu_R=\int_{S^2_R}\varphi\,d\mu_{\bar{h}}
  +\int_{\partial S^2_R}\varphi\,ds_{\bar{h}}.
\end{equation*}
Observe that the $L^2$-norms defined by $\mu_R$ and $\hat{\mu}_R$ are equivalent
in the sense that with a constant $C_R\ge 1$ for every $\varphi\in H^1(S^2_R)$ there holds
\begin{equation*}
  C_R^{-1}\|\varphi\|_{L^2(S^2_R,\mu_R)}\le\|\varphi\|_{L^2(S^2_R,\hat{\mu}_R)}
  \le C_R\|\varphi\|_{L^2(S^2_R,\mu_R)}.
\end{equation*}

Next note that by \eqref{1.5} similar to \eqref{1.19} we have
\begin{equation}\label{5.6}
  \begin{split}
    \int_{S^2_R}&w_R\,d\mu_R
    =\int_{S^2_R}w_R\,d\mu_{\bar{h}}+\int_{\partial S^2_R}w_R\,ds_{\bar{h}}
    =\int_Bw\,d\mu_h+\int_{\partial B}wds_h\\
    &=\int_B(\alpha f-K)e^{2u}dz+\int_{\partial B}(\beta j-k)e^uds_0=0.
  \end{split}
\end{equation}
Moreover, let $\Xi=(\Xi_1,\Xi_2)$ be given by
\begin{equation*}
  \Xi_i=\int_{S^2_R}X_iw_R\,d\mu_R=\int_{S^2_R}X_iw_R\,d\mu_{\bar{h}}
  +\int_{\partial S^2_R}X_iw_R\,ds_{\bar{h}},\ i=1,2.
\end{equation*}
Observe that with error $o(1)\to 0$ as $t=t_l\to\infty$ we have 
\begin{equation}\label{5.7}
  \int_{S^2_R}X_i\,d\mu_{\bar{h}}=o(1),\ \int_{\partial S^2_R}X_i\,ds_{\bar{h}}=o(1).
\end{equation}
In addition, for $1\le i,k\le 2$ with a constant $c_R>0$ there holds
\begin{equation}\label{5.8}
  \int_{S^2_R}X_i\varphi_k\,d\mu_{\bar{h}}+\int_{\partial S^2_R}X_i\varphi_k\,ds_{\bar{h}}
  = c_R\delta_{ik}+o(1);
\end{equation}
hence 
\begin{equation*}
  \Xi_i=c_R\nu_i+o(1)F^{1/2}+O(1)\hat{F}_2^{1/2}, \ i=1,2.
\end{equation*}

Set $Y_1=span\{1,X_1,X_2\}$ and let $Y_2$ be its $L^2(S^2_R,\mu_R)$-
orthogonal complement. Also let $\psi_0=\big(\int_{S^2_R}d\mu_R\big)^{-1/2}$,
$\psi_i=X_i/\|X_i\|_{L^2(S^2_R,\mu_R)}$, $i=1,2$, and let $\psi_k$, $k\ge 3$, be an
$L^2(S^2_R,\mu_R)$-orthonormal basis for $Y_2$. Expand 
$w_R=\sum_{i\ge 0}\kappa_i\psi_i$. Then with $F_0=\rho_t^2$ we can write
\begin{equation*}
  F=F_0+\sum_{i\ge 1}|\kappa_i|^2+o(1)F,
\end{equation*}
where the error $o(1)\to 0$ as $t=t_l\to\infty$ results from \eqref{5.7}, \eqref{5.8}. Split
\begin{equation*}
  F=F_0+F_1+F_2+o(1)F,\ F_1=\sum_{1\le i\le 2}|\kappa_i|^2,\ F_2=\sum_{i\ge 3}|\kappa_i|^2.
\end{equation*}
Equivalence of the $L^2$-norms gives the following result. 

\begin{lemma}\label{lemma5.3}
With error $o(1)\to 0$ as $t=t_l\to\infty$ there holds
\begin{equation*}
 C_R^{-2}F_2\le \hat{F}_2\le C_R^2F_2.
\end{equation*}
\end{lemma}

\begin{proof}
Since for $k\ge 3$ there holds $\varphi_k\perp_{L^2(S^2_R,\hat{\mu}_R)}Y_1$ we can write
\begin{equation*}
   \nu_k=\int_{S^2_R}\varphi_kw_Rd\hat{\mu}_R
   =\int_{S^2_R}\varphi_k\sum_{i\ge 3}\kappa_i\psi_id\hat{\mu}_R,\ k\ge 3.
\end{equation*}
Hence by H\"older's inequality there holds
\begin{equation*}
 \begin{split}
   \hat{F}_2&=\sum_{k\ge 3}|\nu_k|^2
   =\sum_{k\ge 3}\big(\nu_k\int_{S^2_R}\varphi_kw_Rd\hat{\mu}_R\big)
   =\int_{S^2_R}\big(\sum_{k\ge 3}\nu_k\varphi_k\sum_{i\ge 3}\kappa_i\psi_i\big)d\hat{\mu}_R\\
   &\le\|\sum_{i\ge 3}\kappa_i\psi_i\|_{L^2(S^2_R,\hat{\mu}_R)}
   \|\sum_{k\ge 3}\nu_k\varphi_k\|_{L^2(S^2_R,\hat{\mu}_R)}\\
   &\le C_R\|\sum_{i\ge 3}\kappa_i\psi_i\|_{L^2(S^2_R,\mu_R)}\hat{F}_2^{1/2}
   =C_RF_2^{1/2}\hat{F}_2^{1/2}\le\frac12C_R^2F_2+\frac12\hat{F}_2. 
 \end{split}
\end{equation*}
Similarly, since $\psi_i\perp_{L^2(S^2_R,\mu_R)}Y_1$ for $i\ge 3$
we likewise have
\begin{equation*}
   \kappa_i=\int_{S^2_R}\psi_iw_Rd\mu_R
   =\int_{S^2_R}\psi_i\sum_{k\ge 3}\nu_k\varphi_kd\mu_R,\ i\ge 3,
\end{equation*}
and thus
\begin{equation*}
 \begin{split}
   F_2&=\sum_{i\ge 3}|\kappa_i|^2=\sum_{i\ge 3}\big(\kappa_i\int_{S^2_R}\psi_iw_Rd\mu_R\big)
   =\int_{S^2_R}\sum_{i\ge 3}\kappa_i\psi_i\sum_{k\ge 3}\nu_k\varphi_kd\mu_R\\
   &\le\|\sum_{i\ge 3}\kappa_i\psi_i\|_{L^2(S^2_R,\mu_R)}
   \|\sum_{k\ge 3}\nu_k\varphi_k\|_{L^2(S^2_R,\mu_R)}\\
   &\le C_RF_2^{1/2}\|\sum_{i\ge 3}\kappa_i\psi_i\|_{L^2(S^2_R,\hat{\mu}_R)}
   =C_RF_2^{1/2}\hat{F}_2^{1/2}\le\frac12F_2+\frac12C_R^2\hat{F}_2. 
 \end{split}
\end{equation*}
Our claim follows. 
\end{proof}

The components of $\Xi$ evolve as a gradient flow.
To see this, as in \cite{Struwe-2005}
an important ingredient  is a Kazdan-Warner type set of constraints for the curvature that
reflect the action of conformal changes of the metric.

\subsection{The conformal group of $S^2_R$}\label{subsect5.3}
Recall that the gradient vector fields $\nabla X_i$, $1\le i\le 3$,
together with the generators of rotations around the
coordinate axes in $\R^3$ generate the M\"obius group $\tilde{M}$ of the sphere.
Via the map $\Psi_R$  and scaled stereographic projection $\pi_R=\Psi_R^{-1}$ we can lift
the M\"obius group $M$ of the ball $B$ to the subgroup
$M_R=\{\Psi_R\circ\Phi\circ\pi_R;\;\Phi\in M\}$ of $\tilde{M}$ preserving the circle
$\partial S^2_R$. As an important consequence of this correspondence we observe that,
in particular, for any $\xi\in  T_{id}M$ there holds
$(d\Psi_R\cdot\xi)\circ\pi_R\in T_{id}M_R$, and conversely.
This will be crucial in the following section.

Convenient representations can be obtained, as follows. Let $\Phi\in M$.
After a rotation in the plane, by \eqref{2.1} we may assume that with some $0<a<1$ there
holds $\Phi(z)=\Phi_a(z)=\frac{z+a}{1+az}$ for $z=(x,y)=x+iy\in\R^2\cong\C$. Thus, $\Phi$ 
maps the origin to the point $P=(a,0)\in\R^2$ and preserves the points $(-1,0)$ and $(1,0)$.
We call $P=\Phi(0)$ the {\it center of mass} of $\Phi$.

Conformally mapping the ball to the
half plane $\R^2_+=\{(x,y)\in\R^2;\;x>0\}$ via the map $\gamma(z)=\frac{1-z}{1+z}$
with $\gamma^2=\gamma\circ\gamma=id$,
taking the circle to the line $\{(0,y);\;y\in\R\}$ and such that
\begin{equation*}
  \gamma(1,0))=(0,0),\ \gamma(0,0))=(1,0),\ 
  \lim_{z=(x,y)\in B, (x,y)\to(-1,\pm 0)}\gamma(z)=(0,\mp\infty),
\end{equation*}
we may represent $\Phi$ as
$\Phi=\gamma\circ\tilde{\Phi}\circ\gamma$ where $\tilde{\Phi}\colon\R^2\to\R^2$ satisfies
$\tilde{\Phi}(1,0)=(\varepsilon,0)$ with $0<\varepsilon=\frac{1-a}{1+a}<1$ and 
$\tilde{\Phi}(0,0)=(0,0)$, $\tilde{\Phi}(0,y)\to(0,\pm\infty)$ as $y\to\pm\infty$.
Thus, $\tilde{\Phi}=\delta_{\varepsilon}$ with $\delta_{\varepsilon}(z)=\varepsilon z$,
and we have 
\begin{equation*}
  \Phi\circ\gamma=\gamma\circ\delta_{\varepsilon}=:\gamma_{\varepsilon}.
\end{equation*}

Next, now also allowing general $a\in\C$ as in \eqref{2.1}, for any $\psi$ we
have $\Phi_{e^{i\psi}a}(e^{i\psi}z)=e^{i\psi}\Phi_a(z)$; that is, there holds
\begin{equation*}
  \Phi_{e^{i\psi}a}=e^{i\psi}\circ\Phi_a\circ e^{-i\psi},
\end{equation*}
where we identify the number $e^{i\psi}$ with the rotation $e^{i\psi}(z)=e^{i\psi}z$.
Factoring this map via $\gamma$, and setting
\begin{equation*}
  \gamma\circ e^{-i\psi}\circ\gamma=:\rho_{\psi}\colon\R^2_+\ni
  z\mapsto\frac{1-e^{-i\psi}+(1+e^{-i\psi})z}{1+e^{-i\psi}+(1-e^{-i\psi})z}\to\R^2_+,
\end{equation*}
we find the representation
\begin{equation*}
  \Phi_{e^{i\psi}a}\circ\gamma=e^{i\psi}\circ\Phi_a\circ\gamma\circ\gamma\circ e^{-i\psi}\circ\gamma
  =e^{i\psi}\circ\gamma_{\varepsilon}\circ\rho_{\psi}.
\end{equation*}

Since the maps $\Psi_R$ and $\psi_R=\pi_{\R^2}\circ\Psi_R$ commute with rotations, 
there thus also holds
\begin{equation*}
  \psi_R\circ\Phi_{e^{i\psi}a}\circ\gamma
  =e^{i\psi}\circ\psi_R\circ\gamma_{\varepsilon}\circ\rho_{\psi}.
\end{equation*}

Finally, for $a,z\in B$ and any $\phi\in\R$ we also let
\begin{equation*}
  \Phi_{a,\phi}(z)=e^{i\phi}\frac{z+a}{1+\bar{a}z}.
\end{equation*}

The above formulas allow to easily compute the differential of the map
$\psi_R\circ\Phi_a$ in stereographic coordinates with respect to
both $z$ and $a$ in $B$.

\subsection{A Kazdan-Warner identity}
Similar to the case of the prescribed curvature problem on $S^2$,
the conformal invariance of the Liouville energy  
\begin{equation*}
   E_0(u)=\frac12\int_{S^2_R}(|\nabla u|^2+2u)\,d\mu_{g_{S^2}}+k_R\int_{\partial S^2_R}u\,ds_{g_{S^2}},
\end{equation*}
where $k_R$ is the boundary geodesic curvature in the standard metric,
gives rise to a number of constraints that the curvatures
\begin{equation}\label{5.9}
  K_g=e^{-2u}(-\Delta_{g_{S^2}}u+1)
\end{equation}
and
\begin{equation}\label{5.10}
  k_g=e^{-u}(\frac{\partial u}{\partial\nu_{S_R^2}}+k_R)
\end{equation}
of any conformal metric $g=e^{2u}g_{S_R^2}$ on $S^2_R$ naturally satisfy.
Our approach to these results is modelled on the derivation of the corresponding
Kazdan-Warner type constraints on $S^2$ in \cite{Chang-Yang-1987}, Corollary 2.1.

To see conformal invariance of the Liouville energy $E_0$, we first observe that for any
metric $g$ as above the function $u$ is a critical point of the energy
\begin{equation*}
  E_{K,k}(u)=E_0(u)-\frac12\int_{S^2_R}Ke^{2u}\,d\mu_{g_{S^2}}
    -\int_{\partial S^2_R}ke^{u}\,ds_{g_{S^2}},
\end{equation*}
where $K=K_g$, $k=k_g$. We then have the following result.

\begin{lemma}\label{lemma5.4}
For any $u\in H^1(S^2_R)$ and any $\Phi\in M_R$ there holds
\begin{equation*}
   E_0(u)=E_0(v),
\end{equation*}
where 
\begin{equation}\label{5.11}
  v=u\circ\Phi+\frac12\log(\sqrt{det(d\Phi^td\Phi)})\hbox{ in }S^2_R,
\end{equation}
similar to \eqref{2.2}. Again we note that since $\Phi$ is conformal, on $\partial S^2_R$
we have $v=u\circ\Phi+\log(|\frac{\partial\Phi}{\partial\tau}|)$
with the positively oriented unit tangent vector $\tau$ along $\partial S^2_R$. 
\end{lemma}

\begin{proof}
i) First consider the case $u=0$. Observe that by naturality of the curvature and \eqref{5.9},
\eqref{5.10} for $v=\frac12\log(\sqrt{det(d\Phi^td\Phi)})$ with any $\Phi\in M_R$ in view of
$e^{2v}g_{S^2}=\Phi^*g_{S^2}$ we have
\begin{equation*}
  e^{-2v}(-\Delta_{g_{S^2}}v+1)=K_{g_{S^2}}\circ\Phi=1,
\end{equation*}
and $v$ also solves the equation 
$\partial v/\partial\nu_{S^2_R}=k_R(e^v-1)$ on $\partial S^2_R$.

Thus, $v$ is a critical point of the functional $E_{1,k_R}$, where 
\begin{equation*}
  E_{1,k_R}(u)=\frac12\int_{S^2_R}(|\nabla u|^2+2u-e^{2u})d\mu_{g_{S^2}}
  +k_R\int_{\partial S^2_R}(u-e^u)\,ds_{g_{S^2}}
\end{equation*}
for any $u\in H^1(S^2_R)$, and for any $\varphi\in H^1(S^2_R)$ we find
\begin{equation*}
  \langle dE_{1,k_R}(v),\varphi\rangle_{H^{-1}\times H^1}
  =\int_{\partial S^2_R}\varphi(\frac{\partial v}{\partial\nu_{S^2_R}}+k_R(1-e^v))ds_{g_{S^2}}=0.
\end{equation*}
It follows that for any $\Phi\in M_R$ and a $C^1$-family of M\"obius transformations
$\Phi(t)\in M_R$ with $\Phi(0)=id$ and $\Phi(1)=\Phi$, letting
$v(t)=\frac12\log(\sqrt{det(d\Phi^t(t)d\Phi(t))})$ we have 
\begin{equation*}
  \frac{d}{dt}E_{1,k_R}(v(t))=\langle dE_{1,k_R}(v),\frac{dv}{dt}\rangle_{H^{-1}\times H^1}=0,
\end{equation*}
whence
\begin{equation*}
  E_{1,k_R}(v(1))=E_{1,k_R}(v(0))=E_{1,k_R}(0).
\end{equation*}

Since clearly for any $\Phi\in M_R$ and $v=\frac12\log(\sqrt{det(d\Phi^td\Phi)})$
as above we also have 
\begin{equation*}
  \begin{split}
    \frac12\int_{S^2_R}e^{2v}&\,d\mu_{g_{S^2}}+k_R\int_{\partial S^2_R}e^v\,ds_{g_{S^2}}
    = \frac12\int_{S^2_R}\sqrt{det(d\Phi^td\Phi)}\,d\mu_{g_{S^2}}\\
    &+k_R\int_{\partial S^2_R}|\frac{\partial\Phi}{\partial\tau}|ds_{g_{S^2}}
    =\frac12\int_{S^2_R}\,d\mu_{g_{S^2}}+k_R\int_{\partial S^2_R}ds_{g_{S^2}},
  \end{split}
\end{equation*}
we conclude
\begin{equation}\label{5.12}
  E_0(v)=E_0(v(1))=E_0(v(0))=E_0(0)=0.
\end{equation}

ii) Expanding the quadratic term and integrating by parts, for the general case with
$v=u\circ\Phi+\frac12\log(\sqrt{det(d\Phi^td\Phi)})$ by part i) and conformal invariance
of the Dirichlet integral we have
\begin{equation*}
  \begin{split}
   E_0(v)&=\frac12\int_{S^2_R}(|\nabla(u\circ\Phi)|^2+2u\circ\Phi)\,d\mu_{g_{S^2}}\\
   &\quad+\frac12\int_{S^2_R}\nabla(u\circ\Phi)\nabla\log(\sqrt{det(d\Phi^td\Phi)})
   \,d\mu_{g_{S^2}}+k_R\int_{\partial S^2_R}u\circ\Phi\,ds_{g_{S^2}}\\
   &=E_0(u)+\int_{S^2_R}((u\circ\Phi)\sqrt{det(d\Phi^td\Phi)}-u)\,d\mu_{g_{S^2}}\\
   &\quad+\int_{\partial S^2_R}(u\circ\Phi)
   \big(\frac12
   \frac{\partial\log(\sqrt{det(d\Phi^td\Phi)})}{\partial\nu_{S^2_R}}+k_R\big)ds_{g_{S^2}}
            -k_R\int_{\partial S^2_R}u\,ds_{g_{S^2}},
  \end{split}
\end{equation*}
where we also used \eqref{5.12} and the fact that $w=\frac12\log(\sqrt{det(d\Phi^td\Phi)})$
solves \eqref{5.9} with $K_g=1$, as noted in part i) above.
Since $w$ also solves \eqref{5.10} with $k_g=k_R$, and since we have
$\log(\sqrt{det(d\Phi^td\Phi)})=2\log(|\frac{\partial\Phi}{\partial\tau}|)$ on $\partial S^2_R$,
a change of variables gives
\begin{equation*}
   \int_{S^2_R}((u\circ\Phi)\sqrt{det(d\Phi^td\Phi)}-u)\,d\mu_{g_{S^2}}=0
\end{equation*}
as well as
\begin{equation*}
  \begin{split}
    \int_{\partial S^2_R}&(u\circ\Phi)
    (\frac12\frac{\partial\log(\sqrt{det(d\Phi^td\Phi)})}{\partial\nu_{S^2_R}}+k_R)
    \,ds_{g_{S^2}}\\
    &=k_R\int_{\partial S^2_R}(u\circ\Phi)|\frac{\partial\Phi}{\partial\tau}|ds_{g_{S^2}}
    =k_R\int_{\partial S^2_R}u\,ds_{g_{S^2}},
  \end{split}
\end{equation*}
proving our claim.
\end{proof}

Next let $g=e^{2u}g_{S^2}$ be any conformal metric on $S^2_R$ with Gauss curvature
and boundary geodesic curvature given by \eqref{5.9}, \eqref{5.10}.
From the conformal invariance of $E_0$ established in Lemma \ref{lemma5.4}
we obtain the following Kazdan-Warner type condition for the curvature functions
$K=K_g$ and $k=k_g$. For clarity we use the directional derivative $dK$ instead 
of the gradient $\nabla K$, as the latter also depends on the metric whereas the 
former does not.

\begin{lemma}\label{lemma5.5}
For any $\bar{\xi}\in T_{id}M_R$ there holds
\begin{equation*}
  \frac12\int_{S^2_R}dK\cdot\bar{\xi}\,d\mu_g+\int_{\partial S^2_R}dk\cdot\bar{\xi},ds_g=0.
\end{equation*}
\end{lemma}

\begin{proof}
Let $\bar{\xi}\in T_{id}M_R$ and let $\Phi(t)\in M_R$ be a $C^1$-family of 
M\"obius transformations defined in a neighbourhood of $t=0$, with $\Phi(0)=id$
and such that $\frac{\partial\Phi}{\partial t}\big|_{t=0}=\bar{\xi}$.
Invariance $E_0(u)=E_0(u(t))$ of the Liouville energy of
\begin{equation*}
   u(t)=u\circ\Phi(t)+\frac12\log\big(\sqrt{det(d\Phi^t(t)d\Phi(t))}\big)
\end{equation*}
gives 
\begin{equation*}
  E_{K,k}(u(t))=E_0(u)-\frac12\int_{S^2_R}K\circ\Phi(t)^{-1}e^{2u}\,d\mu_{g_{S^2}}
    -\int_{\partial S^2_R}k\circ\Phi(t)^{-1}e^{u}\,ds_{g_{S^2}}.
\end{equation*}
Criticality of $E_{K,k}(u)$ then yields the equation
\begin{equation*}
  0=\frac{d}{dt}\big|_{t=0}E_{K,k}(u(t))=\frac12\int_{S^2_R}dK\cdot\bar{\xi}\,d\mu_g
   +\int_{\partial S^2_R}dk\cdot\bar{\xi}\,ds_g,
\end{equation*}
as claimed.
\end{proof}

\subsection{The motion of the center of mass}
We now return to the setting of Corollary \ref{cor4.4} and again denote as $u=u(t)$
a solution of the flow \eqref{1.13} - \eqref{1.15} concentrating as $t=t_l\to\infty$
at a point $z_0\in\partial B$
and as $v=v(t)$ its normalized companion given by \eqref{2.2} with $R=R(z_0)$
and a family of conformal diffeomorphisms $\Phi=\Phi(t)$ of the disc $B$. 

Recall that we defined
$\bar{\Phi}=\bar{\Phi}(t)=\Psi_R\circ\Phi(t)\circ\pi_R\colon S^2_R\to S^2_R $
and set
\begin{equation*}
  \bar{g}=\pi_R^*g=e^{2\bar{u}}g_{S^2},\ \bar{h}=\pi_R^*h=\bar{\Phi}^*\bar{g}=e^{2\bar{v}}g_{S^2},
\end{equation*}
where
\begin{equation*}
  \bar{v}=\bar{u}\circ\bar{\Phi}+\frac12\log(\sqrt{det(d\bar{\Phi}^td\bar{\Phi}}))
  \hbox{ in } S^2_R,\ 
  \bar{v}=\bar{u}\circ\bar{\Phi}+\log(|\frac{\partial\bar{\Phi}}{\partial\tau}|)
  \hbox{ on } \partial S^2_R
\end{equation*}
for each $t>0$, and where $\tau$ is the oriented unit tangent vector along $\partial S^2_R$.

Let 
\begin{equation}\label{5.13}
  \bar{\xi}=(d\bar{\Phi})^{-1}\frac{\partial\bar{\Phi}}{\partial t}
  =(d\Psi_R\cdot\xi)\circ\pi_R
\end{equation}
be the vector field generating the flow $(\bar{\Phi}(t))_{t>0}$, where $\xi=(d\Phi)^{-1}\Phi_t$ as
before. Similar to \cite{Struwe-2005}, formulas (17) and (18), we then have 
\begin{equation*}
  \bar{v}_t=\bar{u}_t\circ\bar{\Phi}+\frac12e^{-2\bar{v}}div_{S^2}(\bar{\xi} e^{2\bar{v}})\
  \hbox{ in } S^2_R,\ 
  \bar{v}_t=\bar{u}_t\circ\bar{\Phi}
  +e^{-\bar{v}}\frac{\partial(\tau\cdot\bar{\xi} e^{\bar{v}})}{\partial\tau}\
  \hbox{ on } \partial S^2_R.
\end{equation*}

Using that our normalization \eqref{2.6} implies
\begin{equation*}
  \begin{split}
   \frac12\int_{S^2_R} Z\,d\mu_{\bar{h}}+\int_{\partial S^2_R} Z\,ds_{\bar{h}}
   =\frac12\int_B\psi_R\,e^{2v}dz+\int_{\partial B}\psi_Re^vds_0=0,
  \end{split}
\end{equation*}
and observing that we have $dZ=\pi_{\R^2}$ in $S^2_R$ and
$\partial Z/\partial\tau=\tau\cdot dZ=\tau$ as well as $\nu_{S^2_R}\cdot\bar{\xi}=0$ and hence
$\bar{\xi}=\tau\,\tau\cdot\bar{\xi}$ on $\partial S^2_R$, we then compute
\begin{equation}\label{5.14}
 \begin{split}
   0&=\frac{d}{dt}\big(\frac12\int_{S^2_R} Z\,d\mu_{\bar{h}}
      +\int_{\partial S^2_R} Z\,ds_{\bar{h}}\big)
   =\int_{S^2_R} Z\bar{v}_t\,d\mu_{\bar{h}}
      +\int_{\partial S^2_R} Z\bar{v}_t\,ds_{\bar{h}}\\
   &=\int_{S^2_R} Z\bar{u}_t\circ\bar{\Phi}\,d\mu_{\bar{h}}
      +\int_{\partial S^2_R} Z\bar{u}_t\circ\bar{\Phi}\,ds_{\bar{h}}\\
   &\qquad+\frac12\int_{S^2_R} Zdiv_{S^2}(\bar{\xi}e^{2\bar{v}})\,d\mu_{g_{S^2}}
   +\int_{\partial S^2_R} Z\frac{\partial(\tau\cdot\bar{\xi}e^{\bar{v}})}{\partial\tau}\,ds_{g_{S^2}}\\
    &=\int_{S^2_R} Z\bar{u}_t\circ\bar{\Phi}\,d\mu_{\bar{h}}
    +\int_{\partial S^2_R} Z\bar{u}_t\circ\bar{\Phi}\,ds_{\bar{h}}
    -\frac12\int_{S^2_R}\pi_{\R^2}\bar{\xi}\,d\mu_{\bar{h}}
      -\int_{\partial S^2_R}\bar{\xi}\,ds_{\bar{h}}.
 \end{split}
\end{equation}

The vector
\begin{equation*}
  \bar{\Xi}:=\frac12\int_{S^2_R}\pi_{\R^2}\bar{\xi}\,d\mu_{\bar{h}}
  +\int_{\partial S^2_R}\bar{\xi}\,ds_{\bar{h}}
  =\int_{S^2_R} Z\bar{u}_t\circ\bar{\Phi}\,d\mu_{\bar{h}}
    +\int_{\partial S^2_R} Z\bar{u}_t\circ\bar{\Phi}\,ds_{\bar{h}}
\end{equation*}
uniquely determines $\bar{\xi}$. Note that we have
$\bar{u}_t\circ\bar{\Phi}=u_t\circ\Phi\circ\pi_R=w_R$; hence $\bar{\Xi}=\Xi$.

Now, given a time $t_0:=t_l>0$ we choose coordinates such that $\Phi(t_0)=\Phi_{a_0}$
for some $0<a_0<1$ and we express $\Xi=\Xi_1+i\Xi_2$.
At times $t$ near $t_0$ we then have
$\Phi(t)=\Phi_{e^{i\phi}a}=e^{i\phi}\circ\Phi_a\circ e^{-i\phi}$
for $0<a=a(t)<1$, $\phi=\phi(t)$ satisfying $a(t_0)=a_0$, $\phi(t_0)=0$, and, with
$0<\varepsilon_0=\frac{1-a_0}{1+a_0}<1$, for the motion of the center of mass
$P=P(t)=e^{i\phi}a$ the following holds.

\begin{lemma}\label{lemma5.6}
With real coefficients $A_i=A_i(z_0)>0$
and complex error $o(1)\to 0$ in $\C$ as $l\to\infty$ there holds
\begin{equation*}
  \begin{split}
    \big((A_1+o(1))\frac{da}{dt}+i(A_2+o(1))\frac{d\phi}{dt}\big)\big|_{t=t_0}
    &=-\varepsilon_0\Xi.
  \end{split}
\end{equation*}
\end{lemma}

\begin{proof} i) From \eqref{5.13} we have
\begin{equation*}
  \begin{split}
  \Xi&=\frac12\int_{S^2_R}(d\psi_R\cdot\xi)\circ\pi_R\,d\mu_{\pi_R^*h}
  +\int_{\partial S^2_R}(d\psi_R\cdot\xi)\circ\pi_R\,ds_{\pi_R^*h}\\
  &=\frac12\int_B(d\psi_R\cdot\xi)\,d\mu_h
    +\int_{\partial B}(d\psi_R\cdot\xi)\,ds_h.
  \end{split}
\end{equation*}
Given $t_0=t_l$, for $t$ close to $t_0$ as in \cite{Struwe-2005}, p. 39, by slight abuse of notation
we set
\begin{equation*}
  \Phi_{t_0}(t)=\Phi(t_0)^{-1}\Phi(t)=\Phi_{a_0}^{-1}\Phi_{e^{i\phi(t)}a(t)}
\end{equation*}
so that $\xi=\frac{d\Phi_{t_0}}{dt}(t_0)$, and we let 
$\varepsilon=\varepsilon(t)=\frac{1-a}{1+a}$ with $\varepsilon(t_0)=\varepsilon_0$.

In stereographic coordinates we can represent
$\Phi_{a_0}\circ\gamma=\gamma_{\varepsilon_0}$. Thus, and with
$\gamma^{-1}\circ\Phi_{a_0}^{-1}=(\Phi_{a_0}\circ\gamma)^{-1}=\gamma_{\varepsilon_0}^{-1}$,
for $\xi$ we obtain the representation
\begin{equation*}
\begin{split}
 &\xi\circ\gamma=\frac{d}{dt}(\Phi_{t_0}\circ\gamma)
   =\frac{d}{dt}(\gamma\circ\gamma^{-1}\circ\Phi_{a_0}^{-1}\circ e^{i\phi}
   \circ\Phi_a\circ e^{-i\phi}\circ\gamma)\\
 &=d\gamma(d\gamma_{\varepsilon_0}^{-1}\circ\gamma_{\varepsilon_0})i\frac{d\phi}{dt}\gamma_{\varepsilon_0}
   -i\frac{d\phi}{dt}\gamma
   +d\gamma(d\gamma_{\varepsilon_0}^{-1}\circ\gamma_{\varepsilon_0})
   \big(\frac{d\Phi_a}{da}\circ\gamma\big)\frac{da}{dt}\\
 &=i d\gamma(d\gamma_{\varepsilon_0})^{-1}\frac{d\phi}{dt}\gamma_{\varepsilon_0}
   -i\frac{d\phi}{dt}\gamma
   +d\gamma(d\gamma_{\varepsilon_0})^{-1}\frac{d\gamma_{\varepsilon}}{d\varepsilon}
   \frac{d\varepsilon}{da}\frac{da}{dt},
 \end{split}
\end{equation*}
where all terms are evaluated at $t=t_0$. With
\begin{equation*}
  d\gamma(z)=\frac{-2}{(1+z)^2},\quad
  \frac{d\gamma_{\varepsilon}(z)}{d\varepsilon}=z\cdot d\gamma(\varepsilon z),\quad
  \frac{d\varepsilon}{da}=-\frac{2}{(1+a)^2},
\end{equation*}
and writing $\varepsilon=\varepsilon_0$ for brevity, at $t=t_0$ we thus have
\begin{equation*}
  \begin{split}
     \xi(\gamma(z))
    &=i\Big(\frac{(1+\varepsilon z)^2}{\varepsilon(1+z)^2}\gamma_{\varepsilon}(z)
      -\gamma(z)\Big)\frac{d\phi}{dt}
      +\frac{4z}{\varepsilon(1+z)^2(1+a)^2}\frac{da}{dt}\\
    &=i\gamma(z)\Big(\frac{1-(\varepsilon z)^2}{\varepsilon(1-z^2)}-1\Big)\frac{d\phi}{dt}
      +\frac{4z\gamma(z)}{\varepsilon(1-z^2)(1+a)^2}\frac{da}{dt}\\
    &=i\gamma(z)\Big(\frac{1-\varepsilon^2}{\varepsilon(1-z^2)}-(1-\varepsilon)\Big)\frac{d\phi}{dt}
      +\frac{4z\gamma(z)}{\varepsilon(1-z^2)(1+a)^2}\frac{da}{dt}.
   \end{split}
\end{equation*}

Now interpreting each $z$ as a vector $z\in\R^2$ with dual co-vector $z^t$, satisfying
$z^tz=|z|^2$, $z^tiz=0$, we also have
\begin{equation*}
  d\psi_R(z)=\frac{2R(1+R^2|z|^2-2R^2zz^t)}{(1+R^2|z|^2)^2}
\end{equation*}
at each $z\in B$. Thus, and computing
\begin{equation*}
  \begin{split}
    \frac{1}{1-z^2}&=\frac{1-\bar{z}^2}{|1-z^2|^2}
    =\frac{1-x^2+y^2+2ixy}{(1-x^2+y^2)^2+4x^2y^2}
   \end{split}
\end{equation*}
as well as
\begin{equation*}
  \begin{split}
  \frac{z}{1-z^2}&=\frac{z(1-\bar{z}^2)}{|1-z^2|^2}=\frac{(x+iy)(1-x^2+y^2+2ixy)}{|1-z^2|^2}\\
  &=\frac{x(1-x^2-y^2)}{|1-z^2|^2}+i\frac{y(1+x^2+y^2)}{|1-z^2|^2}
  =\frac{x(1-|z|^2)}{|1-z^2|^2}+i\frac{y(1+|z|^2)}{|1-z^2|^2},
   \end{split}
\end{equation*}
we have
\begin{equation*}
  \varepsilon d\psi_R(\gamma(z))\cdot\xi(\gamma(z))
  =\hat{m}\gamma(z)\frac{da}{dt}+(1-\varepsilon^2)\hat{n}\gamma(z)\frac{d\phi}{dt}
\end{equation*}
with $\hat{m}=\hat{m}_1+i\hat{m}_2$, $\hat{n}=\hat{n}_1+i\hat{n}_2$ given by
\begin{equation*}
  \begin{split}
  \hat{m}&=\frac{8R(1-R^2|\gamma(z)|^2)x(1-|z|^2)}
  {(1+R^2|\gamma(z)|^2)^2|1-z^2|^2(1+a)^2}
  +i\frac{8Ry(1+|z|^2)}{(1+R^2|\gamma(z)|^2)|1-z^2|^2(1+a)^2}
  \end{split}
\end{equation*}
and, with error $|O(\varepsilon)|\le C\varepsilon$ as $\varepsilon\to 0$,
\begin{equation*}
  \hat{n}=-\frac{4R(1-R^2|\gamma(z)|^2)xy}{(1+R^2|\gamma(z)|^2)^2
   |1-z^2|^2}+i\frac{2R(1-x^2+y^2)}{(1+R^2|\gamma(z)|^2)|1-z^2|^2}
    +O(\varepsilon).
\end{equation*}

With
\begin{equation*}
  \begin{split}
   \gamma(z)&=\frac{1-z}{1+z}=\frac{(1-z)(1+\bar{z})}{|1+z|^2}
   =\frac{1-|z|^2-2iy}{|1+z|^2}
  \end{split}
\end{equation*}
this gives
\begin{equation}\label{5.15}
  \varepsilon d\psi_R(\gamma(z))\cdot\xi(\gamma(z))
  =m\frac{da}{dt}+(1-\varepsilon^2)(n+O(\varepsilon))\frac{d\phi}{dt}
\end{equation}
where now $m=m_1+im_2$ with
\begin{equation*}
  \begin{split}
  m_1&=\frac{8R(1-R^2|\gamma(z)|^2)x(1-|z|^2)^2}
   {(1+R^2|\gamma(z)|^2)^2|1-z^2|^2|1+z|^2(1+a)^2}\\
  &\qquad+\frac{16Ry^2(1+|z|^2)}{(1+R^2|\gamma(z)|^2)|1-z^2|^2|1+z|^2(1+a)^2}\\
     &=\frac{8R\big((1-R^2|\gamma(z)|^2)x(1-|z|^2)^2+2(1+R^2|\gamma(z)|^2)y^2(1+|z|^2)\big)}
       {(1+R^2|\gamma(z)|^2)^2|1-z^2|^2|1+z|^2(1+a)^2}>0
  \end{split}
\end{equation*}
and 
\begin{equation*}
  \begin{split}
  m_2&=-\frac{16R(1-R^2|\gamma(z)|^2)xy(1-|z|^2)}
   {(1+R^2|\gamma(z)|^2)^2|1-z^2|^2|1+z|^2(1+a)^2}\\
  &\qquad+\frac{8Ry(1+|z|^2)(1-|z|^2)}{(1+R^2|\gamma(z)|^2)|1-z^2|^2|1+z|^2(1+a)^2}\\
   &=\frac{8Ry(1-|z|^2)\big((1+R^2|\gamma(z)|^2)(1+|z|^2)-2(1-R^2|\gamma(z)|^2)x\big)}
      {(1+R^2|\gamma(z)|^2)^2|1-z^2|^2|1+z|^2(1+a)^2};
  \end{split}
\end{equation*}
moreover, $n=n_1+in_2$ with 
\begin{equation*}
  \begin{split}
   n_1&=\frac{4Ry(1-x^2+y^2)}{(1+R^2|\gamma(z)|^2)|1-z^2|^2|1+z|^2}
   -\frac{4R(1-R^2|\gamma(z)|^2)xy(1-|z|^2)}
   {(1+R^2|\gamma(z)|^2)^2|1-z^2|^2|1+z|^2}\\
    &=\frac{4Ry\big((1+R^2|\gamma(z)|^2)(1-x^2+y^2)-(1-R^2|\gamma(z)|^2)x(1-|z|^2)\big)}
        {(1+R^2|\gamma(z)|^2)^2|1-z^2|^2|1+z|^2}
   \end{split}
\end{equation*}
and 
\begin{equation*}
  \begin{split}
   n_2&=\frac{8R(1-R^2|\gamma(z)|^2)xy^2}{(1+R^2|\gamma(z)|^2)^2|1-z^2|^2|1+z|^2}
        +\frac{2R(1-|z|^2)(1-x^2+y^2)}{(1+R^2|\gamma(z)|^2)|1-z^2|^2|1+z|^2}\\
      &=\frac{2R\big(4(1-R^2|\gamma(z)|^2)xy^2+(1+R^2|\gamma(z)|^2)(1-|z|^2)(1-x^2+y^2)\big)}
        {(1+R^2|\gamma(z)|^2)^2|1-z^2|^2|1+z|^2}.
   \end{split}
\end{equation*}

ii) Observe that with $|1-z^2|=|1-z||1+z|$, $|1-z|^2=(1-x)^2+y^2$, and also estimating 
$\big|1-|z|\big|\le|1-z|$, we readily see that $m_1$ is bounded, uniformly in $\varepsilon$.
Moreover, simplifying the expressions for $m_2$ and $n$ derived above, we can see that 
also these terms are uniformly bounded.

Indeed, using that $|\gamma(z)|=\frac{|1-z|}{|1+z|}$ we first compute
\begin{equation*}
  \begin{split}
    (1+&R^2|\gamma(z)|^2)(1+|z|^2)-2(1-R^2|\gamma(z)|^2)x\\
    &=(1+x^2-2x+y^2)+R^2|\gamma(z)|^2(1+x^2+2x+y^2)\\
    &=|1-z|^2+R^2|\gamma(z)|^2|1+z|^2=(1+R^2)|1-z|^2.
   \end{split}
\end{equation*}
Splitting $|1-z^2|^2=|1-z|^2|1+z|^2$ and also noting that for $x>0$ there holds 
\begin{equation*}
  \begin{split}
    (1+R^2|\gamma(z)|^2&)|1+z|^2=|1+z|^2+R^2|1-z|^2\\
    &=(1+x)^2+R^2(1-x)^2+(1+R^2)y^2\ge 1,
  \end{split}
\end{equation*}
we then see that
\begin{equation*}
  \begin{split}
  m_2&=\frac{8R(1+R^2)y(1-|z|^2)}{(1+R^2|\gamma(z)|^2)^2|1+z|^4(1+a)^2},
  \end{split}
\end{equation*}
and $m_2$ is uniformly bounded, as claimed. 

Similarly, with $|\gamma(z)||1+z|=|1-z|$ we compute 
\begin{equation*}
  \begin{split}
   (1+&R^2|\gamma(z)|^2)(1-x^2+y^2)-(1-R^2|\gamma(z)|^2)x(1-x^2-y^2)\\
   &=\big((1-x)(1-x^2)+(1+x)y^2\big)+R^2|\gamma(z)|^2\big((1+x)(1-x^2)+(1-x)y^2\big)\\
   &=(1+x)\big((1-x)^2+y^2\big)+(1-x)R^2|\gamma(z)|^2\big((1+x)^2+y^2\big)\\
   &=(1+x)|1-z|^2+(1-x)R^2|\gamma(z)|^2|1+z|^2
   =\big((1+x)+(1-x)R^2\big)|1-z|^2
  \end{split}
\end{equation*}
to obtain
\begin{equation*}
  \begin{split}
   n_1&=\frac{4Ry\big((1+R^2|\gamma(z)|^2)(1-x^2+y^2)-(1-R^2|\gamma(z)|^2)x(1-|z|^2)\big)}
        {(1+R^2|\gamma(z)|^2)^2|1-z^2|^2|1+z|^2}\\
    &=\frac{4Ry\big((1+x)+(1-x)R^2\big)}
        {(1+R^2|\gamma(z)|^2)^2|1+z|^4},
   \end{split}
\end{equation*}
and also $|n_1|\le C<\infty$, uniformly in $\varepsilon$.

Finally, with 
\begin{equation*}
  \begin{split}
    (1-x^2-y^2)&(1-x^2+y^2)=(1-x^2)^2-y^4=(1-x)^2(1+x)^2-y^4\\
    &=\big((1-x)^2\pm y^2\big)\big((1+x)^2\mp y^2\big)\mp 4xy^2
  \end{split}
\end{equation*}
and
\begin{equation*}
  \begin{split}
    |1-z^2|^2= |1+z|^2|1-z|^2=\big((1+x)^2+y^2\big)\big((1-x)^2+y^2\big)
  \end{split}
\end{equation*}
we can write
\begin{equation*}
  \begin{split}
    &\frac{\big(4(1-R^2|\gamma(z)|^2)xy^2+(1+R^2|\gamma(z)|^2)(1-x^2-y^2)(1-x^2+y^2)\big)}
      {|1-z^2|^2|1+z|^2}\\
    &\qquad=\frac{(1+x)^2-y^2}{|1+z|^4}
     +R^2|\gamma(z)|^2\frac{(1-x)^2-y^2}{|1-z|^2|1+z|^2}\\
    &=\frac{(1+x)^2-y^2+R^2\big((1-x)^2-y^2\big)}{|1+z|^4}
    =\frac{(1+x)^2+R^2(1-x)^2-(1+R^2)y^2}{|1+z|^4}.
    \end{split}
\end{equation*}
to find
\begin{equation*}
    n_2=\frac{2R\big((1+x)^2+R^2(1-x)^2-(1+R^2)y^2\big)}
    {(1+R^2|\gamma(z)|^2)^2|1+z|^4}
\end{equation*}
and again we see that $n_2$ is uniformly bounded on $\R^2_+$, uniformly in $\varepsilon$. 

iii) With \eqref{5.15} we can write
\begin{equation*}
  \begin{split}
    2\varepsilon\Xi&=\varepsilon\int_B(d\psi_R\cdot\xi)e^{2v}dz
    +2\varepsilon\int_{\partial B}(d\psi_R\cdot\xi)e^vds_0
    = A\frac{da}{dt}+B\frac{d\phi}{dt}
 \end{split}
\end{equation*}
with
\begin{equation*}
  \begin{split}
     A&:=\int_{\R^2_+}m\,e^{2v\circ\gamma}\,d\mu_{\gamma^*g_{\R^2}}
     +2\int_{\partial\R^2_+}m\, e^{v\circ\gamma}\,ds_{\gamma^*g_{\R^2}}
 \end{split}
\end{equation*}
and
\begin{equation*}
  \begin{split}
     B&:=\int_{\R^2_+}n\,e^{2v\circ\gamma}\,d\mu_{\gamma^*g_{\R^2}}
     +2\int_{\partial\R^2_+}n\, e^{v\circ\gamma}\,ds_{\gamma^*g_{\R^2}}+O(\varepsilon),
 \end{split}
\end{equation*}
where $m=m_1+im_2$, $n=n_1+in_2$ as above. Since by part ii) the terms $m_2$ and $n_1$ are
both bounded uniformly in $\varepsilon>0$, with error $o(1)\to 0$ as $l\to\infty$ there holds
\begin{equation*}
  \begin{split}
    \int_{\R^2_+}&m_2\,e^{2v\circ\gamma}\,d\mu_{\gamma^*g_{\R^2}}
    +2\int_{\partial\R^2_+}m_2\, e^{v\circ\gamma}\,ds_{\gamma^*g_{\R^2}}\\
    &=\int_{\R^2_+}m_2\,e^{2v_{\infty}\circ\gamma}\,d\mu_{\gamma^*g_{\R^2}}
     +2\int_{\partial\R^2_+}m_2\, e^{v_{\infty}\circ\gamma}\,ds_{\gamma^*g_{\R^2}}+o(1)=o(1)
 \end{split}
\end{equation*}
by symmetry, observing that $v_{\infty}=\lim_{l\to\infty}v(t_l)$ given by Corollary \ref{cor4.4}
is even in $y$, where $z=x\pm iy\in B$, whereas $m_2$ is odd; similarly
\begin{equation*}
  \begin{split}
    \int_{\R^2_+}&n_1\,e^{2v\circ\gamma}\,d\mu_{\gamma^*g_{\R^2}}
    +2\int_{\partial\R^2_+}n_1\, e^{v\circ\gamma}\,ds_{\gamma^*g_{\R^2}}\\
    &=\int_{\R^2_+}n_1\,e^{2v_{\infty}\circ\gamma}\,d\mu_{\gamma^*g_{\R^2}}
     +2\int_{\partial\R^2_+}n_1\, e^{v_{\infty}\circ\gamma}\,ds_{\gamma^*g_{\R^2}}+o(1)=o(1).
 \end{split}
\end{equation*}

With boundedness of $m_1$ and $n_2$ it likewise follows that 
\begin{equation*}
  \begin{split}
     A+o(1)&=\int_{\R^2_+}m_1\,e^{2v_{\infty}\circ\gamma}\,d\mu_{\gamma^*g_{\R^2}}
    +2\int_{\partial\R^2_+}m_1\, e^{v_{\infty}\circ\gamma}\,ds_{\gamma^*g_{\R^2}}>0,
 \end{split}
\end{equation*}
and
\begin{equation*}
  \begin{split}
     B+o(1)&=i\int_{\R^2_+}n_2\,e^{2v_{\infty}\circ\gamma}\,d\mu_{\gamma^*g_{\R^2}}
    +2i\int_{\partial\R^2_+}n_2\, e^{v_{\infty}\circ\gamma}\,ds_{\gamma^*g_{\R^2}}.
 \end{split}
\end{equation*}

iv) Finally, we also determine the sign of the latter integrals. Note that
\begin{equation*}
  d\mu_{\gamma^*g_{\R^2}}(z)=\big(\frac{2}{|1+z|^2}\big)^2dz,\
  e^{2v_{\infty}\circ\gamma}
  =\big(\alpha f(z_0)\big)^{-1}\big(\frac{2R}{1+R^2|\gamma(z)|^2}\big)^2.
\end{equation*}
Thus on $\R^2_+$ we obtain the expression
\begin{equation*}
    n_2e^{2v_{\infty}\circ\gamma}d\mu_{\gamma^*g_{\R^2}}
    =\frac{32 R^3\big((1+x)^2+R^2(1-x)^2-(1+R^2)y^2\big)dz}
    {\alpha f(z_0)(1+R^2|\gamma(z)|^2)^4|1+z|^8}.
\end{equation*}
Moreover,  on $\partial\R^2_+$ with $x=0$ and $|\gamma(z)|=1$ on $\partial\R^2_+$ 
we find 
\begin{equation*}
  \begin{split}
    n_2e^{v_{\infty}\circ\gamma}ds_{\gamma^*g_{\R^2}}
    &=\frac{8R^2\big((1+x)^2+R^2(1-x)^2-(1+R^2)y^2\big)dy}
    {\sqrt{\alpha f(z_0)}(1+R^2|\gamma(z)|^2)^3|1+z|^6}\\
    &=\frac{8R^2(1-y^2)dy}{\sqrt{\alpha f(z_0)}(1+R^2)^2(1+y^2)^3}.
    \end{split}
\end{equation*}
Now, with the recursion formula
\begin{equation}\label{5.16}
  \begin{split}
    \int\frac{dy} {|1+y^2|^n}=\frac1{2n-2}\frac{y}{|1+y^2|^{n-1}}
    +\frac{2n-3}{2n-2}\int\frac{dy}{|1+y^2|^{n-1}},\ n\ge 2,
    \end{split}
\end{equation}
we obtain
\begin{equation*}
  \begin{split}
    \int\frac{1-y^2} {|1+y^2|^3}dy=\int\big(\frac{2}{|1+y^2|^3}-\frac{1}{|1+y^2|^2}\big)dy
    =\frac12\frac{y}{|1+y^2|^2}+\frac12\int\frac{dy} {|1+y^2|^2}.
    \end{split}
\end{equation*}
It then follows that
\begin{equation*}
  \begin{split}
   \int_{\partial\R^2_+}&n_2\, e^{v_{\infty}\circ\gamma}\,ds_{\gamma^*g_{\R^2}}
    =\frac{8R^2}{\sqrt{\alpha f(z_0)}(1+R^2)^2}\int_{\R}\frac{1-y^2}{|1+y^2|^3}dy\\
    &=\frac{4R^2}{\sqrt{\alpha f(z_0)}(1+R^2)^2}\int_{\R}\frac{dy}{|1+y^2|^2}dy>0.
  \end{split}
\end{equation*}

Similarly we have
\begin{equation*}
  \begin{split}  
  \alpha &f(z_0)\int_{\R^2_+}n_2e^{2v_{\infty}\circ\gamma}d\mu_{\gamma^*g_{\R^2}}\\
  &=\int_{\R^2_+}\frac{32 R^3\big((1+x)^2+R^2(1-x)^2-(1+R^2)y^2\big)dz}
  {(1+R^2|\gamma(z)|^2)^4|1+z|^8}.
  \end{split}
\end{equation*}
But writing
\begin{equation*}
  \begin{split}
   (1+x)^2+&R^2(1-x)^2-(1+R^2)y^2\\
   &=2\big((1+x)^2+R^2(1-x)^2\big)-\big(|1+z|^2+R^2|1-z|^2\big),
  \end{split}
\end{equation*}
and recalling that we have
\begin{equation*}
    (1+R^2|\gamma(z)|^2)|1+z|^2=|1+z|^2+R^2|1-z|^2,
\end{equation*}
we find
\begin{equation*}
  \begin{split}
  &\frac{\alpha f(z_0)}{32R^3}\int_{\R^2_+}n_2e^{2v_{\infty}\circ\gamma}d\mu_{\gamma^*g_{\R^2}}\\
  &\quad=2\int_{\R^2_+}\frac{((1+x)^2+R^2(1-x)^2)dz}{(|1+z|^2+R^2|1-z|^2)^4}
 -\int_{\R^2_+}\frac{dz}{(|1+z|^2+R^2|1-z|^2)^3}.
  \end{split}
\end{equation*}

We can compute the latter integrals, as follows. Let $s>0$ such that 
\begin{equation*}
   (1+R^2)s^2=(1+x)^2+R^2(1-x)^2.
\end{equation*}
Then we have
\begin{equation*}
    |1+z|^2+R^2|1-z|^2=(1+R^2)(s^2+y^2)
\end{equation*}
and
\begin{equation*}
  \begin{split}
  &2\int_{\R^2_+}\frac{((1+x)^2+R^2(1-x)^2)dz}{(|1+z|^2+R^2|1-z|^2)^4}
    -\int_{\R^2_+}\frac{dz}{(|1+z|^2+R^2|1-z|^2)^3}\\
  &=\int_{\R^2_+}\frac{2(1+R^2)s^2dxdy}{(1+R^2)^4(s^2+y^2)^4}
    -\int_{\R^2_+}\frac{dxdy}{(1+R^2)^3(s^2+y^2)^3}\\
  &=\frac{1}{(1+R^2)^3}\Big(\int_{\R^2_+}\frac{2s^2dxdy}{(s^2+y^2)^4}
    -\int_{\R^2_+}\frac{dxdy}{(s^2+y^2)^3}\Big).
  \end{split}
\end{equation*}
But with \eqref{5.16} we have
\begin{equation*}
  \begin{split}
    &\int_{\R}\frac{2s^2dy} {(s^2+y^2)^4}=\int_{\R}\frac{2s^{-5}dy} {(1+y^2)^4}
    =\frac{5s^{-5}}{3}\int_{\R}\frac{dy}{(1+y^2)^3}
    =\frac{5}{3}\int_{\R}\frac{dy} {(s^2+y^2)^3},
  \end{split}
\end{equation*}
and
\begin{equation*}
  \begin{split}
    \frac{\alpha f(z_0)}{32R^3}\int_{\R^2_+}n_2e^{2v_{\infty}\circ\gamma}d\mu_{\gamma^*g_{\R^2}}
    &=\frac{2}{3}\int_{\R^2_+}\frac{dxdy}{(1+R^2)^3(s^2+y^2)^3}>0.
  \end{split}
\end{equation*}
The proof is complete.
\end{proof}

\subsection{Expressing $\Xi$ in terms of $\nabla J$}
Using the Kazdan-Warner type identity derived in Lemma \ref{lemma5.5} we now 
show that $\Xi$ is related to the gradient of the functions $J$ at $z_0$.
To set up the proof, recall that we have $\Xi=(\Xi_1,\Xi_2)$ with
\begin{equation*}
 \begin{split}
  \Xi_i&=\int_{S^2_R}X_i(\alpha f_{\Phi_R}-K_{\Phi_R})d\mu_{\bar{h}}
     +\int_{\partial S^2_R}X_i(\beta j_{\Phi_R}-k_{\Phi_R})ds_{\bar{h}}.
 \end{split}
\end{equation*}
Moreover, with error $o(1)\to 0$ as $t=t_l\to\infty$ for $1\le i\le 2$ there holds
\begin{equation*}
  \int_{S^2_R}X_i(\alpha f_{\Phi_R}-K_{\Phi_R})d\mu_{\bar{h}}
  =\frac1{\alpha f(z_0)}\int_{S^2_R}X_i(\alpha f_{\Phi_R}-K_{\Phi_R})d\mu_{g_{S^2}}
  +o(1)F(t)^{1/2}
\end{equation*}
as well as
\begin{equation*}
  \int_{\partial S^2_R}X_i(\beta j_{\Phi_R}-k_{\Phi_R})ds_{\bar{h}}
  =\frac1{\sqrt{\alpha f(z_0)}}\int_{\partial S^2_R}X_i(\beta j_{\Phi_R}-k_{\Phi_R})ds_{g_{S^2}}
  +o(1)F(t)^{1/2}.
\end{equation*}
Using the spherical metric as background, we now seek to find more convenient expressions for
these integrals.

Recall from \eqref{5.4} that for $1\le i\le 2$ there holds the equation
\begin{equation}\label{5.17}
  -\Delta_{S^2}X_i=2X_i.
\end{equation}
Integrating by parts, for $1\le i\le 2$ we then obtain
\begin{equation}\label{5.18}
 \begin{split}
  2&\int_{S^2_R}X_i(\alpha f_{\Phi_R}-K_{\Phi_R})d\mu_{g_{S^2}}
  =-\int_{S^2_R}\Delta_{S^2}X_i(\alpha f_{\Phi_R}-K_{\Phi_R})d\mu_{g_{S^2}}\\
  &=\int_{S^2_R}\nabla X_i\cdot(\alpha\nabla f_{\Phi_R}-\nabla K_{\Phi_R})
  d\mu_{g_{S^2}}
  -\int_{\partial S^2_R}\frac{\partial X_i}{\partial\nu_{S^2_R}}
  (\alpha f_{\Phi_R}-K_{\Phi_R})ds_{g_{S^2}},
 \end{split}
\end{equation}
where the gradient $\nabla X_i$ has the expression
\begin{equation*}
  \nabla X_i=e_i-X_iX \hbox{ on } S^2_R.
\end{equation*}

We would like to use Lemma \ref{lemma5.5} to deal with the term involving $\nabla K_{\Phi_R}$
(and a similar one involving $\nabla k_{\Phi_R}$ appearing later). Unfortunately, 
we cannot assert $\nabla X_i\in T_{id}M_R$; this would only be true if $R=1$, which we exclude.
However, we can compensate the infinitesimal rotation induced by $\nabla X_i$
to obtain some $\xi_i\in T_{id}M_R$, as follows. 
First, consider the case $i=1$. The vector field $X\wedge e_2$ induces
rotations around the $X_2$-axis. Recalling that $X_3\equiv\frac{1-R^2}{1+R^2}=\sigma$
on $\partial S^2_R$, we see that 
\begin{equation*}
  e_3\cdot(\nabla X_1+\sigma X\wedge e_2)=(\sigma-X_3)X_1=0\hbox{ on }\partial S^2_R.
\end{equation*}
Thus, the conformal deformation generated by $\xi_1:=\nabla X_1+\sigma X\wedge e_2$ preserves
$\partial S^2_R$, and $\xi_1\in T_{id}M_R$. Similarly, we have
$\xi_2:=\nabla X_2-\sigma X\wedge e_1\in T_{id}M_R$. Hence, splitting
\begin{equation*}
 \begin{split}
   \int_{S^2_R}\nabla X_1\cdot(\alpha\nabla f_{\Phi_R}&-\nabla K_{\Phi_R})
   d\mu_{g_{S^2}}
   =\int_{S^2_R}\xi_1\cdot(\alpha\nabla f_{\Phi_R}-\nabla K_{\Phi_R})d\mu_{g_{S^2}}\\
    &-\sigma\int_{S^2_R}(X\wedge e_2)\cdot(\alpha\nabla f_{\Phi_R}-\nabla K_{\Phi_R})d\mu_{g_{S^2}},
 \end{split}
\end{equation*}
and similarly for $i=2$, we can use Lemma \ref{lemma5.5} to eliminate the curvature.
The error terms thus arising, as well as the boundary term in \eqref{5.18} and similar
terms involving $\beta j_{\Phi_R}-k_{\Phi_R}$,
can be dealt with by means of the following lemma. As before we represent
$Z=\psi_R(z)=\frac{2Rz}{1+R^2}=rz\in\partial S^2_R$ with $z\in\partial B$, and we extend
$j$ as well as $k_g$ harmonically onto $B$.

\begin{lemma}\label{lemma5.7}
We have the identities
\begin{equation}\label{5.19}
  \int_{\partial S^2_R}\frac{\partial X_i}{\partial\nu_{S^2_R}}(\alpha f_{\Phi_R}-K_{\Phi_R})ds_{g_{S^2}}
    =r\sigma\int_{\partial B}z_i(\beta j_{\Phi}-k_{\Phi})ds_0
\end{equation}
and
\begin{equation}\label{5.20}
  \int_{\partial S^2_R}X_i(\beta j_{\Phi_R}-k_{\Phi_R})ds_{g_{S^2}}
  =r^2\int_{\partial B}z_i(\beta j_{\Phi}-k_{\Phi})ds_0,
\end{equation}
as well as the equations
\begin{equation}\label{5.21}
  \int_{S^2_R}(X\wedge e_2)\cdot(\alpha\nabla f_{\Phi_R}-\nabla K_{\Phi_R})d\mu_{g_{S^2}}
  =-r\int_{\partial B}z_1(\beta j_{\Phi}-k_{\Phi})ds_0
\end{equation}
and
\begin{equation}\label{5.22}
  \int_{\partial S^2_R}(X\wedge e_2)\cdot(\beta\nabla j_{\Phi_R}-\nabla k_{\Phi_R})ds_{g_{S^2}}\\
 =-(1+\sigma)\int_{\partial B}z_1(\beta j_{\Phi}-k_{\Phi})ds_0.
\end{equation}
Moreover, there holds
\begin{equation}\label{5.23}
 \begin{split}
  \int_{\partial S^2_R}&\frac{\partial X_i}{\partial\tau}\frac{\partial
   (\beta j_{\Phi_R}-k_{\Phi_R})}{\partial\tau}ds_{g_{S^2}}
   =\int_{\partial B}z_i(\beta j_{\Phi}-k_{\Phi})ds_0\\
  &=\frac1\sigma\int_{\partial S^2_R}\frac{\partial X_i}{\partial\nu_{S^2_R}}
  \frac{\partial (\beta j_{\Phi_R}-k_{\Phi_R})}{\partial\nu_{S^2_R}}ds_{g_{S^2}}.
 \end{split}
\end{equation}
\end{lemma}

\begin{proof}
With $\nu_{S^2_R}=(\sigma z,-r)$ there holds
\begin{equation*}
  \frac{\partial X_i}{\partial\nu_{S^2_R}}=\nu_{S^2_R}\cdot\nabla X_i=\sigma z_i
  \hbox{ along }\partial S^2_R.
\end{equation*}
Also recalling \eqref{1.18}, with $ds_{g_{S^2}}=r ds_{\pi_R^*g_{\R^2}}$
on $\partial S^2_R$ we then obtain
\begin{equation*}
  \int_{\partial S^2_R}\frac{\partial X_i}{\partial\nu_{S^2_R}}(\alpha f_{\Phi_R}-K_{\Phi_R})ds_{g_{S^2}}
  =\sigma r\int_{\partial B}z_i(\beta j_{\Phi}-k_{\Phi})ds_0;
\end{equation*}
moreover, we have
\begin{equation*}
  \int_{\partial S^2_R}X_i(\beta j_{\Phi_R}-k_{\Phi_R})ds_{g_{S^2}}
  =r^2\int_{\partial B}z_i(\beta j_{\Phi}-k_{\Phi})ds_0,
\end{equation*}
as claimed in \eqref{5.19} and \eqref{5.20}.

The latter integral may also be interpreted differently. Using the equation
\begin{equation*}
  -\Delta_{\partial B}z_i=-\frac{\partial^2z_i}{\partial\phi^2}=z_i\hbox{ on }\partial B,
\end{equation*}
for $1\le i\le 2$ we obtain
\begin{equation*}
  \int_{\partial B}z_i(\beta j_{\Phi}-k_{\Phi})ds_0
  =-\int_{\partial B}\Delta_{\partial B}z_i(\beta j_{\Phi}-k_{\Phi})ds_0
  =\int_{\partial B}\frac{\partial z_i}{\partial\phi}
  \frac{\partial(\beta j_{\Phi}-k_{\Phi})}{\partial\phi}ds_0.
\end{equation*}
With $\frac{\partial z}{\partial\phi}=(-z_2,z_1)$ for $z=e^{i\phi}$ on $\partial B$,
we can express
the unit tangent $\tau$ along $\partial S^2_R$ in scaled stereographic coordinates as
$\tau=(-z_2,z_1,0)=(\frac{\partial z}{\partial\phi},0)$, and 
$\frac{\partial X_i}{\partial\tau}=\tau\cdot\nabla X_i=\frac{\partial z_i}{\partial\phi}$.
Also observing that $\frac{\partial\pi_R}{\partial\tau}=\frac1r\frac{\partial z}{\partial\phi}$,
so that
\begin{equation}\label{5.24}
 \frac{\partial (\beta j_{\Phi_R}-k_{\Phi_R})}{\partial\tau}
 =\frac1r\frac{\partial(\beta j_{\Phi}-k_{\Phi})}{\partial\phi}\circ\pi_R
 \hbox{ on }\partial S^2_R,
\end{equation}
we can write
\begin{equation*}
 \int_{\partial B}\frac{\partial z_i}{\partial\phi}
 \frac{\partial(\beta j_{\Phi}-k_{\Phi})}{\partial\phi}ds_0
 =\int_{\partial S^2_R}\frac{\partial X_i}{\partial\tau}
 \frac{\partial (\beta j_{\Phi_R}-k_{\Phi_R})}{\partial\tau}ds_{g_{S^2}},
\end{equation*}
and the first part of \eqref{5.23} follows. 

With $\nu_B(z)=z$, and thus with $\frac{\partial z_i}{\partial\nu_B}=\nu_B\cdot e_i=z_i$
along $\partial B$, upon integrating by parts we also find
\begin{equation}\label{5.25}
 \begin{split}
   \int_{\partial B}&z_i(\beta j_{\Phi}-k_{\Phi})ds_0
   =\int_{\partial B}\frac{\partial z_i}{\partial\nu_B}(\beta j_{\Phi}-k_{\Phi})ds_0\\
   &=\int_{\partial B}z_i\frac{\partial(\beta j_{\Phi}-k_{\Phi})}{\partial\nu_B}ds_0
   =\int_{\partial B}\frac{\partial z_i}{\partial\nu_B}
     \frac{\partial(\beta j_{\Phi}-k_{\Phi})}{\partial\nu_B}ds_0.
 \end{split}
\end{equation}

On the other hand, with $\pi_R(X)=\frac{Z/R}{1+X_3}$ we have
$\frac{\partial\pi_R}{{\partial\nu_{S^2_R}}}=\nu_{S^2_R}\cdot d\pi_R=\frac1r\nu_B$
on $\partial S^2_R$, and there holds
\begin{equation}\label{5.26}
 \begin{split}
   r\frac{\partial(\beta j_{\Phi_R}-k_{\Phi_R})}{\partial\nu_{S^2_R}}
   =r\nu_{S^2_R}\cdot\nabla\big((\beta j_{\Phi}-k_{\Phi})\circ\pi_R\big)
   =\frac{\partial(\beta j_{\Phi}-k_{\Phi})}{\partial\nu_B}\circ\pi_R.
 \end{split}
\end{equation}
Similarly, with $\frac{\partial X_i}{\partial\nu_{S^2_R}}=\sigma z_i$ we have 
$\frac1\sigma \frac{\partial X_i}{\partial\nu_{S^2_R}}=z_i=\frac{\partial z_i}{\partial\nu_B}$,
and from \eqref{5.25} there results
\begin{equation*}
 \begin{split}
   \int_{\partial B}&z_i(\beta j_{\Phi}-k_{\Phi})ds_0
   =\frac1\sigma\int_{\partial S^2_R}\frac{\partial X_i}{\partial\nu_{S^2_R}}
   \frac{\partial (\beta j_{\Phi_R}-k_{\Phi_R})}{\partial\nu_{S^2_R}}ds_{g_{S^2}},
 \end{split}
\end{equation*}
which completes the proof of \eqref{5.23}.

With the identities $div(X\wedge e_2)=0$, $X\wedge\nu_{S^2_R}=\tau$, and observing that there holds
\begin{equation*}
  (X\wedge e_2)\cdot\nu_{S^2_R}=(\nu_{S^2_R}\wedge X)\cdot e_2=-\tau\cdot e_2=-z_1\hbox{ on }
  \partial S^2_R,
\end{equation*}
with \eqref{1.18} we see that 
\begin{equation*}
 \begin{split}
  \int_{S^2_R}&(X\wedge e_2)\cdot(\alpha\nabla f_{\Phi_R}-\nabla K_{\Phi_R})d\mu_{g_{S^2}}
  =\int_{\partial S^2_R}\nu_{S^2_R}\cdot(X\wedge e_2)(\alpha f_{\Phi_R}-K_{\Phi_R})ds_{g_{S^2}}\\
  &=-\int_{\partial S^2_R}z_1(\beta j_{\Phi_R}-k_{\Phi_R})ds_{g_{S^2}}
  =-r\int_{\partial B}z_1(\beta j_{\Phi}-k_{\Phi})ds_0,
 \end{split}
\end{equation*}
proving \eqref{5.21}.

Similarly, with \eqref{5.24}, \eqref{5.25}, and observing that
\begin{equation}\label{5.27}
  (X\wedge e_2)\cdot\tau=(\tau\wedge X)\cdot e_2=\nu_{S^2_R}\cdot e_2=\sigma z_2\hbox{ on }
  \partial S^2_R,
\end{equation}
as well as recalling \eqref{5.26}, we obtain
\begin{equation*}
 \begin{split}
  \int_{\partial S^2_R}&(X\wedge e_2)\cdot(\beta\nabla j_{\Phi_R}-\nabla k_{\Phi_R})ds_{g_{S^2}}\\
  & =\int_{\partial S^2_R}(X\wedge e_2)\cdot\Big(\nu_{S^2_R}
     \frac{\partial(\beta j_{\Phi_R}-k_{\Phi_R})}{\partial\nu_{S^2_R}}+\tau
     \frac{\partial(\beta j_{\Phi_R}-k_{\Phi_R})}{\partial\tau}\Big)ds_{g_{S^2}}\\
   &=-\int_{\partial S^2_R}z_1\frac{\partial(\beta j_{\Phi_R}-k_{\Phi_R})}{\partial\nu_{S^2_R}}ds_{g_{S^2}}
    +\int_{\partial B}\sigma z_2\frac{\partial(\beta j_{\Phi}-k_{\Phi})}{d\phi}ds_0\\
   &=-\int_{\partial B}z_1\frac{\partial(\beta j_{\Phi}-k_{\Phi})}{\partial\nu_B}ds_0
     -\int_{\partial B}\sigma z_1(\beta j_{\Phi}-k_{\Phi})ds_0\\
   & =-(1+\sigma)\int_{\partial B}z_1(\beta j_{\Phi}-k_{\Phi})ds_0,
 \end{split}
\end{equation*}
showing \eqref{5.22}.
\end{proof}

\begin{lemma}\label{lemma5.8}
With error $o(1)\to 0$ as $t=t_l\to\infty$ there holds
\begin{equation*}
  2\Xi_i=\alpha\int_{S^2_R}\xi_i\cdot df_{\Phi_R}d\mu_{\bar{h}}
  +2\beta\int_{\partial S^2_R}\xi_i\cdot dj_{\Phi_R}ds_{\bar{h}}+o(1)F(t)^{1/2},\ i=1,2.
\end{equation*}
\end{lemma}

\begin{proof}
With the Kazdan-Warner type identity established in Lemma \ref{lemma5.5},
letting $\gamma^2=\alpha f(z_0)$ we find 
\begin{equation*}
 \begin{split}
   I_1&:=\alpha\int_{\partial S^2_R}\xi_1\cdot d f_{\Phi_R}d\mu_{\bar{h}}
       +2\beta\int_{\partial S^2_R}\xi_1\cdot d j_{\Phi_R}ds_{\bar{h}}\\
   &=\int_{S^2_R}\xi_1\cdot(\alpha df_{\Phi_R}-dK_{\Phi_R})d\mu_{\bar{h}}
       +2\int_{\partial S^2_R}\xi_1\cdot(\beta dj_{\Phi_R}-dk_{\Phi_R})ds_{\bar{h}}\\
   &=\gamma^{-2}\int_{S^2_R}\xi_1\cdot(\alpha\nabla f_{\Phi_R}-\nabla K_{\Phi_R})d\mu_{g_{S^2}}\\
     &\quad+2\gamma^{-1}\int_{\partial S^2_R}\xi_1
       \cdot(\beta\nabla j_{\Phi_R}-\nabla k_{\Phi_R})ds_{g_{S^2}}+o(1)F(t)^{1/2}.
 \end{split}
\end{equation*}
But using that
\begin{equation}\label{5.28}
  \xi_1\cdot\nu_{S^2_R}
  =\frac{\partial X_1}{\partial\nu_{S^2_R}}+\sigma(X\wedge e_2)\cdot\nu_{S^2_R}=0
  \hbox{ on }\partial S^2_R,\ div(X\wedge e_2)=0\hbox{ in } S^2_R
\end{equation}
so that also
\begin{equation*}
  div\,\xi_1=div\,\nabla X_1=\Delta X_1=-2X_1,
\end{equation*}
we have
\begin{equation*}
 \begin{split}
  &\int_{S^2_R}\xi_1\cdot(\alpha\nabla f_{\Phi_R}-\nabla K_{\Phi_R})d\mu_{g_{S^2}}
  =2\int_{S^2_R}X_1(\alpha f_{\Phi_R}-K_{\Phi_R})d\mu_{g_{S^2}}.
 \end{split}
\end{equation*}
Similarly, successively using \eqref{5.23}, \eqref{5.22}, and \eqref{5.20} of
Lemma \ref{lemma5.7} we obtain
\begin{equation*}
 \begin{split}
   \int_{\partial S^2_R}&\xi_1\cdot(\beta\nabla j_{\Phi_R}-\nabla k_{\Phi_R})ds_{g_{S^2}}
   =\int_{\partial S^2_R}\nabla X_1\cdot(\beta\nabla j_{\Phi_R}-\nabla k_{\Phi_R})ds_{g_{S^2}}\\
   &+\sigma\int_{\partial S^2_R}(X\wedge e_2)\cdot(\beta\nabla j_{\Phi_R}-\nabla k_{\Phi_R})ds_{g_{S^2}}\\
   &=\big((1+\sigma)-\sigma(1+\sigma)\big)\int_{\partial B}z_1(\beta j_{\Phi}-k_{\Phi})ds_0\\
   &=(1-\sigma^2)\int_{\partial B}z_1(\beta j_{\Phi}-k_{\Phi})ds_0
    =r^2\int_{\partial B}z_1(\beta j_{\Phi}-k_{\Phi})ds_0\\
   & = \int_{\partial S^2_R}X_1(\beta j_{\Phi_R}-k_{\Phi_R})ds_{g_{S^2}}.
 \end{split}
\end{equation*}
It follows that
\begin{equation*}
 \begin{split}
  I_1&=2\gamma^{-2}\int_{S^2_R}X_1(\alpha f_{\Phi_R}-K_{\Phi_R})d\mu_{g_{S^2}}\\
     &\qquad+2\gamma^{-1}\int_{\partial S^2_R}X_1(\beta j_{\Phi_R}-k_{\Phi_R})ds_{g_{S^2}}+o(1)F(t)^{1/2}\\
     &=2\int_{S^2_R}X_1(\alpha f_{\Phi_R}-K_{\Phi_R})d\mu_{\bar{h}}
       +2\int_{\partial S^2_R}X_1(\beta j_{\Phi_R}-k_{\Phi_R})ds_{\bar{h}}+o(1)F(t)^{1/2}\\
     &=2\Xi_1+o(1)F(t)^{1/2}.
 \end{split}
\end{equation*}
With a similar computation for $\Xi_2$, the claim follows.
\end{proof}

For the following key result, as in Lemma \ref{lemma5.6} for some $t_l\to\infty$ as in
Corollary \ref{cor4.4} at a given time $t_0=t_l$ we again rotate
coordinates so that $\Phi(t)=\Phi_{e^{i\phi}a}$ with $0<a=a(t)<1$ and $\phi=\phi(t)$
satisfying $a(t_0)=a_0$ and $\phi(t_0)=0$, and we
set $0<\varepsilon=\frac{1-a}{1+a}<1$ with $\varepsilon(t_0)=:\varepsilon_0$. Since time will be
fixed, for convenience we again simply write $\varepsilon$ instead of
$\varepsilon_0$.

\begin{lemma}\label{lemma5.9}
With $R$ given by \eqref{4.5} and with error $o(1)\to 0$ as
$l\to\infty$ at time $t_0=t_l$ there holds
\begin{equation*}
   \Xi=\frac{16\pi\varepsilon R^3\sqrt{f(z_0)+j^2(z_0)}}{(1+R^2)^2f(z_0)}
   \nabla J(z_0)+o(1)F(t)^{1/2}+o(\varepsilon).
\end{equation*}
\end{lemma}

\begin{proof} i) Recall that with $\gamma^2=\alpha f(z_0)$ from Lemma \ref{lemma5.8} we have 
\begin{equation*}
 \begin{split}
   2\Xi_1&=\alpha\gamma^{-2}\int_{S^2_R}\xi_1\cdot\nabla f_{\Phi_R}d\mu_{g_{S^2}}\\
     &\quad+2\beta\gamma^{-1}\int_{\partial S^2_R}\xi_1
       \cdot\nabla j_{\Phi_R}ds_{g_{S^2}}+o(1)F(t)^{1/2}.
 \end{split}
\end{equation*}
Again using \eqref{5.28}, we can achieve a first reduction by integrating by parts
at the time $t_0$ and using symmetry to obtain, with $a=a_0$ as above,
\begin{equation*}
 \begin{split}
   \int_{S^2_R}\xi_1\cdot\nabla f_{\Phi_R}d\mu_{g_{S^2}}
   =2\int_{S^2_R}X_1f_{\Phi_R}d\mu_{g_{S^2}}=2\int_{S^2_R}X_1(f_{\Phi_R}-f(a))d\mu_{g_{S^2}}=:2I_1.
 \end{split}
\end{equation*}
The second term may be split
\begin{equation*}
 \begin{split}
   &\int_{\partial S^2_R}\xi_1\cdot\nabla j_{\Phi_R}ds_{g_{S^2}}
   =\int_{\partial S^2_R}\nabla X_1\cdot\nabla j_{\Phi_R}ds_{g_{S^2}}
   +\sigma\int_{\partial S^2_R}(X\wedge e_2)\cdot\nabla j_{\Phi_R}ds_{g_{S^2}}\\
   &\ =\int_{\partial S^2_R}\frac{\partial X_1}{\partial\nu_{S^2_R}}
   \frac{\partial j_{\Phi_R}}{\partial\nu_{S^2_R}}ds_{g_{S^2}}
   +\int_{\partial S^2_R}\frac{\partial X_1}{\partial\tau}
   \frac{\partial j_{\Phi_R}}{\partial\tau}ds_{g_{S^2}}\\
   &\qquad\qquad+\sigma\int_{\partial S^2_R}(X\wedge e_2)\cdot\big(\nu_{S^2_R}\frac{\partial j_{\Phi_R}}
   {\partial\nu_{S^2_R}}+\tau\frac{\partial j_{\Phi_R}}{\partial\tau}\big)ds_{g_{S^2}}\\
   &\ =\int_{\partial S^2_R}\Big(\big(\frac{\partial X_1}{\partial\nu_{S^2_R}}
   +\sigma(X\wedge e_2)\cdot\nu_{S^2_R}\big)
   \frac{\partial j_{\Phi_R}}{\partial\nu_{S^2_R}}+\big(\frac{\partial X_1}{\partial\tau}
   +\sigma(X\wedge e_2)\cdot\tau\big)
     \frac{\partial j_{\Phi_R}}{\partial\tau}\Big)ds_{g_{S^2}}.
  \end{split}
\end{equation*}
But with \eqref{5.28} the first of the latter terms vanishes; 
moreover, recalling \eqref{5.27} and $\frac{\partial Z}{\partial\tau}=(-z_2,z_1)$
we see that
\begin{equation*}
  \frac{\partial X_1}{\partial\tau}+\sigma(X\wedge e_2)\cdot\tau
  =(\sigma^2-1)z_2=-r^2z_2=-rX_2.
\end{equation*}
Hence, after integration by parts with $r\frac{\partial X_2}{\partial\tau}=rz_1=X_1$ there results
\begin{equation*}
 \begin{split}
   \int_{\partial S^2_R}&\xi_1\cdot\nabla j_{\Phi_R}ds_{g_{S^2}}
   =-r\int_{\partial S^2_R}X_2\frac{\partial j_{\Phi_R}}{\partial\tau}ds_{g_{S^2}}\\
   &=\int_{\partial S^2_R}X_1j_{\Phi_R}ds_{g_{S^2}}
   =\int_{\partial S^2_R}X_1(j_{\Phi_R}-j(a)))ds_{g_{S^2}}=:II_1.
 \end{split}
\end{equation*}

Arguing similarly for $i=2$, thus we find
\begin{equation*}
   \Xi_i=\alpha\gamma^{-2}I_i+\beta\gamma^{-1}II_i+o(1)F(t)^{1/2}, \ i=1,2.
\end{equation*}

ii) The integrals $I=(I_1,I_2)$, $II=(II_1,II_2)$ may be expanded similar
to \cite{Struwe-2005}, Lemma 4.5. We have
\begin{equation*}
  I=\int_{S^2_R}Z(f_{\Phi_R}-f(a))d\mu_{g_{S^2}}
  =\int_B\psi_R(f\circ\Phi_a-f(a))d\mu_{\Psi_R^*g_{S^2}}
\end{equation*}
with $d\mu_{\Psi_R^*g_{S^2}}(z)=\big(\frac{2R}{1+R^2|z|^2}\big)^2g_{\R^2}$.
In stereographic coordinates we can express $\Phi_a\circ\gamma=\gamma_{\varepsilon}$,
$\psi_R\circ\gamma(z)=\frac{2R\gamma(z)}{1+R^2|\gamma(z)|^2}$, 
$d\mu_{\gamma^*g_{\R^2}}(z)=\big(\frac{2}{|1+z|^2}\big)^2g_{\R^2}$, and we have
$a=\gamma_{\varepsilon}(1)=\gamma(\varepsilon)$.
Thus we can write the above in the form
\begin{equation*}
 \begin{split}
   I&=\int_{\R^2_+}\frac{32R^3\gamma(z)(f(\gamma(\varepsilon z))-f(\gamma(\varepsilon)))dz}
      {(1+R^2|\gamma(z)|^2)^3|1+z|^4}\\
   &=\int_{\{z\in\R^2_+;\;\varepsilon|z|<1/|\log\varepsilon|}
     \frac{32R^3\gamma(z)(f(\gamma(\varepsilon z))-f(\gamma(\varepsilon)))dz}
     {(1+R^2|\gamma(z)|^2)^3|1+z|^4}+O(\varepsilon^2\log^2(1/\varepsilon).
 \end{split}
\end{equation*}
Expanding $f\circ\gamma_{\varepsilon}$ around $z=1$, with $\gamma(\varepsilon)=a$,
$d\gamma_{\varepsilon}(1)=\varepsilon d\gamma(\varepsilon)$ we obtain 
\begin{equation*}
  f(\gamma(\varepsilon z))-f(\gamma(\varepsilon))
  =\varepsilon df(a)d\gamma(\varepsilon))(z-1)+O(\varepsilon^2(1+|z|^2)),
\end{equation*}
and with $d\gamma(\varepsilon))=-\frac{2}{1+\varepsilon^2}$ we find
\begin{equation*}
  df(a)d\gamma(\varepsilon))(z-1)=\frac{-2}{1+\varepsilon^2}
  \Big(\frac{\partial f(a)}{\partial x}(x-1)+\frac{\partial f(a)}{\partial y}y\Big).
\end{equation*}
In complex coordinates, writing $\gamma(z)=\frac{1-|z|^2-2iy}{|1+z|^2}$, $I=I_1+iI_2$,
by symmetry it follows that up to errors $R_i$ of size
$|R_i|\le C\varepsilon^2\log^2(1/\varepsilon)$, $1\le i\le 2$, there holds
\begin{equation*}
 \begin{split}
   I_1&=\frac{-2\varepsilon}{1+\varepsilon^2}\int_{\R^2_+}
       \Big(\frac{\partial f(a)}{\partial x}(x-1)+\frac{\partial f(a)}{\partial y}y\Big)
        \frac{32R^3(1-|z|^2)dz}{(1+R^2|\gamma(z)|^2)^3|1+z|^6}+R_1\\
      &=\frac{2\varepsilon}{1+\varepsilon^2}\frac{\partial f(a)}{\partial x}
        \int_{\R^2_+}\frac{32R^3(1-|z|^2)(1-x)dz}{(1+R^2|\gamma(z)|^2)^3|1+z|^6}+R_1
 \end{split}
\end{equation*}
whereas
\begin{equation*}
 \begin{split}
   I_2&=\frac{2\varepsilon}{1+\varepsilon^2}\int_{\R^2_+}
       \Big(\frac{\partial f(a)}{\partial x}(x-1)+\frac{\partial f(a)}{\partial y}y\Big)
       \frac{64R^3y\,dz}{(1+R^2|\gamma(z)|^2)^3|1+z|^6}+R_2\\ 
      &=\frac{2\varepsilon}{1+\varepsilon^2}\frac{\partial f(a)}{\partial y}
        \int_{\R^2_+}\frac{64R^3y^2\,dz}{(1+R^2|\gamma(z)|^2)^3|1+z|^6}+R_2. 
 \end{split}
\end{equation*}

Similarly, we expand the term $II$. This task can be considerably simplified if we observe that
by harmonicity of the function $j$ and conformal invariance also the composed function
$j\circ\Phi_a$ is harmonic. Since on each circle $\partial B_s(0)\subset B$ the pull-back 
measure $\Psi_R^*g_{S^2}$ is a constant multiple of Euclidean measure, 
by the mean value property of harmonic functions then we have
\begin{equation*}
 \begin{split}
  II&=\int_{\partial S^2_R}Z(j_{\Phi_R}-j(a)))ds_{g_{S^2}}
      =\int_{\partial B}\psi_R(j\circ\Phi_a-j(a))ds_{\Psi_R^*g_{S^2}}\\
     &=\int_{\partial B}(\psi_R+\psi_R(1))(j\circ\Phi_a-j(a))ds_{\Psi_R^*g_{S^2}}.
 \end{split}
\end{equation*}
In stereographic coordinates, with $x=0$ and $|\gamma(z)|=1$ on $\partial\R^2_+$ and
with
\begin{equation*}
  \gamma(z)+1=\frac{2}{1+z}=\frac{2(1+\bar{z})}{|1+z|^2}=\frac{2(1-iy)}{1+y^2}
\end{equation*}
for $z=iy\in\partial\R^2_+$, we thus obtain
\begin{equation*}
 \begin{split}
   II&=\int_{\partial\R^2_+}\frac{8R^2(\gamma(z)+1)
    (j(\gamma(\varepsilon z))-j(\gamma(\varepsilon)))ds_0}{(1+R^2|\gamma(z)|^2)^2|1+z|^2}\\
     &=\frac{16R^2}{(1+R^2)^2}\int_{\partial\R^2_+}\frac{(1-iy)
     (j(\gamma(\varepsilon z))-j(\gamma(\varepsilon)))ds_0}{(1+y^2)^2}.
 \end{split}
\end{equation*}
Again expanding 
\begin{equation*}
 \begin{split}
  j(\gamma(\varepsilon z))&-j(\gamma(\varepsilon))
  =\varepsilon dj(a)d\gamma(\varepsilon))(z-1)+O(\varepsilon^2(1+|z|^2))\\
  &=\frac{-2\varepsilon}{1+\varepsilon^2}
    \Big(\frac{\partial j(a)}{\partial y}y-\frac{\partial j(a)}{\partial x}\Big)
    +O(\varepsilon^2(1+y^2)),
 \end{split}
\end{equation*}
we find
\begin{equation*}
 \begin{split}
   II&=\frac{16R^2}{(1+R^2)^2}\int_{\{z\in\partial\R^2_+;\;\varepsilon|z|<1\}}
       \frac{(1-iy)(j(\gamma(\varepsilon z))-j(\gamma(\varepsilon)))ds_0}{(1+y^2)^2}
        +O(\varepsilon^2)\\
   &=\frac{32R^2\varepsilon}{(1+\varepsilon^2)(1+R^2)^2}\int_{\R}
       \Big(\frac{\partial j(a)}{\partial x}-\frac{\partial j(a)}{\partial y}y\Big)
       \frac{(1-iy)dy}{(1+y^2)^2}
       +O(\varepsilon^2\log(1/\varepsilon)),
 \end{split}
\end{equation*}
and we have
\begin{equation*}
 \begin{split}
   II_1&=\frac{32R^2\varepsilon}{(1+R^2)^2}\frac{\partial j(a)}{\partial x}
         \int_{\R}\frac{dy}{(1+y^2)^2}+O(\varepsilon^2\log(1/\varepsilon))
 \end{split}
\end{equation*}
as well as
\begin{equation*}
 \begin{split}
   II_2&=\frac{32R^2\varepsilon}{(1+R^2)^2}\frac{\partial j(a)}{\partial y}
         \int_{\R}\frac{y^2dy}{(1+y^2)^2}+O(\varepsilon^2\log(1/\varepsilon)).
 \end{split}
\end{equation*}

iii) Next we show that the expression for $I_1$ may be simplified and
that the coefficient of $\partial f(a)/\partial x$ is positive.
As in the proof of Lemma \ref{lemma5.6} we let
\begin{equation*}
 \begin{split}
   (1+R^2&|\gamma(z)|^2)|1+z|^2=|1+z|^2+R^2|1-z|^2\\
   &=(1+x)^2+R^2(1-x)^2+(1+R^2)y^2=(1+R^2)(s^2+y^2)
 \end{split}
\end{equation*}
with $s>0$ such that $(1+R^2)s^2=(1+x)^2+R^2(1-x)^2$. Hence we find
\begin{equation*}
 \begin{split}
   III:&=\int_{\R^2_+}\frac{(1-|z|^2)(1-x)dz}{(1+R^2|\gamma(z)|^2)^3|1+z|^6}
   =\int_{\R_+}\int_{\R}\frac{(1-|z|^2)(1-x)dxdy}{(|1+z|^2+R^2|1-z|^2)^3}\\
   &=\int_{\R_+}\int_{\R}\frac{(1-x^2+s^2-(s^2+y^2))(1-x)dxdy}{(1+R^2)^3(s^2+y^2)^3}.
 \end{split}
\end{equation*}
But using that by \eqref{5.16} we have
\begin{equation*}
  \int_{\R}\frac{dy}{(1+y^2)^3}=\frac34\int_{\R}\frac{dy}{(1+y^2)^2},
\end{equation*}
when substituting $y'=y/s$, and again writing $y$ instead of $y'$, we obtain
\begin{equation*}
 \begin{split}
   (1+R^2)^3&III=\int_{\R_+}\int_{\R}\frac{(1-x^2+s^2-(s^2+y^2))(1-x)dxdy}{(s^2+y^2)^3}\\
   &=\int_{\R_+}\int_{\R}\frac{(1-x^2+s^2)(1-x)dxdy}{s^5(1+y^2)^3}
   -\int_{\R_+}\int_{\R}\frac{(1-x)dxdy}{s^3(1+y^2)^2}\\
   &=\Big(\frac34\int_{\R_+}\frac{(1-x)^2(1+x)dx}{s^5}
   +\frac14\int_{\R_+}\frac{(x-1)dx}{s^3}\Big)\int_{\R}\frac{dy}{(1+y^2)^2}.
 \end{split}
\end{equation*}

Next let
\begin{equation*}
   (1+R^2)s^2=(1+x)^2+R^2(1-x)^2=(1+R^2)(1+x^2+2qx)
\end{equation*}
with $0<q=(1-R^2)/(1+R^2)<1$. Then we have $s^2=1+x^2+2qx$
with $s\,ds/dx=x+q$, and we can compute
\begin{equation*}
 \begin{split}
   IV&:=\int_{\R_+}\frac{(x-1)dx}{s^3}=\int_{\R_+}\frac{(x+q)dx}{s^3}-(1+q)\int_{\R_+}\frac{dx}{s^3}\\
   &=\int_1^{\infty}\frac{ds}{s^2}-(1+q)\int_{\R_+}\frac{dx}{s^3}=1-(1+q)\int_{\R_+}\frac{dx}{s^3}.
  \end{split}
\end{equation*}
But with
\begin{equation*}
 \begin{split}
   \frac{d}{dx}&\Big(\frac{x}{(1+x^2+2qx)^{1/2}}\Big)
     =\frac{1}{(1+x^2+2qx)^{1/2}}-\frac{x(x+q)}{(1+x^2+2qx)^{3/2}}\\
   &=\frac{qx+1}{(1+x^2+2qx)^{3/2}}=\frac{qx+1}{s^3}
  \end{split}
\end{equation*}
and
\begin{equation*}
     \int_{\R_+}\frac{(qx+1)dx}{s^3}
    =q\int_{\R_+} \frac{(x+q)dx}{s^3}+(1-q^2)\int_{\R_+} \frac{dx}{s^3}
    =q+(1-q^2)\int_{\R_+} \frac{dx}{s^3}
\end{equation*}
we obtain
\begin{equation*}
   1=\int_{\R_+}\frac{d}{dx}\Big(\frac{x}{(1+x^2+2qx)^{1/2}}\Big)dx
    =q+(1-q^2)\int_{\R_+}\frac{dx}{s^3}\,.
\end{equation*}
It follows that
\begin{equation}\label{5.29}
 \begin{split}
   \int_{\R_+} \frac{dx}{s^3}=\frac{1-q}{1-q^2}=\frac{1}{1+q},
 \end{split}
\end{equation}
and we conclude that $IV=0$. Thus $III>0$, and our claim follows.

iv) Finally, we relate the leading terms in the above expressions for $\Xi$ to
the gradient of the function $ J=j+\sqrt{f+j^2}$ defined in \eqref{1.24}.
Oberve that we have
\begin{equation*}
  \nabla J=\nabla j+\frac{j\nabla j}{\sqrt{f+j^2}}+\frac{\nabla f}{2\sqrt{f+j^2}},
\end{equation*}
so that with
\begin{equation*}
  R=\frac{\sqrt{f(z_0)+j(z_0)^2}-j(z_0)}{\sqrt{f(z_0)}}
\end{equation*}
given by \eqref{4.5} with $\tilde{k}=j(z_0)/\sqrt{f(z_0)}$, at the point $z_0$ we have
\begin{equation*}
  2R\nabla J\sqrt{f+j^2}=2R(j+\sqrt{f+j^2})\nabla j+R\nabla f
  =2\sqrt{f}\nabla j+R\nabla f,
\end{equation*}
all terms being evaluated at $z_0$.

Recalling that $\gamma^2=\alpha f(z_0)$, $\alpha=\beta^2$, we have
\begin{equation*}
  \Xi=\alpha\gamma^{-2}I+\beta\gamma^{-1}II+o(1)F(t)^{1/2}
  =I/f(z_0)+II/\sqrt{f(z_0)}+o(1)F(t)^{1/2}.
\end{equation*}
For the first component our computations in part iii) give 
\begin{equation*}
 \begin{split}
   I_1+\sqrt{f(z_0)}&II_1
    =\frac{48\varepsilon R^3}{(1+R^2)^3}\frac{\partial f(z_0)}{\partial x}
    \int_{\R^2_+}\frac{(1-x)^2(1+x)dz}{s^5(1+y^2)^2}\\
    &+\frac{32R^2\varepsilon\sqrt{f(z_0)}}{(1+R^2)^2}\frac{\partial j(z_0)}{\partial x}
         \int_{\R}\frac{dy}{(1+y^2)^2}+o(\varepsilon),
 \end{split}
\end{equation*}
where we have replaced $\partial f(a)/\partial x$ by $\partial f(z_0)/\partial x$
and likewise for $j$. Expanding
\begin{equation*}
 \begin{split}
   (1-x)^2&(1+x)=(1-2x+x^2)(1+x)=\big(s^2-2(1+q)x\big)(1+x)\\
   &=s^2(q+x)+(1-q)s^2-2(1+q)x(1+x)
 \end{split}
\end{equation*}
and writing
\begin{equation*}
 \begin{split}
   x(1+x)&=x^2+x=s^2-1+(1-2q)(x+q)-q+2q^2
 \end{split}
\end{equation*}
we obtain
\begin{equation*}
 \begin{split}
   (1-x)^2&(1+x)=s^2(q+x)-(1+3q)s^2\\
   &-2(1+q)(1-2q)(q+x)+2(1+q)(1+q-2q^2).
 \end{split}
\end{equation*}
Thus with \eqref{5.29} we can write 
\begin{equation*}
 \begin{split}
   \int_{\R_+}&\frac{(1-x)^2(1+x)dx}{s^5}=\int_{\R_+}\frac{(q+x)dx}{s^3}
                -(1+3q)\int_{\R_+}\frac{dx}{s^3}\\
   &-2(1+q)(1-2q)\int_{\R_+}\frac{(q+x)dx}{s^5}+2(1+q)(1+q-2q^2)\int_{\R_+}\frac{dx}{s^5}\\
   &=\int_1^{\infty}\frac{ds}{s^2}-\frac{1+3q}{1+q}-2(1+q)V=\frac{-2q}{1+q}-2(1+q)V,
 \end{split}
\end{equation*}
where
\begin{equation*}
 \begin{split}
    V:&=(1-2q)\int_0^{\infty}\frac{ds}{s^4}-(1+q-2q^2)\int_{\R_+}\frac{dx}{s^5}\\
   &=\frac{1-2q}{3}-(1+2q)(1-q)\int_{\R_+}\frac{dx}{s^5}.
 \end{split}
\end{equation*}
But with
\begin{equation*}
 \begin{split}
   \frac{d}{dx}\Big(\frac{x}{(1+x^2+2qx)^{3/2}}\Big)
   &=\frac{1}{s^3}-\frac{3x(x+q)}{s^5}=\frac{(1+x^2+2qx)-3x(x+q)}{s^5}\\
   =\frac{1-2x^2-qx}{s^5}&=\frac{3-2s^2+3q(x+q)-3q^2}{s^5},
  \end{split}
\end{equation*}
and again using \eqref{5.29}, we obtain
\begin{equation*}
 \begin{split}
   3(1-q^2)&\int_{\R_+}\frac{dx}{s^5}=2\int_{\R_+}\frac{dx}{s^3}-3q\int_{\R_+}\frac{(q+x)dx}{s^5}\\
   &=\frac{2}{1+q}-3q\int_1^{\infty}\frac{ds}{s^4}=\frac{2}{1+q}-q.
  \end{split}
\end{equation*}
Thus, with $\frac{q}{1+q}=\frac{1-R^2}2$ there results
\begin{equation*}
 \begin{split}
   \int_{\R_+}&\frac{(1-x)^2(1+x)dx}{s^5}
   =\frac{-2q}{1+q}-\frac23\Big((1+q)(1-2q)-(1+2q)\big(\frac{2}{1+q}-q\big)\Big)\\
   &=\frac23\Big(\frac{-q}{1+q}-(1+q)(1-2q)-q^2+(1+q)\big(\frac{2}{1+q}-q\big)\Big)\\
   &=\frac23\Big(1-\frac{q}{1+q}\Big)=\frac13(1+R^2).
 \end{split}
\end{equation*}
With \eqref{5.16}, moreover, we can compute
\begin{equation*}
 \begin{split}
   \int_{\R}\frac{dy}{(1+y^2)^2}=\frac12\int_{\R}\frac{dy}{1+y^2}=\frac{\pi}2;
  \end{split}
\end{equation*}
hence we find 
\begin{equation*}
 \begin{split}
   f(z_0)&\Xi_1=\frac{8\pi\varepsilon R^3}{(1+R^2)^2}\frac{\partial f(z_0)}{\partial x}
   +\frac{16\pi\varepsilon R^2\sqrt{f(z_0)}}{(1+R^2)^2}\frac{\partial j(z_0)}{\partial x}
   +o(1)F(t)^{1/2}+o(\varepsilon)\\
   &=\frac{8\pi\varepsilon R^2}{(1+R^2)^2}\Big(R\frac{\partial f(z_0)}{\partial x}
   +2\sqrt{f(z_0)}\frac{\partial j(z_0)}{\partial x}\Big)+o(1)F(t)^{1/2}+o(\varepsilon)\\
   &=\frac{16\pi\varepsilon R^3\sqrt{f(z_0)+j^2(z_0)}}{(1+R^2)^2}
   \frac{\partial J(z_0)}{\partial x}+o(1)F(t)^{1/2}+o(\varepsilon).
 \end{split}
\end{equation*}

Similarly, we argue for the second component of $\Xi$. Indeed, we have
\begin{equation*}
 \begin{split}
  I_2&+\sqrt{f(z_0)}II_2
   =\frac{128R^3\varepsilon}{(1+R^2)^3}\frac{\partial f(z_0)}{\partial y}
        \int_{\R^2_+}\frac{y^2\,dz}{(s^2+y^2)^3}\\
   &+\frac{32R^2\varepsilon\sqrt{f(z_0)}}{(1+R^2)^2}\frac{\partial j(z_0)}{\partial y}
         \int_{\R}\frac{y^2dy}{(1+y^2)^2}+O(\varepsilon^2\log^2(1/\varepsilon)).
 \end{split}
\end{equation*}
With \eqref{5.16} and \eqref{5.29}, and computing $1+q=\frac2{1+R^2}$,
when substituting $y'=y/s$ we find
\begin{equation*}
 \begin{split}
   &\int_{\R^2_+}\frac{y^2\,dz}{(s^2+y^2)^3}
   =\int_{\R_+}\frac{dx}{s^3}\int_{\R}\frac{y^2dy}{(1+y^2)^3}\\
   &\quad=\frac1{1+q}\big(\int_{\R}\frac{dy}{(1+y^2)^2}-\int_{\R}\frac{dy}{(1+y^2)^3}\big)
   =\frac{1+R^2}{8}\int_{\R}\frac{dy}{(1+y^2)^2}=\frac{1+R^2}{16}\pi.
  \end{split}
\end{equation*}
Since likewise there holds
\begin{equation*}
 \begin{split}
   \int_{\R}\frac{y^2dy}{(1+y^2)^2}
   =\int_{\R}\frac{dy}{1+y^2}-\int_{\R}\frac{dy}{(1+y^2)^2}=\frac{\pi}2,
  \end{split}
\end{equation*}
we obtain
\begin{equation*}
 \begin{split}
   f(z_0)&\Xi_2
   =\frac{8\pi\varepsilon R^3}{(1+R^2)^2}\frac{\partial f(z_0)}{\partial y}
   +\frac{16\pi\varepsilon R^2\sqrt{f(z_0)}}{(1+R^2)^2}\frac{\partial j(z_0)}{\partial y}
   +o(1)F(t)^{1/2}+o(\varepsilon)\\
   &=\frac{8\pi\varepsilon R^2}{(1+R^2)^2}\Big(R\frac{\partial f(z_0)}{\partial y}
   +2\sqrt{f(z_0)}\frac{\partial j(z_0)}{\partial y}\Big)+o(1)F(t)^{1/2}+o(\varepsilon)\\
   &=\frac{16\pi\varepsilon R^3\sqrt{f(z_0)+j^2(z_0)}}{(1+R^2)^2}
   \frac{\partial J(z_0)}{\partial y}+o(1)F(t)^{1/2}+o(\varepsilon),
 \end{split}
\end{equation*}
as claimed.
\end{proof}

The combination of Lemmas \ref{lemma5.6} and \ref{lemma5.9} gives the following result.

\begin{lemma}\label{lemma5.10}
If $\frac{\partial J(z_0)}{\partial\nu_0}\neq 0$, there holds $z_0=\lim_{t\to\infty}a(t)$,
and for sufficiently large
$l\in\N$ the equations in Lemmas \ref{lemma5.6} and \ref{lemma5.9} hold for all $t\ge t_l$.
\end{lemma}

\begin{proof}
In the setting of Lemma \ref{lemma5.6} with constants $C_i=C_i(z_0)>0$, $1\le i\le 2$,
independent of $\varepsilon>0$ for any $t_1\ge t_0=t_l$ such that
$\sup_{t_0\le t\le t_1}|e^{i\phi(t)}a(t)-z_0|\to 0$ as $l\to\infty$
with error $o(1)\to 0$ as $l\to\infty$ there holds
\begin{equation}\label{5.30}
  \begin{split}
    \big(\frac{da}{dt},\frac{d\phi}{dt}\big)\big|_{t=t_0}
    &+\varepsilon^2\big(C_1\frac{\partial J(z_0)}{\partial x},
      C_2\frac{\partial J(z_0)}{\partial y}\big)\\
    &=o(1)\varepsilon F(t)^{1/2}
      +o(\varepsilon^2)\le o(1)(\varepsilon^2+F).
  \end{split}
\end{equation}
Thus, if $\frac{\partial J(z_0)}{\partial x}\neq 0$, that is, if
$\frac{\partial J(z_0)}{\partial\nu_0}\neq 0$, 
upon integrating over $t_0=t_l\le t\le t_1$ with a constant $C_0>0$, also using
Corollary \ref{cor4.4}, we find
\begin{equation*}
    o(1)\ge C_0|a(t_0)-a(t_1)|\ge\int_{t_0}^{t_1}\varepsilon^2(t)dt+o(1)\int_{t_0}^{t_1}F(t)dt
\end{equation*}
and from \eqref{1.22} for any such $t_1\ge t_l$ it follows that 
$\int_{t_0}^{t_1}\varepsilon^2(t)dt\le o(1)\to 0$ as $l\to\infty$.
Thus, for sufficiently large $l\in\N$ from \eqref{5.30} we conclude that the condition
$\sup_{t_0\le t\le t_1}|e^{i\phi(t)}a(t)-z_0|\to 0$ holds for {\it any} $t_1\ge t_0=t_l$ and
$\int_{t_0}^{\infty}\varepsilon^2(t)dt<\infty$, which then also gives the claimed
convergence $a(t)\to z_0$ as $t\to\infty$. 
\end{proof}

\subsection{Dominance of $\Xi$}
In the setting of Lemma \ref{lemma5.10} the previously defined expansions thus hold
for all sufficiently large $t>0$. Similar to \cite{Struwe-2005} we then also can show that
for large $t>0$ the terms $\Xi_i$, $i=1,2$, in the expansion of $w_R$ dominate. 

\begin{proposition}\label{prop5.11}
Suppose that $\frac{\partial J(z_0)}{\partial\nu_0}\neq 0$. Then with error
$o(1)\to 0$ as $t\to\infty$ and a constant $C>0$ we have
\begin{equation*}
  F(t)=F_1+o(1)F\le C|\Xi|^2.
\end{equation*}
In fact, $F_0$ and $F_2$ both decay exponentially fast.
\end{proposition}

Similar to \cite{Struwe-2005}, we deduce this key proposition from the following result.

\begin{lemma}\label{lemma5.12}
With error $o(1)\to 0$ as $t\to\infty$ there holds $\frac{d\Xi}{dt}=o(1)F^{1/2}$.
\end{lemma}

\begin{proof}
With the equations
\begin{equation*}
  K_{\Phi_R}=K_{\bar{h}}=e^{-2\bar{v}}(-\Delta_{S^2}\bar{v}+1) \hbox{ on }S^2_R,\
  k_{\Phi_R}=e^{-\bar{v}}(\frac{\partial\bar{v}}{\partial\nu_{S^2_R}}+k_R) \hbox{ on }\partial S^2_R
\end{equation*}
analogous to \eqref{1.1}, \eqref{1.2}, and using \eqref{5.4}, upon integrating by parts
we find   
\begin{equation*}
 \begin{split}
  \int_{S^2_R}&X_iK_{\Phi_R}d\mu_{\bar{h}}
     +\int_{\partial S^2_R}X_ik_{\Phi_R}\,ds_{\bar{h}} \\          
  &=\int_{S^2_R}X_i(-\Delta_{S^2}\bar{v}+1)\,d\mu_{g_{S^2}}
    +\int_{\partial S^2_R}X_i (\frac{\partial\bar{v}}{\partial\nu_{S^2_R}}+k_R)\,ds_{g_{S^2}}\\
  &=2\int_{S^2_R}X_i\bar{v}\,d\mu_{g_{S^2}}
    +\int_{\partial S^2_R}\frac{\partial X_i}{\partial\nu_{S^2_R}}\bar{v}\,ds_{g_{S^2}}\\
  &=2\int_{S^2_R}X_i\bar{v}\,d\mu_{g_{S^2}}+k_R\int_{\partial S^2_R}X_i\bar{v}\,ds_{g_{S^2}}.
  \end{split}
\end{equation*}
Thus, we have 
\begin{equation*}
  \Xi_i=\int_{S^2_R}X_i(2\bar{v}-\alpha f_{\Phi_R}e^{2\bar{v}})\,d\mu_{g_{S^2}}
  +\int_{\partial S^2_R}X_i(k_R\bar{v}-\beta j_{\Phi_R}e^{\bar{v}})\,ds_{g_{S^2}}
\end{equation*}
and the proof may be completed exactly as in \cite{Struwe-2005}, Lemma 4.1.
\end{proof}

\begin{proof}[Proof of Proposition \ref{prop5.11}]
We now argue similar to \cite{Struwe-2005}, Lemma 4.2.

Let $\delta=\delta(t)\ge 0$ such that
$\hat{F}_2+\hat{G}_2+\big(\rho_t+2c_{\rho}(\hat{w}_R-\tilde{w}_R)\big)^2=\delta F$,
where we recall the defintion $c_{\rho}=\frac{\rho(\pi-\rho)}{\pi+\rho}$.
Note that since $span\{\varphi_k;\;k\ge 3\}$ includes all non-constant 
functions which are radially symmetric,
by a variant of Poincar\'e's inequality we can bound
$|\hat{w}_R-\tilde{w}_R|^2\le C(\hat{F}_2+\hat{G}_2)$.
Thus, for suitable $c_0>0$, whenever $\hat{F}_2+\hat{G}_2\le c_0F_0=c_0\rho_t^2$
we also have
$\big(\rho_t+2c_{\rho}(\hat{w}_R-\tilde{w}_R)\big)^2\ge\frac12F_0\ge\frac14\delta F$.

Suppose that for some $\delta_0>0$, $t_0\ge 0$ there holds
$\delta\ge\delta_0>0$ for $t\ge t_0$. Then from \eqref{5.5} we obtain that
$\frac{dF}{dt}\le -\delta_1F$ for some $\delta_1>0$ and all $t\ge t_0$ and we thus
have exponential decay $F(t)\le C_1e^{-\delta_1t}$ for some $C_1>0$. Using the
argument from \cite{Struwe-2005}, Lemma 4.2, we then derive a contradiction, as follows.

In view of \eqref{1.20}, with constants $C_2,C_3>0$ for any fixed $r_0>0$ we have
\begin{equation*}
  \begin{split}
  \big|\frac{d}{dt}&\big(\int_{B_{r_0}(z_0)}e^{2u}dz\big)\big|\le2\int_B|u_t|e^{2u}dz\\
  &\le2\big(\int_B|u_t|^2e^{2u}dz\int_Be^{2u}dz\big)^{1/2}\le C_2F^{1/2}
  \le C_3e^{-\delta_1t/2}.
  \end{split}
\end{equation*}
For any $t_1\ge t_0$ with a constant $C_0>0$ independent of $r_0$ and $t_1$ we then obtain 
\begin{equation*}
  \limsup_{t\to\infty}\int_{B_{r_0}(z_0)}e^{2u(t)}dz\le
  \int_{B_{r_0}(z_0)}e^{2u(t_1)}dz+C_0e^{-\delta_1t_1/2}.
\end{equation*}
Similarly, we find the estimate
\begin{equation*}
  \limsup_{t\to\infty}\int_{B_{r_0}(z_0)\cap\partial B}e^{u(t)}ds_0\le
  \int_{B_{r_0}(z_0)\cap\partial B}e^{u(t_1)}ds_0+C_0e^{-\delta_1t_1/2}.
\end{equation*}
Choosing $t_1\ge t_0$ such that $2C_0e^{-\delta_1t_1/2}<m_0$ 
and then also fixing $r_0>0$ suitably, we can achieve that 
\begin{equation*}
  \frac12\int_{B_{r_0}(z_0)}e^{2u(t_1)}dz+\int_{B_{r_0}(z_0)\cap\partial B}e^{u(t_1)}ds_0
  +2C_0e^{-\delta_1t_1/2}<m_0,
\end{equation*}
which contradicts Proposition \ref{prop4.1} and \eqref{1.20}.

Thus, there are times $t_i\to\infty$ such that with error $o(1)\to 0$ as $i\to\infty$
at $t=t_i$ there holds
\begin{equation*}
  \min\{c_0^{-1},1/2\}F_0\le\hat{F}_2+\hat{G}_2+\big(\rho_t+2c_{\rho}(\hat{w}_R-\tilde{w}_R)\big)^2=o(1)F.
\end{equation*}

Let $F=(1+\varepsilon)F_1$ for some $\varepsilon=\varepsilon(t)$, and then also
$F_0+F_2=\varepsilon F_1=\frac{\varepsilon}{1+\varepsilon}F$. By 
Lemma \ref{lemma5.3} this gives
\begin{equation*}
  C_R^{-2}\varepsilon F_1=C_R^{-2}(F_0+F_2)\le F_0+\hat{F}_2
  \le C_R^2(F_0+F_2)=C_R^2\varepsilon F_1.
\end{equation*}
Near any time $t_i$ by Lemma \ref{lemma5.12} and \eqref{5.5} we have
\begin{equation*}
 \begin{split}
   &\frac{d\varepsilon}{dt}F_1+o(1)F
     =\frac{d\varepsilon}{dt}F_1+(1+\varepsilon)\frac{dF_1}{dt}\\
   &=\frac{dF}{dt}
     \le-\frac{\lambda_3-1}{2\lambda_3}(\hat{F}_2+\hat{G}_2)
     -\big(\rho_t+2c_{\rho}(\hat{w}_R-\tilde{w}_R)\big)^2+o(1)F.
 \end{split}
\end{equation*}
But now either there holds $\hat{F}_2+\hat{G}_2\ge c_0F_0$, or
\begin{equation*}
  \big(\rho_t+2c_{\rho}(\hat{w}_R-\tilde{w}_R)\big)^2\ge\frac12F_0
  \ge\frac12 c^{-1}_0(\hat{F}_2+\hat{G}_2).
\end{equation*}
It follows that with constants $c_{1,2}>0$ for $t$ near $t_i$ we have
\begin{equation*}
 \begin{split}
   &\frac{d\varepsilon}{dt}F_1
   \le-c_1(F_0+\hat{F}_2)+o(1)F\le-2c_2\varepsilon F_1+o(1)F
   \le-c_2\varepsilon F_1
 \end{split}
\end{equation*}
when $i\in\N$ is suitably large. We conclude that 
\begin{equation*}
  \frac{d\varepsilon}{dt}\le-c_2\varepsilon,
\end{equation*}
and $\varepsilon(t)\to 0$ as $t\to\infty$, which gives the claim.
\end{proof}

\subsection{Conclusion}
From Proposition \ref{prop5.11} and Lemma \ref{lemma5.9} we deduce that under the assumptions
of Lemma \ref{lemma5.10} for sufficiently large $t>0$ with uniform constants $C>0$ there holds
$F^{1/2}\le C|\Xi|\le C\varepsilon$. Thus in this case we can also simplify the equation
\eqref{5.30} to read
\begin{equation}\label{5.31}
    \big(\frac{da}{dt},\frac{d\phi}{dt}\big)\big|_{t=t_0}
    +\varepsilon^2\big(C_1\frac{\partial J(z_0)}{\partial x},
      C_2\frac{\partial J(z_0)}{\partial y}\big)=o(\varepsilon^2),
\end{equation}
with constants $C_i=C_i(z_0)>0$, $1\le i\le 2$. Moreover, we can give a more precise
quantitative bound for the convergence $a(t)\to z_0$ as $t\to\infty$.

Indeed, computing $\frac{d\varepsilon}{dt}=-\frac2{(1+a)^2}\frac{da}{dt}$ we conclude
that if $\frac{\partial J(z_0)}{\partial x}<0$ with constants $0<c<C$ for sufficiently
large $t>0$ we have
\begin{equation*}
  c\varepsilon^2\le -\frac{d\varepsilon}{dt}\le C\varepsilon^2.
\end{equation*}
Thus, there holds $c\le d\varepsilon^{-1}/dt\le C$, and for sufficiently large $t_0>0$
we have 
\begin{equation*}
  \varepsilon(t)\le\frac{\varepsilon(t_0)}{1+c\varepsilon(t_0)(t-t_0)}\hbox{ for all }
  t\ge t_0.
\end{equation*}
It then follows that 
\begin{equation*}
 \begin{split}
  |\phi(t_1)&-\phi(t_0)|\le C\int_{t_0}^{t_1}\varepsilon^2(t)dt
  \le C\varepsilon^2(t_0)\int_{t_0}^{t_1}\frac{dt}{(1+c\varepsilon(t_0)(t-t_0))^2}\\
  &\le C\varepsilon(t_0)\int_1^{1+c\varepsilon(t_0)(t_1-t_0)}\frac{ds}{s^2}
  \le C\varepsilon^2(t_0)\frac{t_1-t_0}{1+c\varepsilon(t_0)(t_1-t_0)},
 \end{split}
\end{equation*}
and for sufficiently large $t_0=t_l$ we have $|a(t)-a(t_0)|\le C\varepsilon(t_0)$ for 
all $t>t_0$.

Moreover, we can now give the proof of our main result.

\begin{proof}[Proof of Theorem \ref{thm1.1}]
i) If $\partial J(z_0)/\partial\nu_0>0$ for all $z_0\in\partial B$,
assuming that the flow always concentrates, from \eqref{5.31} for any initial data $u_0$
we have $da/dt<0$ for sufficiently large $t>0$, which contradicts our assumption.

ii) On the other hand, if there holds
$\partial J(z_0)/\partial\nu_0<0$ for all $z_0\in\partial B$, similar to an argument of
Gehrig \cite{Gehrig-2020}, Section 8.3, for $a\in B$ we consider the
flow \eqref{1.13}-\eqref{1.15} with initial data $g_{a0}=e^{2u_{a0}}$ given by
\begin{equation*}
   g_{a0}=(\alpha_{a0}f(a))^{-1}\Phi_{-a}^*\Psi_R^*g_{S^2}, 
\end{equation*}
where $R=R(a)$, and with data $0<\rho_{a0}<\pi$ determined such that for $\alpha_{a0}$ 
and $\beta_{a0}$ given by \eqref{1.17} there holds $\alpha_{a0}=\beta_0^2$. 
Note that the normalised metric $h_{a0}=\Phi_{a}^*g_{a0}$ then satisfies 
\begin{equation*}
   \pi_R^*h_{a0}=\pi_R^*\Phi_{a}^*g_{a0}=(\alpha_{a0}f(a))^{-1}g_{S^2},
\end{equation*}
and the corresponding $F_{a}(0)\to 0$ as $|a|\to 1$.
From \eqref{5.30} and Lemma \ref{lemma5.9}, which in particular 
bounds the $\Xi$-component of $F=F_a$ in terms of $\varepsilon$, together with \eqref{5.5},
which gives control of the ``high frequency'' components of $F$,
we conclude that the corresponding evolving metrics $g_a(t)$ as $t\to\infty$ concentrate
at a boundary point $z_a\in\partial B$, where $|a-z_a|\to 0$ as $|a|\to 1$.

Thus, if we again assume that the flow always concentrates, the flow induces a retraction
of $B$ to $\partial B$, which is impossible.

iii) Similarly, in the case of the assumptions in part iii) of the Theorem,
when considering the flow with data $(u_{a0},\rho_{a0})$ for any $a\in B$ as in
part ii) above, we find that if the flow always concentrates it induces a
retraction of $B$ to a subset of $\partial B$ which is not connected,
and a topological contradiction results.
\end{proof}


\end{document}